\documentclass[a4paper,10pt,reqno]{amsart}

\usepackage{amssymb,amsthm,amsmath,mathrsfs}
\usepackage{mathtools}
\usepackage{bbm}
\usepackage{stackengine}
\usepackage{enumitem}
\usepackage{soul}

\usepackage[foot]{amsaddr}

\usepackage{caption,color,graphicx}
\usepackage[dvipsnames]{xcolor}
\usepackage{soul}

\usepackage{hyperref}
\usepackage{doi}
\hypersetup{colorlinks  = true,
            urlcolor    = blue,
            linkcolor   = red,
            citecolor   = blue}

\usepackage{verbatim}

\usepackage{environ}
\mathtoolsset{showonlyrefs,showmanualtags}

\usepackage{graphicx}
\newcommand\smallO{
  \mathchoice
    {{\scriptstyle\mathcal{O}}}
    {{\scriptstyle\mathcal{O}}}
    {{\scriptscriptstyle\mathcal{O}}}
    {\scalebox{.7}{$\scriptscriptstyle\mathcal{O}$}}
  }

\usepackage[numbers]{natbib}

\usepackage[centering, top=3cm, bottom=2.5cm]{geometry}

\usepackage{tikz}
\usetikzlibrary{arrows.meta}

\newcommand{\1}{\mathbbm{1}}

\renewcommand{\d}{\operatorname{d}\!}

\newcommand{\e}{\varepsilon}

\renewcommand{\implies}{\Rightarrow}
\newcommand{\vertiii}[1]{{\left\vert\kern-0.25ex\left\vert\kern-0.25ex\left\vert #1 \right\vert\kern-0.25ex\right\vert\kern-0.25ex\right\vert}}

\let\oldtocsection=\tocsection
\let\oldtocsubsection=\tocsubsection
\renewcommand{\tocsection}[2]{\hspace{0em}\oldtocsection{#1}{#2}}
\renewcommand{\tocsubsection}[2]{\hspace{1em}\oldtocsubsection{#1}{#2}}

\theoremstyle{plain}
\newtheorem*{theorem*}{Theorem}
\newtheorem{theorem}{Theorem}[section]
\newtheorem{lemma}[theorem]{Lemma}
\newtheorem{proposition}[theorem]{Proposition}

\newtheorem{maintheorem}{Theorem}

\theoremstyle{definition}
\newtheorem{definition}[theorem]{Definition}
\newtheorem{remark}[theorem]{Remark}
\newtheorem{example}[theorem]{Example}
\newtheorem{notation}[theorem]{Notation}

\newenvironment{stepproof}[1][]{%
  \begin{proof}[#1]%
}{%
  \end{proof}%
}

\newtheorem{manualstep}{\bf Step}
\newenvironment{step}[1][]{%
    \renewcommand\themanualstep{#1}%
    \begin{manualstep}\itshape
}{\end{manualstep}}

\newtheorem{manualhyp}{Hypothesis}
\newenvironment{hypothesis}[1][]{%
    \renewcommand\themanualhyp{#1}%
    \begin{manualhyp}\itshape
}{\end{manualhyp}}

\IfPackageLoadedTF{MnSymbol}{}{}

\author[Bernat Bassols Cornudella]{Bernat Bassols Cornudella$^{1}$}
\email{\href{mailto:bernat.bassols-cornudella20@imperial.ac.uk}{bernat.bassols-cornudella20@imperial.ac.uk}}

\author[Matheus M.~Castro]{Matheus M. Castro$^{2}$}
\email{\href{mailto:mmcastro@ime.unicamp.br}{mmcastro@ime.unicamp.br}}

\address{$^{1}$Department of Mathematics, Imperial College London, London SW7 2AZ, UK.}
\address{$^{2}$Departamento de Matemática – Universidade Estadual de Campinas (UNICAMP), 13083-859, Campinas, SP, Brazil.}

\title[Conditional Stable Laws and Rare-Event Limits for Absorbing Markov Chains]{Conditional Stable Laws and Rare-Event Limits for Absorbing Markov Chains} 

\date{\today}

\begin{document}

\begin{abstract}
We establish conditional limit theorems, pointwise in the initial state, for absorbing Markov chains on a compact metric space $M$. We assume $L^1(M,\rho)$-continuous transition densities, irreducibility and aperiodicity. For the observable $f_\beta(x)=d_M(x,x_0)^{-\beta}$, with suitable $x_0$ satisfying $\rho(B_r(x_0))\sim C_d(x_0)r^d$, we prove that the point-process of normalised observations converges to a Poisson random measure. This yields totally right-skewed $\alpha$-stable laws for $\alpha=d/\beta\in(0,2)$ and, at the boundary value $\alpha=2$, a Gaussian limit with the non-standard normalisation $\sqrt{n\log n}$. We also establish a conditional central limit theorem for $L^2$ observables, exponential deviation bounds for bounded observables and a conditional Poisson law for visits to shrinking targets. The results and remarks have been formalised in Lean~4.
\end{abstract}


\vspace*{-1cm}
\maketitle
{\thispagestyle{empty}\tableofcontents\newpage}

\section{Introduction}

The theory of absorbing Markov processes provides a mathematical framework to analyse the evolution of systems before they reach a terminal time~\cite{DarrochSeneta1965, ColletMartinezSanMartin2013}. This includes systems subject to extinction, such as birth--death population models~\cite{MeleardVillemonais2012, FritschVillemonaisZalduendo2026}; systems that may leak from a given bounded domain, such as molecular transport across cellular compartments~\cite{Pinsky1985, SchussSingerHolcman2007} or passive tracers in an open fluid flow~\cite{AltmannPortelaTel2013}; as well as systems prone to sudden collapse, as considered in system reliability theory~\cite{RausandHoyland2004}. Each of these examples can be modelled by a process whose trajectory terminates upon reaching an absorbing set.

If absorption occurs almost surely, every stationary probability measure is concentrated on the absorbing set and therefore carries no information about transient trajectories that remain active for a long time. To retain information about these pre-absorption dynamics, we consider trajectories conditioned on survival for sufficiently long time.

In this paper, we address whether observations along such long-lived conditioned trajectories obey classical statistical limit laws. More precisely, we investigate whether they satisfy central and stable limit theorems, exponential deviation bounds and Poisson limits. Let us begin by briefly recalling these properties.

For a Markov chain $X_n$ on $M$ preserving an ergodic probability measure $\pi$ and an observable $\phi:M\to\mathbb R$, consider the Birkhoff sums
$$
S_n\phi
:=
\sum_{i=0}^{n-1}\phi\circ X_i.
$$
If $\phi \in L^1(M,\pi)$ the law of large numbers describes the deterministic first-order behaviour of $S_n\phi$. To understand the fluctuations of $S_n \phi$ we generally rely on limit theorems. Namely, we seek constants $A_n$ and $B_n$ such that 
$$\frac{1}{B_n}\left( S_n\phi - A_n\right) \xrightarrow[d]{n\to\infty} Z$$
for some suitable distribution $Z$, where $\xrightarrow[d]{}$ denotes convergence in distribution.
Depending on the regularity of the observable $\phi$, one may establish different limit theorems:
\begin{itemize}
\item if $\phi\in L^2(M,\pi)$, sufficiently strong mixing conditions may yield a central limit theorem (CLT), i.e.~$Z$ a normal distribution, $A_n = n\pi(\phi)$ and $B_n = \sqrt{n};$
\item if $\phi \in L^p(M,\pi)$ for every $p<2$ but not $p=2$, the Gaussian limit may still hold with a non-standard normalisation, e.g.~$B_n = \sqrt{n\log n};$
\item
if $\phi\in L^p(M,\pi)$ for every $p<\alpha$ and $\phi\notin L^p(M,\pi)$ for every $p\geq\alpha$, and under suitable mixing and tail-balance assumptions, $Z$ may correspond to an $\alpha$-stable law, where $B_n = n^{1/\alpha}$ and the centring $A_n$ depends on the $\alpha$ considered.
\end{itemize}
We refer to~\cite{Feller1971} for general background on these limit theorems and~\cite{KulikSoulier2020} for a modern exposition.

While the preceding limit theorems concern a fixed observable $\phi$, another class of distributional limits arises when the observable varies with $n$. In particular, this is the case when analysing the number of visits to increasingly rare sets, as is often done in the study of extreme value theory~\cite{FreitasFreitasTodd2010}. Given a sequence of measurable sets $U_n\subset M$, consider the Birkhoff sums
$$
P_n:=\sum_{i=0}^{n-1}\mathbbm 1_{U_n}\circ X_i.
$$
The scaling $n\pi(U_n)\xrightarrow[]{n\to\infty}t\in(0,\infty)$
ensures that the expected number of visits to $U_n$ converges to a finite positive limit. Under suitable assumptions on mixing and temporal dependence, $P_n$ may converge in distribution to a Poisson or compound Poisson random variable~\cite{LeadbetterLindgrenRootzen1983, Resnick1987}.

In this paper, we establish several limit theorems and rare event statistics as presented above for absorbing Markov chains conditioned upon survival, where the law of the process also depends on $n$.

\subsection{Absorbing Markov processes and quasi-ergodic measures}
Most of the theory described above is formulated for stationary Markov chains. For an absorbing Markov chain $X_n$, let $\tau$ denote the absorption time and let $M$ be the non-absorbing state space. Conditioning the statistics of the process upon survival corresponds to analysing Birkhoff averages up to time $n$ conditionally on the event $\{\tau>n\}$. In other words, $S_n \phi$ is considered under the $n$-dependent probability measure $\mathbb P_x(\,\cdot\mid\tau>n).$
At the level of conditional expectations, the first-order behaviour of the Birkhoff average is described by the so-called \emph{quasi-ergodic} limit
\begin{equation}\label{eq:qed1}
\mathbb E_x\left[\left.\frac{1}{n}\sum_{i=0}^{n-1}\phi \circ X_i\,\right|\,\tau>n\right]\xrightarrow[]{n\to\infty}\nu(\phi):=\int_M \phi\,\d\nu.
\end{equation}
This limit is viewed as the conditional counterpart of the law of large numbers and has been studied extensively within the broader theory of quasi-stationarity (see~\cite{DarrochSeneta1965,BreyerRoberts1999,ChenJian2017,HeZhangZhu2019,Ocafrain2020,ColoniusRasmussen2021,CastroGoverseLambRasmussen2024,CastroLambOliconMendezRasmussen2024,BassolsCastroLamb2025} and references therein).

The quasi-ergodic measure $\nu$ turns out to be a stationary measure for $X_n$ under the initial family of probability measures
\[\mathbb Q_x (\cdot) := \lim_{n\to \infty} \mathbb P_x(\cdot\mid\tau>n)\]
defined on a suitable $\sigma$-algebra~\cite{ChampagnatVillemonais2016,Castroetall}. This new chain is referred to as the \emph{$Q$-process} and, in particular, is a standard (non-absorbing) Markov chain. Nevertheless, limit theorems for the $Q$-process under its stationary measure $\nu$ do not directly imply limit theorems for $X_n$ under the changing conditioned measures $\mathbb P_x(\cdot \mid \tau > n)$~\cite{ChampagnatVillemonais2016,Ocafrain2020,CastroLambOliconMendezRasmussen2024}. This presents a major challenge when establishing our results.

Relatively few limit theorems are available beyond the first-order quasi-ergodic behaviour in \eqref{eq:qed1}. Matthews~\cite{Matthews1970} and Al-Eideh~\cite{AlEideh1994} proved a conditional central limit theorem for finite-state absorbing Markov chains. More recently, Oçafrain~\cite{Ocafrain2023} established a conditional central limit theorem for a broad class of continuous-time absorbing Markov processes. The probability that conditioned Birkhoff averages deviate from their quasi-ergodic limit, in the spirit of a weak law of large numbers, has also been studied through deviation inequalities and large-deviation principles~\cite{ChenDeng2013, ChenJian2017, KimTagawaVelleret2026, GuillinNectouxWu2024}. The existing results leave open the behaviour of heavy-tailed observables and rare-event point-processes under survival conditioning.

\subsection{Existing techniques}
One of the principal methods for proving limit theorems for Birkhoff sums of Markov chains is the spectral technique introduced by Nagaev~\cite{Nagaev1957}. Given a transition operator $\mathcal P$ and an observable $\phi$, one considers the twisted operators $
\mathcal P_z h:=\mathcal P\left(e^{z\phi}h\right)$
and studies the perturbation of their leading eigenvalue near $z=0$. Imaginary perturbations encode the characteristic functions of the Birkhoff sums and lead to central limit theorems, whereas real perturbations provide control over exponential moments and yield concentration and large-deviation estimates.
Guivarc'h--Hardy~\cite{GuivarchHardy1988} applied this method to Markov chains for which their transition operators have a spectral gap on the space of Lipschitz functions, and Hennion--Hervé~\cite{HennionHerve2001} generalised it to a framework based on quasi-compact operators.
Further limiting results based on spectral methods for Markov processes, beyond the central limit and large deviation principle, were obtained in~\cite{KontoyiannisMeyn2003, Gouezel2010ASIP}.

The fluctuations of $S_n\phi$ for heavy-tailed observables $\phi$, e.g.~$\pi(|\phi| > t)\sim t^{-\alpha}$ with $\alpha \in (0,2),$ are governed by the exceptionally large values of $\phi$ attained along the chain. 
Point-process methods provide fine control of such large values, and yield stable laws for $S_n\phi$ after proper scaling and centring.
This technique was introduced by Durrett--Resnick~\cite{DurrettResnick1978} and further developed by Davis--Hsing~\cite{DavisHsing1995}. Tyran-Kamińska~\cite{TyranKaminska2010} formulated necessary and sufficient criteria, based on the point-process method, for obtaining convergence to a Lévy process under strong mixing assumptions. Moreover, Cattiaux--Manou-Abi~\cite{CattiauxManouAbi2014} obtained stable laws for heavy-tailed observables by establishing sufficient conditions in terms of spectral-gap and integrability properties of the transition operator.

Point-processes also arise naturally in the study of visits to rare sets. {Early Poisson limit laws appeared in Doeblin's work on continued fractions~\cite{Doeblin1940}. Iosifescu~\cite{Iosifescu1977} subsequently established a Poisson limit theorem for $\psi$-mixing sequences.} Poisson laws for returns to shrinking targets were later obtained for stationary Markov chains by Pitskel~\cite{Pitskel1991} and for Axiom A diffeomorphisms by Hirata~\cite{Hirata1993}. Hirata--Saussol--Vaienti~\cite{HirataSaussolVaienti1999} developed a general framework for return-time statistics. If such returns occur in clusters,
$P_n$ may instead converge to a compound Poisson law, as shown in~\cite{HaydnVaienti2009}. A complementary spectral approach to rare events, escape rates and quasi-stationarity was developed by Keller--Liverani~\cite{KellerLiverani2009}. More recently, this spectral perspective was extended to random dynamical systems by Atnip--Froyland--González-Tokman--Vaienti, who established quenched compound Poisson laws for visits to shrinking random targets through perturbations of transfer operator cocycles~\cite{AtnipFroylandGonzalezTokmanVaienti2025}.

\subsection{Brief statement of the main results}
We work within the framework of~\cite{CastroLambOliconMendezRasmussen2024}. Hypothesis~\ref{hyp:H}, stated precisely in Section~\ref{sec:main-results}, imposes compactness, irreducibility and aperiodicity conditions on the killed transition operator. These assumptions guarantee the existence and uniqueness of a quasi-stationary measure $\mu$ and a quasi-ergodic measure $\nu$, both absolutely continuous with respect to a reference measure $\rho$.

The main observable of interest for which we establish stable laws in this paper is
$$
f_\beta(x):=\begin{cases}
d_M(x,x_0)^{-\beta},&x\neq x_0,\\1,&x=x_0,
\end{cases}
\ 
\beta>0,
$$
where $x_0$ is a $(\rho, d)$-admissible point (see Definition~\ref{def:reg}). 

We set $f_\beta(x_0)=1$ for convenience. The $(\rho,d)$-admissibility condition implies that $\nu(f_\beta>u)\sim C(x_0)u^{-\alpha}$, where $\alpha:=d/\beta$ and $C(x_0):=C_d(x_0)\eta(x_0)m(x_0)>0$ (see Proposition~\ref{prop:tail-asymptotics}). Thus, the index $\alpha$ is determined by the local dimension of $\rho$ at $x_0$ and the strength of the singularity of $f_\beta$.

The main results are the following:
\begin{enumerate}[label = (R\arabic*)]
\item \label{it:brief-1} For $0<\alpha<2$, let $B_n:=\left(C(x_0)n\right)^{1/\alpha}$
and define the point-process
$$
N_n:=\sum_{i=0}^{n-1}\delta_{\frac{f_\beta\circ X_i}{B_n}}.
$$
Theorem~\ref{thm:ppp} shows that, under $\mathbb P_x(\,\cdot\mid\tau>n)$, $
N_n\xrightarrow[d]{n\to\infty}\operatorname{PRM}(\upsilon_\alpha)$ in $\mathcal M_p((0,\infty])$ (see Definition \ref{def:Mp}), a Poisson random measure of intensity $\upsilon_\alpha(\d y)=\alpha y^{-\alpha-1}\,\d y$.
\item \label{it:brief-2} Theorem~\ref{thm:main-stable-laws} establishes that, for suitable centring and normalising sequences $(A_n)$ and $(B_n)$, under $\mathbb P_x(\cdot \mid \tau>n)$
$$
\frac{1}{B_n}\left(\sum_{i=0}^{n-1}f_\beta\circ X_i-A_n\right)\xrightarrow[d]{n\to\infty}Z_\alpha.
$$
For $0<\alpha<2$, $Z_\alpha$ is a totally right-skewed $\alpha$-stable random variable, while $Z_2\sim\mathcal N(0,1)$. The precise choices of $A_n$ and $B_n$ in the different regimes of $\alpha$ are given in Theorem~\ref{thm:main-stable-laws}.
\item \label{it:brief-3} Theorem~\ref{thm:conditional-clt-l2} establishes, for every $g\in L^2(M,\mu)$,
$$
\frac{1}{\sqrt n}\sum_{i=0}^{n-1}\left(g-\nu(g)\right)\circ X_i\xrightarrow[d]{n\to\infty}\mathcal N(0,\sigma_g^2).
$$
For every bounded measurable observable $h$, Theorem~\ref{thm:conditional-exponential-concentration} gives there exists $C_x$ and $a_h$ such that
$$
\mathbb P_x \left(\left. \left| \frac{1}{n} \sum_{i=0}^{n-1} h\circ X_i -\nu(h)\right|>\varepsilon\right|\tau>n\right)\leq C_{x} e^{- a_h \e^2 n} $$
for every $\e>0.$
\item \label{it:brief-4} Theorem~\ref{thm:conditional-poisson} proves that, if $(U_n)$ is a sequence of measurable neighbourhoods shrinking to a point $x_1$ and $
n\nu(U_n)\xrightarrow[]{n\to\infty}t\in(0,\infty),$
then, under $\mathbb P_x(\,\cdot\mid\tau>n)$,
$$
\sum_{i=0}^{n-1}\mathbbm 1_{U_n}\circ X_i\xrightarrow[d]{n\to\infty}\operatorname{Poi}(t).
$$
\end{enumerate}
The point-process, stable, Gaussian boundary, central limit and concentration results \ref{it:brief-1}--\ref{it:brief-3} hold for every $x\in M\setminus Z$, where $Z$ denotes the set of points with an almost surely bounded escape time (see item \ref{it:hypH-ii} in Hypothesis~\ref{hyp:H}). The shrinking-target result \ref{it:brief-4} holds for every $x\in M\setminus(Z\cup\{x_1\})$. Thus, all the results are pointwise in the initial state (quenched) and do not require averaging against an initial distribution.

To the best of our knowledge, the Poisson point-process and stable limit laws in \ref{it:brief-1} and \ref{it:brief-2} are new for absorbing Markov chains. Our main technical contribution is to prove these results directly under the conditioned measures $\mathbb P_x(\cdot\mid\tau>n)$. Their proofs essentially combine: (i) ``conditional mixing of all orders'' estimates for well-separated observations, obtained from the spectral decomposition of the killed transition operator; together with (ii) estimates under the $Q$-process started from its invariant measure $\nu$, which control clusters of nearby rare observations. These ingredients allow us to establish the convergence to a Poisson point-process via the method of moments, which then yields the stable laws above.

The remaining results (\ref{it:brief-3} and \ref{it:brief-4}), namely the classical central limit theorem, large deviations and the Poisson limit law, are new for absorbing systems satisfying Hypothesis~\ref{hyp:H}. Their proofs adapt existing techniques in the theory of (classical) Markov processes to absorbing Markov processes.
 
The paper is organised as follows. Section~\ref{sec:main-results}
introduces the absorbing-chain setting, states the main conditional limit theorems and proves the tail asymptotics that determine the normalising sequence for the singular observable $f_\beta$. Section~\ref{sec:examples-hypothesis-H} verifies Hypothesis~\ref{hyp:H} for several classes of randomly perturbed systems and killed diffusions. Section~\ref{sec:spectral-framework} develops the spectral tools on $L^\infty$ and $L^2$ and proves a uniform estimate for shrinking targets. In Section~\ref{sec:q-process} we introduce the $Q$-process, later used to establish refined
quasi-ergodic limits in Appendix~\ref{ref:appendix}. Section~\ref{sec:mixing} derives conditional mixing estimates for well-separated observations and for separated clusters. Section~\ref{sec:ppp} proves the Poisson point-process
limit. Section~\ref{sec:stable-laws} deduces the stable laws for $0<\alpha<2$. Section~\ref{sec:gaussian-laws} treats the boundary case $\alpha=2$, the conditional central limit theorems and exponential concentration. Finally, Section~\ref{sec:conditional-poisson} proves the general conditional Poisson law for visits to shrinking targets by spectral perturbation. Theorems \ref{thm:ppp}--\ref{thm:refined-conditional-lln} and the remarks of this paper have been formalised Lean~4.35.0 as explained in Section \ref{sec:lean-formalisation}.

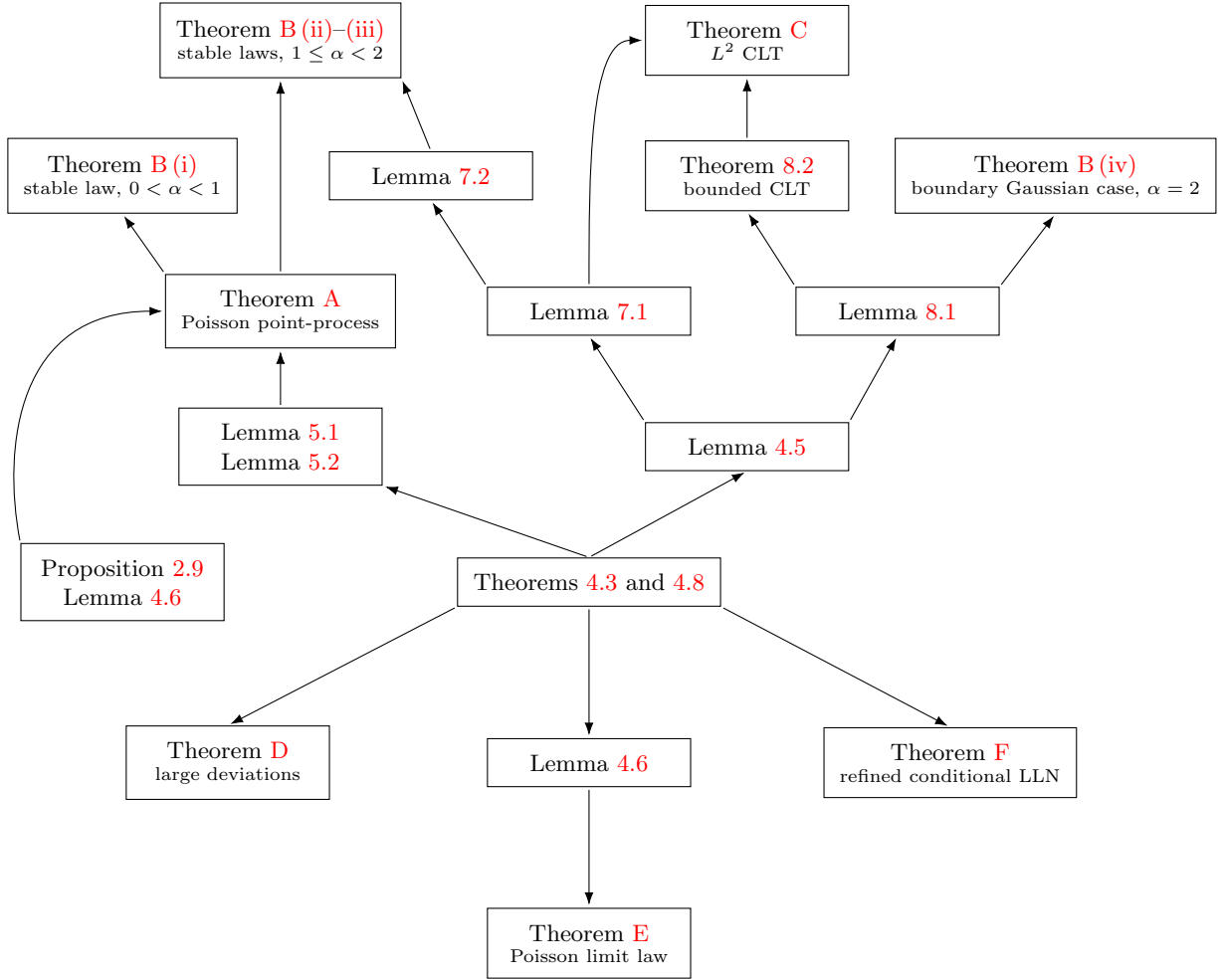
\begin{figure}[htbp]
\centering
\makebox[\textwidth]{
\begin{tikzpicture}[
    result/.style={
        draw=black,
        line width=0.4pt,
        align=center,
        font=\small,
        minimum width=2.7cm,
        inner xsep=6pt,
        inner ysep=6pt
    },
    dependency/.style={
        -{Latex[length=1.7mm, width=1.2mm]},
        line width=0.4pt,
        shorten >=1pt,
        shorten <=1pt
    }
]

\node[result] (tail) at (-6.2,0) {Proposition~\ref{prop:tail-asymptotics}\\Lemma~\ref{lem:uniform-small-targets}};
\node[result] (foundation) at (0,0) {Theorems~\ref{thm:spectralgap} and~\ref{thm:Qprocess}};

\node[result] (mixing) at (-4.1,1.8) {Lemma~\ref{lem:mix}\\Lemma~\ref{lem:cluster-mixing}};
\node[result] (compact) at (2.1,1.8) {Lemma~\ref{lem:l2-compactness}};

\node[result] (ppp) at (-4.1,3.6) {Theorem~\ref{thm:ppp}\\[-1mm]\scriptsize Poisson point-process};
\node[result] (moment) at (0,3.6) {Lemma~\ref{lem:conditional-sum-second-moment}};
\node[result] (triangular) at (4.1,3.6) {Lemma~\ref{lem:conditional-triangular-clt}};

\node[result] (stable1) at (-6.2,5.4) {Theorem~\ref{thm:main-stable-laws}\,\ref{it:stab-law-i}\\[-1mm]\scriptsize stable law, $0<\alpha < 1$};
\node[result] (smalljumps) at (-2.1,5.4) {Lemma~\ref{lem:conditional-second-moment}};
\node[result] (bounded) at (2.1,5.4) {Theorem~\ref{thm:conditional-clt-bounded}\\[-1mm]\scriptsize bounded CLT};
\node[result] (gaussian) at (6.2,5.4) {Theorem~\ref{thm:main-stable-laws}\,\ref{it:stab-law-iv}\\[-1mm]\scriptsize boundary  Gaussian case, $\alpha = 2$};

\node[result] (stable23) at (-4.1,7.2) {Theorem~\ref{thm:main-stable-laws}\,\ref{it:stab-law-ii}--\ref{it:stab-law-iii}\\[-1mm]\scriptsize stable laws, $1\leq \alpha < 2$};
\node[result] (l2) at (2.1,7.2) {Theorem~\ref{thm:conditional-clt-l2}\\[-1mm]\scriptsize $L^2$ CLT};

\draw[dependency] (foundation.north) -- (mixing.south east);
\draw[dependency] (foundation.north) -- (compact.south);
\draw[dependency] (tail.north west) .. controls (-7.8,1.8) and (-7.6,3.6) .. (ppp.west);
\draw[dependency] (mixing.north) -- (ppp.south);
\draw[dependency] (compact.north west) -- (moment.south);
\draw[dependency] (compact.north east) -- (triangular.south);
\draw[dependency] (ppp.north west) -- (stable1.south);
\draw[dependency] (ppp.north) -- (stable23.south);
\draw[dependency] (moment.north west) -- (smalljumps.south);
\draw[dependency] (smalljumps.north) -- (stable23.south east);
\draw[dependency] (triangular.north west) -- (bounded.south);
\draw[dependency] (triangular.north east) -- (gaussian.south);
\draw[dependency] (bounded.north) -- (l2.south);
\draw[dependency] (moment.north) .. controls (0,5.4) and (0,7.2) .. (l2.west);

\node[result] (concentration) at (-4.8,-2.4) {Theorem~\ref{thm:conditional-exponential-concentration}\\[-1mm]\scriptsize large deviations};
\node[result] (targets) at (0,-2.4) {Lemma~\ref{lem:uniform-small-targets}};
\node[result] (lln) at (4.8,-2.4) {Theorem~\ref{thm:refined-conditional-lln}\\[-1mm]\scriptsize refined conditional LLN};
\node[result] (poisson) at (0,-4.8) {Theorem~\ref{thm:conditional-poisson}\\[-1mm]\scriptsize Poisson limit law};

\draw[dependency] (foundation.south west) -- (concentration.north);
\draw[dependency] (foundation.south) -- (targets.north);
\draw[dependency] (foundation.south east) -- (lln.north);
\draw[dependency] (targets.south) -- (poisson.north);

\end{tikzpicture}
}
\caption{Logical structure of the proofs. Arrows indicate the main dependencies.}
\label{fig:theorem-dependencies}
\end{figure}

\section{Setting and statements of the main results}
\label{sec:main-results}

\subsection{Setting and standing hypothesis}
Let $(M,d_M)$ be a compact metric space and let $\rho$ be a fully supported Borel probability measure on $M$.
We may think of $\rho$ as a reference measure.
We enlarge $M$ with an isolated point $\partial$ and set $
\widehat{M}:=M\sqcup\{\partial\}.$
Following~\cite[\S2]{CastroLambOliconMendezRasmussen2024}, we use the standard definition of a Markov chain given in~\cite[Definition~III.1.1]{RogersWilliams1994}. More precisely, the tuple
$$
\mathbf X:=\left(
\Omega,\mathcal F,(\mathcal F_n)_{n\in\mathbb N_0},(X_n)_{n\in\mathbb N_0},(\mathcal P^n)_{n\in\mathbb N_0},(\mathbb P_x)_{x\in \widehat{M}}\right)
$$
is a time-homogeneous Markov chain on $\widehat{M}$, meaning that:
\begin{itemize}
\item $(\Omega,\mathcal F,(\mathcal F_n)_{n\in\mathbb N_0})$ is a filtered measurable space;
\item $(X_n)_{n\in\mathbb N_0}$ is an $(\mathcal F_n)_{n\in\mathbb N_0}$-adapted process taking values in $\widehat{M}$;
\item $(\mathcal P^n)_{n\in\mathbb N_0}$ is a time-homogeneous transition function satisfying the Chapman--Kolmogorov equations;
\item $(\mathbb P_x)_{x\in \widehat{M}}$ is a family of probability measures satisfying $\mathbb P_x(X_0=x)=1$ and, for every bounded measurable function $h:\widehat{M}\to\mathbb R$ and every $m,n\in\mathbb N_0$,
$$
\mathbb E_x\left[h(X_{n+m})\mid\mathcal F_n\right]=\mathcal P^mh(X_n) \quad \mathbb P_x\text{-almost surely,}$$ where
$$\mathcal P^mh(z):=\int_{\widehat{M}}h(y)\,\mathcal P^m(z,\d y).
$$
\end{itemize}
The one-step transition kernel is therefore
$$\mathcal P(x,A):=\mathbb P_x(X_1\in A),\  x\in \widehat{M},\ A\in\mathcal B( \widehat{M}).$$
We assume that $\partial$ is absorbing, namely $
\mathcal P(\partial,\{\partial\})=1.$
Equivalently, once the process reaches $\partial$, it remains there forever. Define the absorption time by
$$\tau:=\inf\left\{n\geq0:X_n=\partial\right\}=\inf\left\{n\geq0:X_n\notin M\right\},$$
with the convention $\inf\varnothing=\infty$. Since $\partial$ is absorbing, $
\{\tau>n\}=\{X_n\in M\}.$
For a bounded measurable function $f:M\to\mathbb R$, extended by zero at $\partial$, the killed transition operator satisfies
$$
\mathcal P^nf(x) = \mathbb E_x\left[f(X_n)\mathbbm 1_{\{\tau>n\}}\right] = \int_Mf(y)\,\mathcal P^n(x,\d y),\  x\in M.
$$
In particular, $\mathcal P^n\mathbbm 1_M(x) = \mathbb P_x(\tau>n).$

In the study of absorbing Markov chains, quasi-stationary measures are a natural generalisation of stationary measure, whereas quasi-ergodic measures describe the behaviour of Birkhoff averages conditioned upon survival. 
We stress that quasi-stationary and quasi-ergodic measures generally differ, as opposed to the classical non-absorbing setting, where the stationary measure also describes the statistics of the Birkhoff averages. Below, we recall their definitions. 

\begin{definition}
A probability measure $\mu$ on $M$ is called a \emph{quasi-stationary measure} for $X_n$ if there exists $\lambda\in(0,1]$ such that
$$
\int_M\mathcal P(x,A)\,\mu(\d x) = \lambda\mu(A)
$$
for every Borel set $A\in\mathcal B(M)$. Equivalently, $
\mathbb P_\mu(X_n\in A\mid\tau>n) = \mu(A)$ for every $n\geq1$ and every $A\in\mathcal B(M)$. In this case, $\lambda=\mathbb P_\mu(\tau>1).$

A probability measure $\nu$ on $M$ is called a \emph{quasi-ergodic measure} for $X_n$ if, for every bounded measurable observable $h:M\to\mathbb R$, 
\begin{align}\mathbb E_x\left[
\left.\frac{1}{n}\sum_{i=0}^{n-1}h\circ X_i\,\right|\,\tau>n\right]\xrightarrow[]{n\to\infty}\int_M h\d\nu\label{eq:quasi-ergodiclimit5}
\end{align}
for $\nu$-almost every $x\in M$. We refer to the expression on the left-hand side of \eqref{eq:quasi-ergodiclimit5} as a \emph{conditioned Birkhoff average}.
\end{definition}
\begin{remark}
Different indexing conventions for the conditioned Birkhoff average appear in the literature. Besides \eqref{eq:quasi-ergodiclimit5}, one may consider
$$
\mathbb E_x\left[\left.\frac{1}{n}\sum_{i=1}^{n}h\circ X_i\,\right|\,\tau>n\right]\ \text{or }\mathbb E_x\left[\left.\frac{1}{n+1}\sum_{i=0}^{n}h\circ X_i\,\right|\,\tau>n\right],
$$
the latter being used, for instance, in
\cite{DarrochSeneta1965,ColoniusRasmussen2021}. These conventions are equivalent for bounded $h$, since their differences from the average in \eqref{eq:quasi-ergodiclimit5} are bounded, respectively, by $
2\|h\|_\infty/n$ and $2\|h\|_\infty/(n+1).$ For the unbounded observables considered in this paper, the same equivalence holds under the hypotheses of the corresponding results. Indeed, reindexing only introduces endpoint terms that vanish after the relevant normalisation. The choice in \eqref{eq:quasi-ergodiclimit5} above follows~\cite{Ocafrain2023,BassolsLamb2023, CastroLambOliconMendezRasmussen2024,CastroGoverseLambRasmussen2024,BassolsCastroLamb2025}.
\end{remark}

Throughout the paper, we assume that $X_n$ satisfies Hypothesis~\ref{hyp:H} (cf.~\cite[Hypothesis H]{CastroLambOliconMendezRasmussen2024}), ensuring existence and uniqueness of quasi-stationary and quasi-ergodic measures (see Theorem~\ref{thm:Castro}).

\begin{hypothesis}[H]
\label{hyp:H}
We assume that the following conditions hold.
\begin{enumerate}[label=(H\arabic*)]
\item \label{it:hypH-i}
For every $x\in M$, $\mathcal P(x,\cdot)\ll\rho$. Denoting the corresponding Radon--Nikodym derivative by
$$
\kappa(x,y):=\frac{\mathcal P(x,\d y)}{\rho(\d y)},
$$
the map $x\mapsto\kappa(x,\cdot)$ is continuous from $M$ into $L^1(M,\rho)$. Equivalently, for every $\varepsilon>0$, there exists $\delta>0$ such that, for all $x,z\in M$,
$$
d_M(x,z)<\delta\implies\int_M|\kappa(x,y)-\kappa(z,y)|\,\d\rho(y)<\varepsilon.
$$

\item \label{it:hypH-ii}
Define the set of eventually escaping points by
$$
Z:=\left\{x\in M:\mathcal P^k(x,M)=0
\text{ for some }k\geq1\right\}.
$$
We assume that $\rho(M\setminus Z)>0$ and that, for every $x\in M\setminus Z$ and every non-empty relatively open set $A\subset M\setminus Z$, there exists $n=n(x,A)\geq1$ such that $\mathcal P^n(x,A)>0$.

\item \label{it:hypH-iii}
The chain is aperiodic on $M\setminus Z$: if
$$
M\setminus Z=C_0\sqcup C_1\sqcup\cdots\sqcup C_{k-1}
$$
is a partition into non-empty Borel sets satisfying
$$
\left\{x\in M\setminus Z:\mathcal P(x,C_i)>0\right\}
\subset C_{i+1\ (\mathrm{mod}\ k)}
$$
for every $i\in\{0,\ldots,k-1\}$, then $k=1$.
\end{enumerate}
\end{hypothesis}

In Section~\ref{sec:examples-hypothesis-H}, we provide several examples of absorbing Markov chains satisfying Hypothesis~\ref{hyp:H}. As mentioned above, under this hypothesis $X_n$ admits a unique quasi-ergodic probability measure on $M.$

\begin{theorem}[{\cite[Theorems~A--C]{CastroLambOliconMendezRasmussen2024}}]
\label{thm:Castro}
Assume that the absorbing Markov chain $X_n$ satisfies Hypothesis~\ref{hyp:H}. Then $X_n$ admits a unique quasi-stationary probability measure $\mu$ and a unique quasi-ergodic probability measure $\nu$ on $M$. Moreover, there exist a non-negative density $m\in L^1(M,\rho)$ and a continuous function $\eta\in\mathcal C^0(M)$ such that
$$
\mu(\d x)=m(x)\rho(\d x)
\quad \text{and} \quad
\nu(\d x)=\eta(x)\mu(\d x)=\eta(x)m(x)\rho(\d x).
$$
Moreover, $M\setminus Z\subset\operatorname{supp}\mu$ and the function $\eta$ is strictly positive on $M\setminus Z$, vanishes on $Z$ and satisfies $\mu(\eta)=1$.
\end{theorem}

\begin{remark}
\label{rmk:no-kill}
In the absence of killing, $\eta\equiv1$ and $\nu=\mu$. This is generally not the case for absorbing Markov processes.
\end{remark}

Throughout the remainder of the paper, $\mu$, $\nu$ and $\eta$ denote the measures and function in Theorem~\ref{thm:Castro}. We fix a non-negative, finite-valued Borel representative $m$ of the Radon--Nikodym derivative $\d\mu/\d\rho$. Since $\mu(\{m=0\})=0$, we have $0<m<\infty$ $\mu$-almost everywhere, and hence also $\nu$-almost everywhere.

\subsection{Main results} 
We begin by introducing some notation essential to the statement of the main theorems. In particular, we employ different modes of convergence and asymptotic behaviour throughout.

\begin{notation}
We use the following convergence notations:
\begin{enumerate}
\item Let $Y_n$ and $Y$ be random variables on a metric space $(S, d_S)$, where $Y_n$ is considered under $\mathbb P_{x}(\cdot\mid \tau>n)$ and $Y$ under $\mathrm{Prob}$. We write $
Y_n\xrightarrow[d]{n\to\infty}Y$ if $$\mathbb E_{x}[h(Y_n)\mid \tau>n]
\xrightarrow[]{n\to\infty}
\mathbb E[h(Y)]$$ for every bounded continuous function $h:S\to\mathbb R$. Equivalently,
$$
\mathbb P_{x}(Y_n\in\cdot\mid \tau>n)
\xRightarrow{n\to\infty}
\mathrm{Prob}(Y\in\cdot).
$$
Thus, $\xRightarrow{n\to\infty}$ denotes weak$^*$ convergence of probability measures.
\item We say that $Y_n$ converges to $Y$ in $\mathbb P_{x}(\cdot\mid\tau>n)$-probability if, for every $\varepsilon>0$,
$$
\mathbb P_{x}\bigl(d_S(Y_n,Y)>\varepsilon\mid \tau>n\bigr)
\xrightarrow[]{n\to\infty}
0.
$$
\item For $u_n,u\in L^2(M,\mu)$, the notation $
u_n\xrightarrow[w]{n\to\infty}u
\ \text{in }L^2(M,\mu)$
means that
$$
\int_Mu_nv\,\d\mu \xrightarrow[]{n\to\infty} \int_Muv\,\d\mu\ \text{for every }v\in L^2(M,\mu).$$
\end{enumerate}
\end{notation}

\begin{notation}
Let $a\in\mathbb{R}$, and let $f,g:(a,\infty)\to\mathbb{R}$, with $g(t)\neq 0$ for all sufficiently large $t$.
\begin{enumerate}
    \item We write $f\sim g$ if $\lim_{t\to\infty} f(t)/g(t)=1.$
    \item We write $f(t)=\mathcal{O}(g(t))$ if there exist constants $C>0$ and $t_0>a$ such that ${
    |f(t)|\leq C|g(t)|}$
    for every $t\geq t_0$.
    \item We write $f(t)= \smallO(g(t))$ if $\lim_{t\to\infty} f(t)/g(t)=0$.
\end{enumerate}

More generally, let $\Lambda$ be a set and let $
F:\Lambda\times(a,\infty)\to\mathbb{R}.$
We write $F(x,t)=\mathcal{O}_x(g(t))$
if, for every $x\in\Lambda$, there exist constants $C_x>0$ and $t_x>a$, possibly depending on $x$, such that
$$
|F(x,t)|\leq C_x|g(t)|
$$
for every $t\geq t_x$. We say that $F(x,t) = \mathcal O(g(t))$ when $C_x$ and $t_x$ are uniform on $x\in \Lambda.$
\end{notation}

We are now ready to state the main results of this paper. 

\subsubsection{Poisson point-process convergence for \texorpdfstring{$f_\beta$}{f\_β}} \label{sec:ppp-intro}

Before introducing the observable $f_\beta$, let us specify the local conditions imposed on its singularity point $x_0$.

\begin{definition}
\label{def:reg}
Fix $d>0$. A point $y\in M$ is called $(\rho,d)$-\emph{regular} if there exists $C_d(y)\in(0,\infty)$ such that
$$
\lim_{r\to 0}\frac{\rho(B_r(y))}{r^d}=C_d(y).
$$
We say that $x_0$ is \emph{a $(\rho,d)$-admissible point}, if $x_0$ is $(\rho,d)$-regular, $m(x_0)>0$ and
\begin{equation}
\lim_{r\to 0}
\frac{1}{\rho(B_r(x_0))}
\int_{B_r(x_0)}|m(y)-m(x_0)|\,\d\rho(y)
=0,\label{eq:lebesgue}
\end{equation}
where we recall that $m \in L^1(M,\rho)$ is the density of the unique quasi-stationary measure provided by Theorem \ref{thm:Castro}.
\end{definition}

\begin{remark}
For a Riemannian volume form $\rho$ restricted to a compact domain, every interior point is $(\rho,d)$-regular, where $d$ is the dimension of the ambient manifold. Since $m\in L^1(M,\rho)$, we have that \eqref{eq:lebesgue} holds for $\rho$-almost every point of $M\setminus Z$. 
In particular, any $(\rho,d)$-regular point $x_0\in M\setminus Z$ at which $m$ is continuous and strictly positive is $(\rho,d)$-admissible. 
\label{rmk:admissible}
\end{remark}

We fix, once and for all, a $(\rho,d)$-admissible point $x_0\in M\setminus Z$. For $\beta>0$, define
$$
f_\beta(x):=
\begin{cases}
d_M(x,x_0)^{-\beta},&x\neq x_0,\\
1,&x=x_0,
\end{cases} \quad \alpha:=\frac d\beta,
\ C(x_0):=C_d(x_0)\eta(x_0)m(x_0).
$$

The value at $x_0$ makes $f_\beta$ finite and positive everywhere. Since $(\rho,d)$-regularity implies $\rho(\{x_0\})=0$ and $\mu,\nu\ll\rho$, this convention does not affect any tail asymptotic, integral or centring term. Any fixed positive value gives the same limiting results, and its time-zero contribution vanishes after normalisation.

We are interested in determining the limiting distribution, under $\mathbb P_x(\cdot\mid\tau>n)$, of the Birkhoff sums
$$
S_nf_\beta:=\sum_{i=0}^{n-1}f_\beta\circ X_i.
$$
To do so, we begin by determining the tail of $f_\beta$ with respect to the quasi-ergodic measure $\nu$. This identifies the index $\alpha$, the normalising sequence $(B_n)$ and the intensity measure appearing in the limiting point-process.
\begin{proposition}
\label{prop:tail-asymptotics}
The observable $f_\beta$ satisfies $
\nu(f_\beta>t)\sim C(x_0)t^{-\alpha}.$
Consequently, for $B_n:=\left(C(x_0)n\right)^{1/\alpha},$
we have 
\begin{align}
n\nu(f_\beta>uB_n)\xrightarrow[]{n\to\infty}u^{-\alpha} = \int_u^\infty\alpha y^{-\alpha-1}\,\d y,\ \text{for every }u>0.\label{eq:tailbn}
\end{align}
\end{proposition}

\begin{proof}
From~\eqref{eq:lebesgue} we have that
\begin{align}
\left|\frac{\mu(B_r(x_0))}{\rho(B_r(x_0))}-m(x_0)\right|
&\leq\frac{1}{\rho(B_r(x_0))}\int_{B_r(x_0)}|m(y)-m(x_0)|\,\d\rho(y)\xrightarrow[]{r\to 0}0.
\label{eq:mu-ball-density}
\end{align}
As $m(x_0)>0$, this gives $\mu(B_r(x_0))\sim m(x_0)\rho(B_r(x_0))$. Moreover, continuity of $\eta$ at $x_0$ implies
$$
\left|\nu(B_r(x_0))-\eta(x_0)\mu(B_r(x_0))\right|\leq\sup_{y\in B_r(x_0)}|\eta(y)-\eta(x_0)|\,\mu(B_r(x_0))=\smallO\bigl(\mu(B_r(x_0))\bigr).
$$
Combining this estimate with \eqref{eq:mu-ball-density} and $(\rho,d)$-regularity of $x_0$, we obtain
\begin{align}
\nu(B_r(x_0))&\sim\eta(x_0)\mu(B_r(x_0))\sim\eta(x_0)m(x_0)\rho(B_r(x_0))\sim C(x_0)r^d.
\label{eq:nu-ball-asymptotics}
\end{align}
Regularity also implies $\rho(\{x_0\})=0$, and hence $\mu(\{x_0\})=\nu(\{x_0\})=0$. Since $f_\beta(x_0)=1$, for every $t>1$ we have
\begin{align}
\nu(f_\beta>t)&=\nu\left(B_{t^{-1/\beta}}(x_0)\setminus\{x_0\}\right)=\nu\left(B_{t^{-1/\beta}}(x_0)\right)\sim C(x_0)t^{-\alpha},
\label{eq:tail}
\end{align}
where the asymptotic follows from \eqref{eq:nu-ball-asymptotics}. With $B_n=(C(x_0)n)^{1/\alpha}$, the proof follows.
\end{proof}

Proposition~\ref{prop:tail-asymptotics} gives $\nu(f_\beta>uB_n)\sim u^{-\alpha}/n$ for every fixed $u>0$. Thus, among the first $n$ observations $f_\beta \circ X_{0},\ldots, f_\beta \circ X_{n-1}$, the number of values larger than $uB_n$ is expected to remain of order one. 

To describe the limiting distribution of these normalised large observations, we leverage properties of the space $\mathcal M_p(E)$ of locally finite point measures on $E=(0,\infty]$, properties of random point measures and of Poisson random measures. We use the standard framework of random measures and vague convergence~\cite[Ch.~1--4]{Kallenberg2017}, which we recall below.

\begin{definition}\label{def:Mp}
Let $E:=(0,\infty]$, where $\infty$ is included with its usual neighbourhoods $(R,\infty], R > 0$. A \emph{point measure} on $E$ is a measure of the form $
\xi
=
\sum_{i\in J}\delta_{x_i},$
where $J$ is at most countable, $x_i\in E$ and $\delta_{x_i}$ denotes the Dirac measure at $x_i$. We say that a measure $\xi$ is locally finite in $E$, if $\xi([\varepsilon,\infty])<\infty$
for every $\varepsilon>0$. We denote the space of all locally finite point measures on $E$ by $\mathcal M_p(E)$.

The space $\mathcal M_p(E)$ is endowed with the vague topology  i.e.
$$
\xi_n\xrightarrow[]{n\to\infty}\xi \ \text{in }\mathcal M_p(E)\ \text{if and only if }\int_E h\,\d\xi_n \xrightarrow[]{n\to\infty}\int_E h\,\d\xi\ \text{for every }h\in C_c(E),
$$
where $C_c(E)$ is the space of continuous real-valued functions on $E$ with compact support. Notice that such a function vanishes in a neighbourhood of zero, but it need not vanish at infinity.
\end{definition}
\begin{definition}
A \emph{random point measure} on $E$ is a measurable map
$$
N:(\Omega,\mathcal F,\mathrm{Prob})\to\mathcal M_p(E),
$$
where $\mathcal M_p(E)$ is equipped with the Borel $\sigma$-algebra generated by the vague topology.

Let $\upsilon$ be a locally finite measure on $E$. We say that the random point measure $N$ is a \emph{Poisson random measure} with intensity $\upsilon$, and write $
N\sim\operatorname{PRM}(\upsilon),$ if
\begin{itemize}
\item  for any Borel set $A$ such that $\upsilon(A) <\infty$ we have that $N(A)\sim\operatorname{Poi}\left(\upsilon(A)\right)$, i.e.
$$\mathrm{Prob}\left(N(A)=m\right) = e^{-\upsilon(A)}\frac{\upsilon(A)^m}{m!}$$
for every $m\in\mathbb N_0$; and
\item for every finite collection of pairwise disjoint Borel sets $A_1,\ldots,A_k\subset E$ satisfying ${\upsilon(A_j)<\infty}$ for every $j\in\{1, \dots, k\}$, it holds that the random variables $N(A_1),\ldots,N(A_k)$ are independent.
\end{itemize}
\end{definition}

For $0<\alpha<2$ and $B_n = \left(C(x_0)n\right)^{1/\alpha}$ we write, for the remainder of the paper,
$$
N_n := \sum_{j=0}^{n-1} \delta_{\frac{f_\beta\circ X_j}{B_n}}
\quad
\text{and}
\quad \upsilon_\alpha(\d y) := \alpha y^{-\alpha-1}\,\d y.
$$

The following result describes the asymptotic distribution of the normalised large observations of $f_\beta$. It shows that, conditioning upon survival, their point-process converges to a Poisson random measure with intensity $\upsilon_\alpha$. 
\begin{maintheorem}
\label{thm:ppp}
Let $0<\alpha<2$. For every $x\in M\setminus Z$, under the probability measures ${\mathbb P_x(\cdot\mid\tau>n)}$,
$$N_n \xrightarrow[d]{n\to\infty} N\ \text{in }\mathcal M_p((0,\infty]),$$
where $N\sim\operatorname{PRM}(\upsilon_\alpha)$. Equivalently,
$$\mathbb P_x\left(N_n\in\cdot\mid\tau>n\right)\xRightarrow[]{n\to\infty}\mathrm{Prob}\left(N\in\cdot\right).$$
 Moreover, if $I_1,\ldots,I_k$ are pairwise disjoint intervals of the form $I_j=(a_j,b_j]\subset(0,\infty]$, with $a_j>0$, $j = 1, \dots, k$, then
$$\left(N_n(I_1),\ldots,N_n(I_k)\right) \xrightarrow[d]{n\to\infty}\left(Y_1,\ldots,Y_k\right),$$
where $Y_1,\ldots,Y_k$ are independent and $Y_j \sim \operatorname{Poi}\left(\upsilon_\alpha(I_j)\right),$ for every $j = 1, \dots, k$.
\end{maintheorem}
Theorem~\ref{thm:ppp} is proved in Section~\ref{sec:ppp}.

\begin{remark}
    We emphasise that the random variables $Y_1,\ldots,Y_k$ in Theorem~\ref{thm:ppp} are independent, while the random variables $N_n(I_1),\ldots,N_n(I_k)$ need not be independent.
\end{remark}

\subsubsection{Conditional stable and Gaussian laws for \texorpdfstring{$f_\beta$}{f\_β}} \label{sec:cond-stab-laws-intro}
 
The point-process limit describes the rare large observations of $f_\beta$, but the corresponding Birkhoff sums also contain contributions from smaller observations. The next theorem combines the point-process convergence with suitable truncation estimates and gives the limiting distribution of the normalised sums in every regime of $\alpha$. 

\begin{maintheorem}
\label{thm:main-stable-laws}
Let $x\in M\setminus Z$ and let $N \sim \mathrm{PRM}(\upsilon_\alpha) $. Under the probability measures $\mathbb P_x(\cdot\mid\tau>n)$, the following convergences hold.
\begin{enumerate}[label = (\roman*)]
\item \label{it:stab-law-i} If $0<\alpha<1$, set $B_n:=\left(C(x_0)n\right)^{1/\alpha}$. Then
$$
\frac{1}{B_n}\sum_{j=0}^{n-1}f_\beta\circ X_j\xrightarrow[d]{n\to\infty}Z_\alpha,$$ 
where $Z_\alpha:=\int_{(0,\infty]}y\,N(\d y)$ is the positive $\alpha$-stable random variable satisfying 
$$
\mathbb E\left[e^{-tZ_\alpha}\right]=\exp\left(-\Gamma(1-\alpha)t^\alpha\right),\ t\geq0.
$$
\item \label{it:stab-law-ii} If $\alpha=1$, set $B_n:=\left(C(x_0)n\right)^{1/\alpha}$ and  $A_n:=n\nu\left(f_\beta\mathbbm 1_{\{f_\beta\leq B_n\}}\right).$
Then
$$
\frac{1}{B_n}\left(\sum_{j=0}^{n-1}f_\beta\circ X_j-A_n\right)\xrightarrow[d]{n\to\infty}Z_1,$$ 
where $Z_1:=\lim_{\varepsilon\to0}\left(\int_{(\varepsilon,\infty]}y\,N(\d y)-\int_{(\varepsilon,1]}y\,\upsilon_1(\d y)\right).$
Moreover, for all $t \in \mathbb R$,
$$
\mathbb E\left[e^{itZ_1}\right]=\exp\left(\int_0^\infty\left(e^{ity}-1-ity\mathbbm 1_{(0,1]}(y)\right)y^{-2}\,\d y\right).
$$
\item \label{it:stab-law-iii} If $1<\alpha<2$, set $B_n:=\left(C(x_0)n\right)^{1/\alpha}$, then $f_\beta\in L^1(M,\nu)$ and
$$
\frac{1}{B_n}\left(\sum_{j=0}^{n-1}f_\beta\circ X_j-n\nu(f_\beta)\right)\xrightarrow[d]{n\to\infty}Z_\alpha,
$$
where $
Z_\alpha:=\lim_{\varepsilon\to0}\left(\int_{(\varepsilon,\infty]}y\,N(\d y)-\int_{(\varepsilon,\infty]}y\,\upsilon_\alpha(\d y)\right)$. Moreover, for all $t \in \mathbb R$,
$$
\mathbb E\left[e^{itZ_\alpha}\right]=\exp\left(\int_0^\infty\left(e^{ity}-1-ity\right)\alpha y^{-\alpha-1}\,\d y\right).
$$
\item \label{it:stab-law-iv} If $\alpha=2$, set $
B_n:=\sqrt{C(x_0)n\log n}.$
Then $f_\beta\in L^1(M,\nu)$ and
$$
\frac{1}{B_n}\left(\sum_{j=0}^{n-1}f_\beta\circ X_j-n\nu(f_\beta)\right)\xrightarrow[d]{n\to\infty}\mathcal N(0,1).
$$
\end{enumerate}
\end{maintheorem}

\begin{remark}
The conclusions of Theorem~\ref{thm:main-stable-laws} extend to
$$f_{\beta,g}:=f_\beta g\ \text{with}\ g\in\mathcal C^0(M)\ \text{and}\ g(x_0)\neq0.
$$
As above, $f_{\beta,g}(x_0)$ may be assigned any fixed finite value. Set $
C_g(x_0):=|g(x_0)|^\alpha C(x_0)$ and $B_{n,g}:=|g(x_0)|B_n.$ For $0<\alpha<1$, no centring is required. For $\alpha=1$, the centring is given by $
A_{n,g}:=n\nu\left(f_{\beta,g}\mathbbm 1_{\{|f_{\beta,g}|\leq B_{n,g}\}}\right).$ For $1<\alpha\leq2$, use $A_{n,g}:=n\nu(f_{\beta,g})$. If $g(x_0)>0$, the limiting laws are unchanged, whereas if $g(x_0)<0$, the stable limits are reflected about the origin. The Gaussian limit for $\alpha=2$ is unchanged. These extensions follow from the same arguments; we take $g\equiv1$ to simplify the notation.
\end{remark}

Parts \ref{it:stab-law-i}--\ref{it:stab-law-iii} of Theorem~\ref{thm:main-stable-laws}, which relate to $\alpha \in (0,2)$, are proved in Section~\ref{sec:stable-laws}. Part~\ref{it:stab-law-iv}, relating to the Gaussian boundary case $\alpha =2$, is proved in Section~\ref{sec:gaussian-laws}.

\subsubsection{Conditional central limit and large deviations}

The preceding results in Sections~\ref{sec:cond-stab-laws-intro} and~\ref{sec:ppp-intro} concern the particular heavy-tailed observable $f_\beta$. We now turn to general observables in the finite-variance regime. The next theorem establishes a conditional central limit theorem for observables in $L^2(M,\mu)$ and identifies their asymptotic variance.  

The proof of conditional central limit theorem for $L^2(M,\mu)$ observables is divided into two stages. We first prove the conditional central limit theorem for bounded measurable observables in Theorem~\ref{thm:conditional-clt-bounded}, using spectral perturbation of the twisted killed operator. We then extend this result to observables in $L^2(M,\mu)$ via truncation, yielding the following theorem.

\begin{maintheorem}
\label{thm:conditional-clt-l2}
Let $g:M\to\mathbb R$ be a finite-valued measurable function lying in $L^2(M,\mu)$, and set $\overline g:=g-\nu(g)$. Then, for every $x\in M\setminus Z$, under the probability measures $\mathbb P_x(\cdot\mid\tau>n)$,
$$
\frac{1}{\sqrt n}\sum_{j=0}^{n-1}\overline g\circ X_j\xrightarrow[d]{n\to\infty}\mathcal N(0,\sigma_g^2),\quad \text{where } \sigma_g^2:=\nu(\overline g^2)+
2\sum_{m=1}^{\infty}\mu\left(\overline g\,\frac{1}{\lambda^m}\mathcal P^m(\eta\overline g)\right).
$$
The series converges absolutely and $\sigma_g^2\geq0$.
\end{maintheorem}

Theorem~\ref{thm:conditional-clt-l2} is proved in Section~\ref{sec:l2-clt}.

The central limit theorem describes fluctuations on the scale of $\sqrt n$. For bounded observables, we can also control deviations of the Birkhoff average from its quasi-ergodic limit on the larger scale $n$. The following theorem gives an exponential bound for these deviations.
\begin{maintheorem}
\label{thm:conditional-exponential-concentration}
Let $h:M\to\mathbb R$ be bounded and measurable. For every $x\in M\setminus Z$, there exist constants $C_{x}>0$ and $a_h>0$ such that for every $\e>0$
\begin{align}
\mathbb P_x\left(\left.\left|\frac{1}{n}\sum_{j=0}^{n-1}h\circ X_j-\nu(h)\right|>\varepsilon\right|\tau>n\right)\leq C_{x}e^{-a_h \e^2 n}
\label{ref:conditional-exponential-concentration}
\end{align}
for every $n\geq1$. 
\end{maintheorem}

Theorem~\ref{thm:conditional-exponential-concentration} is proved in Section~\ref{sec:exponential-concentration}. 

\subsubsection{A conditional Poisson law for shrinking targets}

We conclude the main results in this paper with a conditional Poisson law for visits to shrinking neighbourhoods.
\begin{definition}
Let $x_1\in M$. A sequence $(U_n)_{n\geq1}$ of measurable sets is called a sequence of \emph{shrinking neighbourhoods} of $x_1$ if $x_1\in U_n$ for every $n$ and, for every neighbourhood $V$ of $x_1$, there exists $n_V\geq1$ such that
$U_n\subset V$
for every $n\geq n_V$. 
\end{definition}

The following theorem shows that, if $n\nu(U_n)\to t\in(0,\infty)$, then the number of visits to $U_n$ converges to a Poisson random variable with parameter $t$.

\begin{maintheorem} \label{thm:conditional-poisson}
Let $(U_n)_{n\geq1}$ be a sequence of measurable neighbourhoods shrinking to $x_1\in M\setminus Z$. Suppose that $n\nu(U_n) \xrightarrow[]{n\to\infty}t$ for some $t\in(0,\infty)$. Then, for every $x\in M\setminus(Z\cup\{x_1\})$, under the probability measures $\mathbb P_x(\cdot\mid\tau>n)$,
$$\sum_{j=0}^{n-1}\mathbbm 1_{U_n}\circ X_j\xrightarrow[d]{n\to\infty}Y,$$
where $Y\sim\operatorname{Poi}(t)$. Equivalently,
$$
\mathbb P_x\left(\left.\sum_{j=0}^{n-1}\mathbbm 1_{U_n}\circ X_j\in\cdot\right|\tau>n\right)\xRightarrow[]{n\to\infty}\operatorname{Poi}(t).
$$
\end{maintheorem}

Theorem~\ref{thm:conditional-poisson} is proved in Section~\ref{sec:conditional-poisson}.
\begin{remark}
The target point $x_1$ in Theorem~\ref{thm:conditional-poisson} is independent of the singularity point $x_0$ and may coincide with it. If $X_0=x_1$, then
$$
\sum_{j=0}^{n-1}\mathbbm 1_{U_n}\circ X_j=1+\sum_{j=1}^{n-1}\mathbbm 1_{U_n}\circ X_j.
$$
Applying the argument in the proof of Theorem~\ref{thm:conditional-poisson} to the count starting at time $1$ therefore gives, under $\mathbb P_{x_1}(\cdot\mid\tau>n)$,
$$
\sum_{j=0}^{n-1}\mathbbm 1_{U_n}\circ X_j
\xrightarrow[d]{n\to\infty}1+Y,
\quad \text{with}\ Y\sim\operatorname{Poi}(t).
$$
\end{remark}

\subsubsection{Refined conditional laws of large numbers}
The quasi-ergodic limit \eqref{eq:quasi-ergodiclimit5} typically concerns the expectation of the conditioned Birkhoff average of a bounded observable. We may refine this result by considering conditional convergence in $L^1$ and in probability, as well as more general observables. Indeed, extensions to unbounded observables have been obtained in symmetric continuous-time settings under Kato-type assumptions~\cite{ChenJian2017, ChenJian2018, KimTagawaVelleret2026}. 

For discrete-time absorbing Markov chains satisfying Hypothesis~\ref{hyp:H}, we prove conditional convergence in probability for observables in $L^1(M,\nu)$ and conditional convergence in $L^1$ for observables in $L^1(M,\mu)$, for $\nu$-almost every initial state in $M\setminus Z$, assuming that the observable vanishes on $Z$. For bounded measurable observables, conditional convergence in $L^1$ holds for every initial state in $M\setminus Z$, without the assumption of vanishing on $Z$.

\begin{maintheorem}
\label{thm:refined-conditional-lln}
Let $X_n$ be an absorbing Markov chain satisfying Hypothesis~\ref{hyp:H}.
\begin{enumerate}[label=(\roman*)]
\item \label{it:cond-lln-1}
Let $h:M\to\mathbb R$ be bounded and measurable. Then, for every $x\in M\setminus Z$,
\begin{align}
\mathbb E_x\left[\left.\left|\frac1n\sum_{i=0}^{n-1}h\circ X_i-\nu(h)\right|\,\right|\,\tau>n\right]
&\xrightarrow[]{n\to\infty}0.
\label{eq:conditional-lln-bounded-l1}
\end{align}
 
\item \label{it:cond-lln-2}
Let $h:M\to\mathbb R$ be measurable, with $h(z)=0$ for every $z\in Z$.
If $h\in L^1(M,\mu)$, then, for $\nu$-almost every $x\in M\setminus Z$,
$$
\mathbb E_x\left[
\left.\left|\frac1n\sum_{i=0}^{n-1}h\circ X_i-\nu(h)\right|\,\right|\,\tau>n\right]\xrightarrow[]{n\to\infty}0.
$$

\item \label{it:cond-lln-3}
Let $h:M\to\mathbb R$ be measurable, with $h(z)=0$ for every $z\in Z$.
If $h\in L^1(M,\nu)$, then, for $\nu$-almost every $x\in M\setminus Z$
and every $\varepsilon>0$,
\begin{align}
\mathbb P_x\left(\left.\left|\frac1n\sum_{i=0}^{n-1}h\circ X_i-\nu(h)\right|>\varepsilon\,\right|\,\tau>n\right)&\xrightarrow[]{n\to\infty}0.
\label{ref:l1-conditional-probability}
\end{align}
\end{enumerate}
\end{maintheorem}

We defer the proof of Theorem~\ref{thm:refined-conditional-lln} to Appendix~\ref{ref:appendix} since it does not concern the paper's central topic of conditional limit laws. Nevertheless, these stronger forms of convergence for absorbing Markov processes satisfying Hypothesis~\ref{hyp:H} are not readily available in the literature.

\begin{remark}
The distinction between $\mu$ and $\nu$ in Theorem~\ref{thm:refined-conditional-lln} is relevant because $\nu=\eta\mu$ and $\eta$ may approach zero near $Z$. Hence, $L^1(M,\mu)\subset L^1(M,\nu),$ but the reverse inclusion need not hold. Our proof of convergence in conditional mean and of the quasi-ergodic expectation limit requires integrability with respect to $\mu$, whereas convergence in conditional probability only requires integrability with respect to $\nu$. It remains unclear whether these conclusions remain valid when integrability with respect to $\mu$ is replaced by integrability with respect to $\nu$. If $\eta$ is bounded away from zero, the two spaces coincide with equivalent norms.
\end{remark}

\section{Examples satisfying Hypothesis~\ref{hyp:H}}
\label{sec:examples-hypothesis-H}
We present five examples of systems satisfying Hypothesis~\ref{hyp:H}. The first two use the soft-killing construction introduced in~\cite{BassolsCastroLamb2025,BassolsCastro2025} for weighted absorbing Markov processes, whereas the remaining three concern hard killing.
The systems considered in
\cite[Examples~2.6--2.9]{CastroLambOliconMendezRasmussen2024}
also satisfy Hypothesis~\ref{hyp:H} and are therefore covered by the results of the present paper.

Let $\mathcal K$ be a Markov transition kernel on $M$ and let
$\varphi:M\to\mathbb R$ be continuous. Set
$$
\overline\varphi:=\varphi-\sup_M\varphi^+\quad \text{and} \quad\varphi^+(x):=\max\left\{\varphi(x),0\right\},
$$
so that $\overline\varphi\leq0$. Given a Markov chain $X_n$ with transition kernel $\mathcal K$, its associated $e^{\overline\varphi}$-weighted process is obtained by setting
$$
X_{n+1}^{\varphi}
=
\begin{cases}
X_{n+1},&\text{with probability }e^{\overline\varphi(X_n)},\\
\partial,&\text{with probability }1-e^{\overline\varphi(X_n)},
\end{cases}
$$
where the probability of soft killing are, in time, independent. In other words, the Markov chain $X_n^\varphi$ is described by the transition kernel
\begin{align*}
\mathcal P_\varphi(x,\d y)&=e^{\overline\varphi(x)}
\mathcal K(x,\d y).
\end{align*}
Since $e^{\overline\varphi(x)}>0$, weighting does not change which transitions have positive probability. Consequently, the irreducibility and aperiodicity of $\mathcal K$ are inherited by $\mathcal P_\varphi$.

We emphasise that our results also cover systems without escape. Indeed, as noted in Remark~\ref{rmk:no-kill}, if
$\mathcal P(x,M)=1$
for every $x\in M$, then
$\tau=\infty$
almost surely, $\lambda=1,$ $\eta=\mathbbm 1_M$ and $\nu=\mu.$
Moreover, $\mathbb P_x(\,\cdot\mid\tau>n)=\mathbb P_x$ for every $n\geq1$. Consequently, all the conditional limit theorems proved in this paper reduce to the corresponding classical limit theorems for the conservative Markov chain.

In all the examples below, $\partial$ denotes an absorbing cemetery state, so that $X_n=\partial$ implies $X_{n+1}=\partial.$ In all examples below, $\rho$ is the normalised volume measure on $M$. Moreover, the quasi-stationary density $m$ (with respect to $\rho$) in such examples is easily verified to be a continuous function on $M\setminus Z$. By Remark~\ref{rmk:admissible}, every point $x_0 \in M \setminus Z$ is $(\rho, d)$-admissible, where $d = \mathrm{dim} M$.

\begin{example}[Additive bounded noise on the torus]

Let $M=\mathbb T^d$, $d\in \mathbb N$, let $\rho$ be the normalised Lebesgue measure on $\mathbb T^d$ and let $T:\mathbb T^d\to\mathbb T^d$ be a continuous map with a dense forward orbit. Let $(\omega_n)_{n\geq0}$ be independent and identically distributed random variables on $\mathbb T^d$ with density
$q\in L^1(\mathbb T^d,\rho)$ satisfying
$$q>0\ \text{$\rho$-almost everywhere on }B_\varepsilon(0)\quad \text{and}\quad\operatorname{supp}(q)\subset\overline{B_\varepsilon(0)},$$
for $\varepsilon > 0$.
Let $(\xi_n)_{n\geq0}$ be independent and uniformly distributed random variables on $[0,1]$, independent of $(\omega_n)_{n\geq0}$. Given $X_0=x\in M$, define
$$
X_{n+1}:=\begin{cases}T(X_n)+\omega_n\ \operatorname{mod}\mathbb Z^d,&X_n\in M\quad \text{and}\quad \xi_n\leq e^{\overline\varphi(X_n)},\\
\partial,&\text{otherwise}.\end{cases}
$$
Thus, at each step, the chain first survives the soft killing with probability
$
e^{\overline\varphi(X_n)}
$
and then follows the random perturbation
$
T_{\omega_n}(X_n)=T(X_n)+\omega_n.$
Its killed transition kernel is
\begin{align*}
\mathcal P_\varphi(x,\d y)&=e^{\overline\varphi(x)}q\left(y-T(x)\right)\,\d\rho(y).
\end{align*}
Continuity of translations in $L^1(\mathbb T^d,\rho)$ gives~\ref{it:hypH-i}. The dense orbit of $T$, together with the positivity of $q$ on $B_\varepsilon(0)$, allows the random perturbation to follow finite $\varepsilon$-pseudo-orbits between any point and any non-empty open set. Hence, ~\ref{it:hypH-ii} holds and $Z=\varnothing$. Finally, connectedness of $\mathbb T^d$ and
\cite[Theorem~B]{CastroLambOliconMendezRasmussen2024}
give~\ref{it:hypH-iii}. We mention that if we additionally that $q\in L^\infty(\mathbb T^d,\rho)$
and that $T$ is nonsingular with respect to $\rho$, meaning that $\rho(T^{-1}(A))=0$ whenever $\rho(A)=0$. If, in addition, $q\in L^\infty(\mathbb T^d,\rho)$ and $T$ is nonsingular with respect to $\rho$, meaning that $\rho(T^{-1}(A))=0$ whenever $\rho(A)=0$, then it is possible to conclude that the quasi-stationary measure $\mu$ of $X_n^\varphi$ has a continuous, strictly positive density with respect to $\rho$.
\end{example}
\begin{example}[Bounded random perturbations on Riemannian manifolds]

Let $M$ be a compact connected Riemannian manifold and $\rho$ be a normalised volume form on $M$. Let $T:M\to M$ be a continuous map with a dense forward orbit. For $\varepsilon>0$, let
$q_\varepsilon(z,y)$ be a jointly continuous Markov density satisfying
$$
q_\varepsilon(z,y)>0
\ \text{if }d_M(z,y)<\varepsilon
\quad \text{and}\quad 
q_\varepsilon(z,y)=0
\ \text{if }d_M(z,y)>\varepsilon.
$$
Such a density may be constructed by pushing forward a smooth radial density on the tangent ball by the exponential map.

Let $Y_n$ be the Markov chain describing the random perturbation of $T$ via the transition density $q_\varepsilon(T(x),y)$. More precisely, given $Y_n=x$, one first applies the deterministic map $T$ and then chooses $Y_{n+1}$ randomly in an $\varepsilon$-neighbourhood of $T(x)$ according to the density $q_\varepsilon(T(x),\cdot)$. Thus, $$ \mathbb P\left( Y_{n+1}\in\d y \mid Y_n=x \right) = q_\varepsilon\left( T(x),y \right) \,\d\rho(y). $$
Let $(\xi_n)_{n\geq0}$ be independent of $Y_n$ and uniformly distributed on $[0,1]$. We define
$$
X_{n+1}:=\begin{cases}Y_{n+1},&X_n\in M\quad \text{and}\quad \xi_n\leq e^{\overline\varphi(X_n)},\\
\partial,
&
\text{otherwise}.
\end{cases}
$$
The killed transition kernel of this chain is
\begin{align*}
\mathcal P_\varphi(x,\d y)
&=
e^{\overline\varphi(x)}
q_\varepsilon\left(
T(x),y
\right)
\,\d\rho(y).
\end{align*}
Joint continuity of $q_\varepsilon$ gives~\ref{it:hypH-i}. The dense orbit of $T$ and the positivity of the density on geodesic balls imply~\ref{it:hypH-ii}. We have $Z=\varnothing$, and connectedness of $M$, together with
\cite[Theorem~B]{CastroLambOliconMendezRasmussen2024},
gives~\ref{it:hypH-iii}.

\end{example}

\begin{example}[Gaussian random perturbations with hard killing]

Let $\widetilde M$ be a connected Riemannian manifold, let
$\mathcal H\subset\widetilde M$ be a non-empty open set so that $M:=\widetilde M\setminus\mathcal H$ is a compact manifold with boundary and $M=\overline{\operatorname{int}(M)}$. Let $T:\widetilde M\to\widetilde M$ be a continuous map. Denote by $p_\sigma(z,y)$ the heat kernel on
$\widetilde M$ at some fixed time $\sigma>0$ (see~\cite[Ch.~4]{Hsu2002}).

Consider the Markov process $Y_n$ on $\widetilde{M}$ such that
$$
\mathbb P\left(Y_{n+1}\in\d y \mid Y_n=x \right)
=
p_\sigma\left(T(x),y\right) \d\operatorname{vol}_{\widetilde M}(y)\ \text{for any }n\in\mathbb N_0.
$$
Define the absorbing Markov process $X_n$ as
$$
X_{n+1}:=
\begin{cases}
Y_{n+1},&Y_{n+1}\in M,\\
\partial,&\text{otherwise}.
\end{cases}
$$
Thus the process is absorbed whenever its noisy image belongs to
$\mathcal H$, often called a hole. Its killed transition kernel on $M$ is $\mathcal P(x,\d y)=
p_\sigma \left(T(x),y\right) \mathrm{vol}_{\widetilde{M}}(M)\,\d\rho(y),$ where $\rho$ denotes the normalised restriction of
$\operatorname{vol}_{\widetilde M}$ to $M$.

The heat kernel is continuous and strictly positive. Hence,
\ref{it:hypH-i} holds and, for every non-empty relatively open set
$U\subset M$, $\mathcal P(x,U)>0$ for every $x\in M$. Thus,~\ref{it:hypH-ii} holds with one iterate,
$Z=\varnothing$, and strict positivity rules out any non-trivial cyclic decomposition. Therefore,~\ref{it:hypH-iii} also holds. Notice that no transitivity assumption on $T$ is required in this example.
\end{example}

\begin{example}[Killed elliptic diffusions in discrete time]

Let $D$ be a bounded connected domain with smooth boundary in a
Riemannian manifold $\widetilde M$, and set
$
M:=\overline D.
$
Consider the diffusion $(Y_t)_{t\geq0}$ in $D$ with generator 
$$\mathcal L f=b\cdot\nabla f+\frac{1}{2}\operatorname{tr}\left(a\nabla^2f\right),
$$
where both $a$ and $b$ are smooth, and $a$ is uniformly elliptic.  Let $\rho$ be a normalised volume form on $M$. Let $
\tau_D:=\inf\left\{t\geq0:Y_t\notin D\right\}$ be the first exit time from $D$. Fix $\Delta>0$ and define the
discrete-time absorbing Markov chain
$$
X_n
:=
\begin{cases}
Y_{n\Delta},&n\Delta<\tau_D,\\
\partial,&n\Delta\geq\tau_D.
\end{cases}
$$
Thus the process is absorbed as soon as the diffusion exits $D$,
including when the exit occurs between two consecutive observation
times.

Let $p_D(t,x,y)$ denote the Dirichlet heat kernel associated with
$\mathcal L$. The killed transition kernel of $(X_n)_{n \geq 0}$ is $\mathcal P(x,\d y)= p_D(\Delta,x,y)\,\d\rho(y),$ for $x\in D$, while $\mathcal P(x,M)=0$ for $x\in\partial D$ (see
\cite[\S5.2, in particular equation~(5.2.1) and
Theorem~5.2.8]{Stroock2008}).

For every fixed $\Delta>0$, the Dirichlet heat kernel is continuous on
$\overline D\times\overline D$ and satisfies $
p_D(\Delta,x,y)>0$
for every $x,y\in D$. Consequently, the map $x\to p_D(\Delta,x,\cdot)$
is continuous from $\overline D$ to $L^1(D,\rho)$, and hence
\ref{it:hypH-i} holds. Moreover, $
Z=\partial D.$ If $x\in D$ and $U\subset D$ is a non-empty open set, then
$$
\mathcal P(x,U)=\int_U p_D(\Delta,x,y)\,\d\rho(y)>0.$$
Thus,~\ref{it:hypH-ii} holds with one iterate. The strict positivity
of the kernel on $D\times D$ also excludes any non-trivial cyclic
decomposition, and therefore~\ref{it:hypH-iii} holds.
\end{example}

\begin{example}[A cubic map with bounded noise and hard killing]
We consider the setting of~\cite[Example~2.6]{CastroLambOliconMendezRasmussen2024}. Let $\varepsilon>2/(3\sqrt 3)$, let $r_\varepsilon>1$ be the unique solution of $r_\varepsilon^3-r_\varepsilon=\varepsilon$, and set $M=[-r_\varepsilon,r_\varepsilon]$. Let $(\omega_n)_{n\geq0}$ be independent and uniformly distributed random variables on $[-\varepsilon,\varepsilon]$. Define
$$
X_{n+1}=\begin{cases}
X_n^3+\omega_n,&X_n^3+\omega_n\in M,\\
\partial,&X_n^3+\omega_n\notin M,
\end{cases}
$$
where $\partial$ is absorbing. The killed transition kernel is
\begin{align*}
\mathcal P(x,\d y)&=\frac{1}{2\varepsilon}\mathbbm 1_M(y)\mathbbm 1_{[-\varepsilon,\varepsilon]}(y-x^3)\,\d y.
\end{align*}
Continuity of translations in $L^1(\mathbb R)$ gives~\ref{it:hypH-i}. Moreover, $Z=\{-r_\varepsilon,r_\varepsilon\}$, since $r_\varepsilon^3-\varepsilon=r_\varepsilon$ and $(-r_\varepsilon)^3+\varepsilon=-r_\varepsilon$. The monotonicity argument of~\cite[Example~2.6]{CastroLambOliconMendezRasmussen2024} gives irreducibility on $(-r_\varepsilon,r_\varepsilon)$, and hence~\ref{it:hypH-ii}. Finally, choose $\delta>0$ such that $\delta+\delta^3<\varepsilon$. Then the transition density is strictly positive on $(-\delta,\delta)\times(-\delta,\delta)$, which, together with irreducibility, yields~\ref{it:hypH-iii}.
\end{example}

\section{Spectral framework, the \texorpdfstring{$Q$}{Q}-process and refined quasi-ergodic limits}

This section develops the spectral and probabilistic tools used throughout the paper. We first recall the spectral decomposition of the killed transition operator and derive its consequences on $L^2(M,\mu)$, together with a uniform estimate for transitions into sets of small $\rho$-measure. We then introduce the $Q$-process, which describes the original chain conditioned on survival for an increasingly long time, and recall the transfer principle connecting convergence under the $Q$-process with convergence under $\mathbb P_x(\cdot\mid\tau>n)$. Finally, we combine these ingredients to prove the refined quasi-ergodic limits stated in Theorem~\ref{thm:refined-conditional-lln}.

To reduce notation, we often make use of the following symbols.
\begin{notation}
For $x\in M\setminus Z$ and $n\geq1$, write
$$
\mathbb P_{x,n}(A):=\mathbb P_x(A\mid\tau>n)=\frac{\mathbb P_x(A\cap\{\tau>n\})}{\mathbb P_x(\tau>n)}
$$
for $A\in\mathcal F$. For a measurable random variable $Y$, write
$$\mathbb E_{x,n}[Y]:=\mathbb E_x[Y\mid\tau>n]=\frac{\mathbb E_x[Y\mathbbm 1_{\{\tau>n\}}]}{\mathbb P_x(\tau>n)}$$
whenever the expectations are defined.
\end{notation}

\subsection{Spectral framework}\label{sec:spectral-framework}

In this section, we describe the spectral consequences of Hypothesis~\ref{hyp:H} that are leveraged throughout the paper. Before stating the spectral decomposition of the killed transition operator, we fix the notation for the function spaces on which the relevant operators act.

\begin{notation}
We use the following notation:
\begin{enumerate}
\item For $1\leq p<\infty$, $L^p(M,\rho)$ denotes the space of equivalence classes of measurable functions $f:M\to\mathbb R$ such that
$$\|f\|_{L^p(M,\rho)}:=\left[\int_M|f|^p\,\d\rho\right]^{1/p}<\infty.$$
\item $L^\infty(M,\rho)$ denotes the space of essentially bounded measurable functions, endowed with the essential supremum norm $\|f\|_{L^\infty(M,\rho)}
:=
\operatorname*{ess\,sup}_{x\in M}|f(x)|.$

\item $\mathcal C^0(M)$ denotes the Banach space of continuous functions $f:M\to\mathbb R$, endowed with the
$\|\cdot\|_\infty$-norm, $\|f\|_\infty:= \sup_{x\in M}|f(x)|.$
\end{enumerate}
\end{notation}

The following spectral decomposition result plays two main roles in this paper. Firstly, it yields conditional mixing of all orders, which is used to prove the Poisson limits for rare observations in Section~\ref{sec:conditional-poisson}, and secondly, it provides the perturbative framework for the conditional central limit theorem in Section~\ref{sec:l2-clt}. 

\begin{theorem}[{\cite[Theorems~A--C]{CastroLambOliconMendezRasmussen2024}}]
\label{thm:spectralgap}
Let $X_n$ be an absorbing Markov chain satisfying Hypothesis~\ref{hyp:H}. Then the operator
$$
\mathcal P:L^\infty(M,\rho)\to L^\infty(M,\rho),\quad 
(\mathcal Pf)(x):=\mathbb E_x\left[f\circ X_1\cdot \mathbbm 1_M\circ X_1\right],
$$
is compact, positive and linear. Moreover,
\begin{enumerate}[label = (\roman*)]
\item \label{it:sg-it1} The chain $X_n$ is strong Feller, meaning that $
\mathcal P L^\infty(M,\rho)\subset\mathcal C^0(M).$
\item \label{it:sg-it2} Let $\lambda=r(\mathcal P)>0$ be the spectral radius of $\mathcal P$. There exist a probability measure $\mu$ on $M$, a continuous function $\eta:M\to\mathbb R$ satisfying $\mu(\eta)=1$, and $0\leq\lambda_0<\lambda$ such that, for every $f\in L^\infty(M,\rho)$,
\begin{equation}
\mathcal P^nf=\lambda^n\eta\mu(f)+\mathcal O\left(\lambda_0^n\|f\|_\infty\right).
\label{2.1}
\end{equation}
Moreover, $
\mathcal P^*\mu:= \int_M \mathcal P(x,\cdot) \mu(\d x)=\lambda\mu,\ \mathcal P\eta=\lambda\eta$ and $\mu(\d x)=m(x)\rho(\d x),$ where $m\in L^1(M,\rho)$ is non-negative. Define
\begin{align}\widehat{\mathcal P}&:=\lambda^{-1}\mathcal P,\quad \Pi g:=\eta\mu(g),\quad R:=\widehat{\mathcal P}-\Pi.
\label{eq:op}
\end{align}
Then $\Pi^2=\Pi, \Pi R=R\Pi=0$ and $\widehat{\mathcal P}^{n}=\Pi+R^n$
for every $n\geq1$. Furthermore, there exists $C>0$ such that
$$\|R^ng\|_\infty\leq C\left(\frac{\lambda_0}{\lambda}\right)^n\|g\|_\infty$$
for every $g\in L^\infty(M,\rho)$ and every $n\geq1$.
\item $\eta(x) >0$ for every $x\in M\setminus Z$ and $\eta(x)=0$ for every $x\in Z$.
\item Define the probability measure $
\nu(\d x):=\eta(x)\mu(\d x)$ on $M$. Then, for every bounded measurable function $f:M\to\mathbb R$ and every $x\in M\setminus Z$,
$$\mathbb E_x\left[\left.\frac{1}{n}\sum_{i=0}^{n-1}f\circ X_i\right|\tau>n\right]\xrightarrow[]{n\to\infty}\nu(f).
$$ 
\end{enumerate}
\end{theorem}

We recall a useful operator-compactness result we shall immediately leverage to show that $\widehat{\mathcal P}$ and $R$ have good spectral properties.

\begin{theorem}[{\cite[Ch.~4, \S2, Theorem~2.9, p.~203]{BennettSharpley1988}}] \label{thm:compact-interpolation}
Let $(R,\mu)$ and $(S,\nu)$ be finite measure spaces. Suppose that $1\leq p_0,p_1, q_0, q_1\leq\infty,$
and let $T$ be a linear operator such that
$$
T:L^{p_0}(R,\mu)\to L^{q_0}(S,\nu)
\ \text{is bounded} \ \text{and}\ T:L^{p_1}(R,\mu)\to L^{q_1}(S,\nu)
\ \text{is compact}.
$$
For $0<\theta<1$, let $1 \leq p,q \leq \infty$ be defined by
$$\frac{1}{p}=\frac{1-\theta}{p_0}+\frac{\theta}{p_1}\quad \text{and}\quad\frac{1}{q}=\frac{1-\theta}{q_0}+\frac{\theta}{q_1}.$$
Then $T:L^p(R,\mu)\to L^q(S,\nu)$ is compact.
\end{theorem}

\begin{lemma}
\label{lem:l2-compactness}
The operators $\widehat{\mathcal P}$ and $R$ defined in \eqref{eq:op} are compact on both $L^\infty(M,\mu)$ and $L^2(M,\mu)$. Moreover, there exist $C>0$ and $\kappa\in(0,1)$ such that
\begin{align}
\|R^mf\|_{L^2(\mu)}&\leq C\kappa^m\|f\|_{L^2(\mu)}
\label{ref:l2-spectral-estimate}
\end{align}
for every $f\in L^2(M,\mu)$ and every $m\geq1$. Consequently, $R^m$ is compact on $L^2(M,\mu)$ for every $m\geq1$.
\end{lemma}
\begin{proof}
We first prove that $R$ is bounded on $L^1(M,\mu)$. Since $\widehat{\mathcal P}$ is positive and $\widehat{\mathcal P}^{*}\mu=\mu$,
$$
\|\widehat{\mathcal P}f\|_{L^1(\mu)}\leq\int_M\widehat{\mathcal P}|f|\,\d\mu=\int_M|f|\,\d\mu=\|f\|_{L^1(\mu)}.$$
Moreover, since $\Pi f=\eta\mu(f)$, $\eta\geq0$ and $\mu(\eta)=1$, $
\|\Pi f\|_{L^1(\mu)}=|\mu(f)|\mu(\eta)\leq\|f\|_{L^1(\mu)}.$ Therefore,
$\|Rf\|_{L^1(\mu)}\leq 2\|f\|_{L^1(\mu)}.$
The same argument, using $R^m=\widehat{\mathcal P}^{m}-\Pi$, gives
\begin{align}
\|R^mf\|_{L^1(\mu)}&\leq 2\|f\|_{L^1(\mu)}\ \text{for every }m\geq1.
\label{ref:l1-bound-Rm}
\end{align}

We next prove compactness on $L^2(M,\mu)$. We first regard $\widehat{\mathcal P}$ as an operator on $L^\infty(M,\mu)$. This operator is well defined. Indeed, if $f=g$ $\mu$-almost everywhere, then positivity and $\widehat{\mathcal P}^{*}\mu=\mu$ give
$$
\int_M\left|\widehat{\mathcal P}f-\widehat{\mathcal P}g\right|\,\d\mu\leq\int_M\widehat{\mathcal P}|f-g|\,\d\mu=\int_M|f-g|\,\d\mu=0.
$$
Hence, $\widehat{\mathcal P}f=\widehat{\mathcal P}g$ $\mu$-almost everywhere.

Let $(f_n)$ be a bounded sequence in $L^\infty(M,\mu)$. By redefining each $f_n$ on a $\mu$-null set, we may choose bounded measurable representatives $\widetilde f_n$ satisfying $
\|\widetilde f_n\|_{L^\infty(\rho)}\leq\|f_n\|_{L^\infty(\mu)}.$
Thus, $(\widetilde f_n)$ is bounded in $L^\infty(M,\rho)$. Since $\widehat{\mathcal P}$ is compact on $L^\infty(M,\rho)$, there exists a subsequence $(\widetilde f_{n_k})$ such that $
\widehat{\mathcal P}\widetilde f_{n_k}$ converges in $L^\infty(M,\rho)$. The strong Feller property ensures that these functions have continuous representatives. Since $\mu\ll\rho$, $
\|h\|_{L^\infty(\mu)}\leq\|h\|_{L^\infty(\rho)},$ and therefore the same subsequence converges in $L^\infty(M,\mu)$. Consequently, $\widehat{\mathcal P}$ is compact on $L^\infty(M,\mu)$. Since $\Pi$ has rank one, $
R=\widehat{\mathcal P}-\Pi$ is also compact on $L^\infty(M,\mu)$.

The preceding estimates show that $\widehat{\mathcal P}$ and $R$ are bounded on $L^1(M,\mu)$. Applying Theorem~\ref{thm:compact-interpolation} with $
p_0=q_0=1,\ p_1=q_1=\infty$ and $\theta=1/2,$ we conclude that both $\widehat{\mathcal P}$ and $R$ are compact on $L^2(M,\mu)$.

It remains to prove \eqref{ref:l2-spectral-estimate}. Let $f\in L^\infty(M,\mu)$ and choose a bounded measurable representative $\widetilde f$ such that $\|\widetilde f\|_{L^\infty(\rho)}\leq\|f\|_{L^\infty(\mu)}.$ The spectral decomposition on $L^\infty(M,\rho)$ gives constants $C_\infty>0$ and $\vartheta\in(0,1)$ such that
$$\|R^m\widetilde f\|_{L^\infty(\rho)}\leq C_\infty\vartheta^m\|\widetilde f\|_{L^\infty(\rho)}
$$
for every $m\geq1$. Since $\mu\ll\rho$, it follows that
$$\|R^mf\|_{L^\infty(\mu)}\leq C_\infty\vartheta^m\|f\|_{L^\infty(\mu)}.$$
Combining this estimate with \eqref{ref:l1-bound-Rm} and applying the Riesz--Thorin interpolation theorem~\cite[\S1.1]{BerghLofstrom1976}, we obtain
$$
\|R^mf\|_{L^2(\mu)}\leq\sqrt{2C_\infty}\vartheta^{m/2}\|f\|_{L^2(\mu)}.
$$
Thus, \eqref{ref:l2-spectral-estimate} holds with $
C:=\sqrt{2C_\infty}$ and $\kappa:=\sqrt{\vartheta}\in(0,1).$ Finally, since $R$ is compact and $R^{m-1}$ is bounded on $L^2(M,\mu)$, we have that $R^m = R\circ R^{m-1}$ is compact on $L^2(M,\mu)$ for every $m\geq1$.
\end{proof}

We finish the section with a useful lemma regarding shrinking target sets.

\begin{lemma}
\label{lem:uniform-small-targets}
Assume condition~\ref{it:hypH-i}. If $(U_n)_{n\geq1}$ is a sequence of Borel sets satisfying $\rho(U_n)\to0$, then $
\sup_{z\in M}\mathcal P(z,U_n)\xrightarrow[]{n\to\infty}0.$
Under Hypothesis~\ref{hyp:H}, the same conclusion holds whenever $\nu(U_n)\to0$ and $U_n\subset M\setminus Z$ for all sufficiently large $n$.
\end{lemma}

\begin{proof}
By condition~\ref{it:hypH-i}, the map $z\mapsto\kappa(z,\cdot)$ is continuous from the compact space $M$ into $L^1(M,\rho)$. Fix $\varepsilon>0$. There exist $z_1,\ldots,z_N\in M$ such that, for every $z\in M$, some $j\in\{1,\ldots,N\}$ satisfies
$\|\kappa(z,\cdot)-\kappa(z_j,\cdot)\|_{L^1(\rho)}
<\e /2.$ Consequently, for every Borel set $A\subset M$,
\begin{align}
\sup_{z\in M}\mathcal P(z,A)&\leq\frac{\varepsilon}{2}+\max_{1\leq j\leq N}\mathcal P(z_j,A).
\label{eq:uniform-target-net}
\end{align}
If $\rho(U_n)\xrightarrow[]{n\to\infty}0$, absolute continuity of the integral gives $\mathcal P(z_j,U_n)\xrightarrow[]{n\to\infty}0$ for every $j\in \{1, \dots, N\}$. Equation~\eqref{eq:uniform-target-net}, followed by $\varepsilon\to 0$, proves the first assertion.

For the second part of the lemma, let $A\subset M\setminus Z$ be Borel with $\nu(A)=0$. Since $\nu=\eta\mu$ and $\eta>0$ on $M\setminus Z$, we have $\mu(A)=0$. Quasi-stationarity gives
$$
\int_M\mathcal P(z,A)\,\d\mu(z)=\lambda\mu(A)=0.
$$
By condition~\ref{it:hypH-i}, the function $z\mapsto\mathcal P(z,A)$ is continuous and non-negative. It therefore vanishes on $\operatorname{supp}\mu$, which contains $M\setminus Z$ by Theorem~\ref{thm:Castro}. For $z\in Z$,
$$
0=\lambda\eta(z)=\mathcal P\eta(z) =\int_{M\setminus Z}\eta(y)\,\mathcal P(z,\d y),
$$
so $\mathcal P(z,M\setminus Z)=0$. Consequently, $
\mathcal P(z,\cdot)|_{M\setminus Z}\ll\nu$ for every $z\in M.$ Since $\nu(U_n)\to0$ and $U_n\subset M\setminus Z$ eventually, absolute continuity of these finite measures gives
$$
\max_{1\leq j\leq N}\mathcal P(z_j,U_n)\xrightarrow[]{n\to\infty}0.
$$
Again, applying \eqref{eq:uniform-target-net} and letting $\varepsilon\to0$ concludes the proof.
\end{proof}

\subsection{The \texorpdfstring{$Q$}{Q}-process} \label{sec:q-process}

We continue by recalling the definition of a $Q$-process associated with an absorbing Markov chain. The following theorem gives its finite-dimensional distributions, its transition operator and its invariant probability measure.
\begin{definition}[$Q$-process]
Let $\mathcal F_n := \sigma (X_0,X_1,\ldots,X_n)\subset \mathcal F$. A family of probability measures $(\mathbb Q_x)_{x\in M\setminus Z}$ on
$(\Omega,\mathcal F)$ is called a \emph{$Q$-process} if:
\begin{enumerate}
\item for every $x\in M\setminus Z$, every $n\in\mathbb N_0$ and every $A\in\mathcal F_n$,
\begin{align*}
\mathbb Q_x(A):=\lim_{t\to\infty}\mathbb P_x(A\mid\tau>t);
\end{align*}
\item the tuple $
\left(\Omega,(\mathcal F_n)_{n\in\mathbb N},X_n,(\mathcal Q^n)_{n\in\mathbb N},(\mathbb Q_x)_{x\in M\setminus Z}\right)$
is a Markov process, where
$$
\mathcal  Q^n(x,B):=\mathbb Q_x[X_n\in B]
$$
for every $n\in\mathbb N$, $x\in M\setminus Z$ and $B\in\mathcal B(M)$.
\end{enumerate}
Note that, by definition, a $Q$-process is unique. 
Given a probability measure $\nu$ on $M$, define
$$
\mathbb Q_\nu(\d x):=\int_{M\setminus Z}\mathbb Q_y(\d x)\,\nu(\d y).
$$
\end{definition}

The following theorem establishes the existence of the $Q$-process under Hypothesis~\ref{hyp:H}.
\begin{theorem}[{\cite[{Proposition~2.6}]{Castroetall}}] \label{thm:Qprocess}
Let $X_n$ be an absorbing Markov chain satisfying Hypothesis~\ref{hyp:H}. Then, for every $x\in M\setminus Z$, it holds that:
\begin{enumerate}[label = (\roman*)]
\item \label{it:Qproc-i} for every $n\in\mathbb N_0$ and every $A\in\mathcal F_n$,
$$
\lim_{t\to\infty}\mathbb P_x\left(A\mid\tau>t\right)=\frac{1}{\lambda^n\eta(x)}\mathbb E_x\left[\mathbbm 1_A\eta(X_n)\right];
$$
\item \label{it:Qproc-ii} for every $n\in\mathbb N_0$ and every $B\in\mathcal B(M)$,
$$
\mathcal Q^n(x,B)=\frac{1}{\lambda^n\eta(x)}\mathcal P^n\left(\eta\mathbbm 1_B\right)(x);
$$
\item \label{it:Qproc-iii} the quasi-ergodic measure
$
\nu(\d x)=\eta(x)\mu(\d x)
$
is the unique stationary probability measure of the $Q$-process.
In particular, $\mathbb Q_\nu$ is stationary and ergodic.
\end{enumerate}
\end{theorem}

To apply Theorem~\ref{thm:Qprocess} in our setting, we realise $X_n$ on the canonical path space $\Omega=\widehat M^{\mathbb N_0}$, equipped with the coordinate $\sigma$-algebra and the natural filtration. We may do so without loss of generality, since our results depend only on the law of the process.

\begin{remark} For every $x\in M\setminus Z$, Theorem~\ref{thm:Qprocess} and $\eta|_Z=0$ give
$$\mathbb Q_x(\tau>n)=1\quad \text{and}\quad\mathcal Q^n(x,Z)=0\ \text{for every }n\geq0.
$$
Thus the $Q$-process remains in $M\setminus Z$ forever:
$$
\mathbb Q_x\left(X_n\in M\setminus Z
\text{ for every }n\geq0\right)=1.
$$
\end{remark}

\section{Conditional mixing estimates}
\label{sec:mixing}

This section establishes the independence estimates used in the proof
of Theorem~\ref{thm:ppp}, which appears in Section~\ref{sec:ppp}. Lemma~\ref{lem:mix} treats well-separated
observations, while Lemma~\ref{lem:cluster-mixing} controls separated
clusters and is used to show that clustered configurations give a
negligible contribution to the factorial moments.

\subsection{Conditional mixing of all orders}

The proof of the point-process limit requires estimates for joint conditional observations at several separated times. The spectral decomposition in Theorem~\ref{thm:spectralgap} gives the following conditional mixing estimate of all orders.
\begin{lemma}\label{lem:mix}
Fix $x\in M\setminus Z$, $k> 1$, and choose $\vartheta\in(\lambda_0/\lambda,1)$. Let
$g_1,\ldots,g_k\in L^\infty(M,\rho)$ and $n_1,\ldots,n_k\in\mathbb N$. Set $
s_j:=n_1+\cdots+n_j$ and $
m:=\min_{1\leq j\leq k}n_j.$ Then
\begin{align}\mathbb E_x\left[\left.\prod_{j=1}^k g_j\circ X_{s_j}\,\right|\,\tau>s_k\right]=\left(\prod_{j=1}^{k-1}\nu(g_j)\right)\mu(g_k)+\mathcal O_{x,k}\left(
\vartheta^m\prod_{j=1}^k\|g_j\|_\infty
\right).\label{eq:problematic}
\end{align}
Moreover, if $n>s_k$ and $
m_n:=\min\{n_1,\ldots,n_k,n-s_k\},$
then
\begin{align}
\mathbb E_x\left[\left.\prod_{j=1}^k g_j\circ X_{s_j}\,\right|\,\tau>n\right]=\prod_{j=1}^k\nu(g_j)+\mathcal O_{x,k}\left(\vartheta^{m_n}\prod_{j=1}^k\|g_j\|_\infty\right).\label{eq:mixingallorders}\end{align}
\end{lemma}

\begin{proof}
Define $
\widehat{\mathcal P}^n:=\lambda^{-n}\mathcal P^n$
and let $\Pi$ be the rank-one operator $
\Pi g:=\eta\mu(g).$
By the spectral decomposition in \eqref{2.1} and \eqref{eq:op}, there are operators $R$ such that
\begin{align}
\widehat{\mathcal P}^ng=\Pi g+R^ng,\quad \Pi \circ R = R\circ \Pi = 0\quad \text{and}\quad \|R^ng\|_\infty
\leq C\vartheta^n\|g\|_\infty.
\label{ref:normalised-spectral-decomposition}
\end{align}
It follows from \eqref{ref:normalised-spectral-decomposition} that $
\sup_{n\geq1}
\|\widehat{\mathcal P}^n\|_{L^\infty\to L^\infty}
<\infty.$
Applying the same spectral decomposition to $\mathbbm 1_M$, we obtain, for every $\ell\geq1$,
\begin{align}
\widehat{\mathcal P}^\ell \mathbbm 1_M(x) = \frac{1}{\lambda^{\ell}}\mathcal P^\ell\mathbbm 1_M(x)
=
\eta(x)+\mathcal O_x(\vartheta^\ell).
\label{ref:normalised-survival-probability}
\end{align}
By the Markov property, and since $s_k = n_1 + \cdots + n_k$,
\begin{align}\frac{1}{\lambda^{s_k}}\mathbb E_x\left[
\prod_{j=1}^k g_j\circ X_{s_j}
\mathbbm 1_{\{\tau>s_k\}}
\right] = \widehat{\mathcal P}^{n_1}\left(
g_1\widehat{\mathcal P}^{n_2}\left(g_2\cdots
\widehat{\mathcal P}^{n_{k-1}}\left(
g_{k-1}\widehat{\mathcal P}^{n_k}g_k
\right)\right)\right)(x).\label{eq:markovbernat}
\end{align}
Since $\widehat{\mathcal P}=\Pi+R$, $\Pi^2=\Pi$ and $\Pi R=R\Pi=0$, we have $
\widehat{\mathcal P}^n=(\Pi+R)^n=\Pi+R^n.$
Recall that, by definition, $
\Pi g_k=\eta\mu(g_k).$ Therefore,
$$
\begin{aligned}
\Pi\left(g_{k-1}\Pi g_k\right)=\Pi\left(g_{k-1}\eta\mu(g_k)\right)=\eta\mu(g_{k-1}\eta)\mu(g_k)= \eta\nu(g_{k-1})\mu(g_k).
\end{aligned}
$$
Repeating this calculation gives
\begin{align}
\Pi\left(
g_1\Pi\left(g_2\cdots
\Pi\left(g_{k-1}\Pi g_k\right)\right)\right)=\eta\left(\prod_{j=1}^{k-1}\nu(g_j)\right)\mu(g_k),\label{eq:average}
\end{align}
which corresponds to the factor in~\eqref{eq:markovbernat} containing only terms in $\Pi$.
To obtain the remainder in~\eqref{eq:problematic} we estimate the difference between \eqref{eq:markovbernat} and \eqref{eq:average}.
In order to control and keep track of the several crossed terms with $\Pi$, $R$ and $g_j$, for each for $j=1,\ldots,k$, we introduce the multiplication operator $M_j h = g_j h$. Observe that $\|M_j\|_{L^\infty\to L^\infty}=\|g_j\|_\infty$.

Since $\widehat{\mathcal P}^{n_j}-\Pi=R^{n_j}$, the telescoping identity gives
\begin{equation}
\prod_{\ell=1}^k
  \left(\widehat{\mathcal P}^{n_\ell}M_\ell\right)-\prod_{\ell=1}^k \left(\Pi M_\ell\right)=\sum_{j=1}^k\left[\prod_{\ell=1}^{j-1}\left(\Pi M_\ell\right)\right]R^{n_j}M_j\left[\prod_{\ell=j+1}^k
\left(\widehat{\mathcal P}^{n_\ell}M_\ell\right)\right].\label{eq:telescomping}
\end{equation}
Thus each of the $k$ summands contains exactly one remainder $R^{n_j}$,
with projections $\Pi$ acting to its left and powers of $\widehat{\mathcal P}$ to its right.

Setting $
K:=
\max\left\{1,\|\Pi\|_{L^\infty\to L^\infty},\sup_{n\geq1}\|\widehat{\mathcal P}^n\|_{L^\infty\to L^\infty}\right\}<\infty,$ \eqref{ref:normalised-spectral-decomposition}, \eqref{eq:average} and \eqref{eq:telescomping} we obtain
\begin{align*}
&\left|\widehat{\mathcal P}^{n_1}\left(
g_1\widehat{\mathcal P}^{n_2}\left(g_2\cdots\widehat{\mathcal P}^{n_{k-1}}\left(g_{k-1}\widehat{\mathcal P}^{n_k}g_k
\right)\right)\right)(x) -  \eta(x)\left(\prod_{j=1}^{k-1}\nu(g_j)\right)\mu(g_k)\right|\\
&\leq \left\|\prod_{\ell=1}^k\left(\widehat{\mathcal P}^{n_\ell}M_\ell\right)-\prod_{\ell=1}^k\left(\Pi M_\ell\right)\right\|_{L^\infty\to L^\infty} \leq\sum_{j=1}^k K^{j-1}C\vartheta^{n_j}K^{k-j}\prod_{\ell=1}^k\|g_\ell\|_\infty
\\
&=
CK^{k-1}\left(\sum_{j=1}^k\vartheta^{n_j}\right)\prod_{\ell=1}^k\|g_\ell\|_\infty\leq kCK^{k-1}\vartheta^m \prod_{\ell=1}^k\|g_\ell\|_\infty =  \mathcal O_{k}\left(
\vartheta^m\prod_{j=1}^k\|g_j\|_\infty\right),
\end{align*}
where the last inequality uses $n_j\geq \min_{1\leq j \leq k} n_j =m $ for every $j$.
From the above bound and \eqref{eq:markovbernat} we obtain that 
\begin{align}
&\frac{1}{\lambda^{s_k}}\mathbb E_x\left[\prod_{j=1}^k g_j\circ X_{s_j}\mathbbm 1_{\{\tau>s_k\}}\right]=\eta(x)\left(\prod_{j=1}^{k-1}\nu(g_j)\right)\mu(g_k)+\mathcal O_k\left(\vartheta^m\prod_{j=1}^k\|g_j\|_\infty\right).
\label{ref:last-observation-numerator}
\end{align}
By the definition of conditional expectation,
$$
\begin{aligned}
\mathbb E_x\left[\left.\prod_{j=1}^k g_j\circ X_{s_j}\,\right|\,\tau>s_k\right]=\frac{\lambda^{-s_k}\mathbb E_x\left[\prod_{j=1}^k g_j\circ X_{s_j}\mathbbm 1_{\{\tau>s_k\}}\right]
}{
\lambda^{-s_k}\mathbb P_x(\tau>s_k)
}.
\end{aligned}
$$
Since $\eta(x)>0$, dividing \eqref{ref:last-observation-numerator} by
\eqref{ref:normalised-survival-probability}, with $\ell=s_k$, gives
$$
\mathbb E_x\left[\left.\prod_{j=1}^k g_j\circ X_{s_j}\,\right|\,\tau>s_k\right]=\left(\prod_{j=1}^{k-1}\nu(g_j)\right)\mu(g_k)+\mathcal O_{x,k}\left(\vartheta^m\prod_{j=1}^k\|g_j\|_\infty\right),
$$
which proves \eqref{eq:problematic}. Equation \eqref{eq:mixingallorders} follows from the above identity by setting $k = k'+1$ and $g_{k'+1} = \mathbbm 1_M$.
\end{proof}

\subsection{Cluster mixing}
Conditional mixing of all orders controls observations separated by long gaps, but it does not directly apply to groups of observations occurring close together. We treat each such group as a cluster and describe its internal distribution using the stationary $Q$-process. The following lemma shows that clusters separated by long gaps become asymptotically independent.

\begin{lemma}
\label{lem:cluster-mixing}
Fix $x\in M\setminus Z$ and choose $\vartheta\in(\lambda_0/\lambda,1)$. Let $
0<t_1<\cdots<t_r<N$
and let $D_1,\ldots,D_r\subset M$ be measurable sets. Suppose that the selected times are divided into $c$ consecutive clusters determined by indices
$$
0=p_0<p_1<\cdots<p_c=r.
$$
The $\ell$-th cluster is $
\mathcal C_\ell:=\{t_{p_{\ell-1}+1},\ldots,t_{p_\ell}\},\  1\leq\ell\leq c.$
Assume that
$$
t_1\geq q,\ t_{p_\ell+1}-t_{p_\ell}\geq q\ \text{for }1\leq\ell<c,\ N-t_r\geq q,
$$
i.e.~the gap from time zero to the first cluster, every gap between consecutive clusters and the gap from the last cluster to the conditioning time $N$ are all at least of size $q$.
For every $\ell\in\{1,\ldots,c\}$, define the stationary $Q$-probability of the events belonging to the $\ell$-th cluster by
$$
\pi_\ell:=\mathbb Q_\nu\left(X_{t_a-\min\mathcal C_\ell}\in D_a\text{ for every }a\text{ such that }t_a\in\mathcal C_\ell
\right).
$$
Then
\begin{align*}
\mathbb P_x\left(X_{t_a}\in D_a\text{ for every }1\leq a\leq r\mid\tau>N\right)&=\prod_{\ell=1}^c\pi_\ell+\mathcal O_{x,r}(\vartheta^q).
\end{align*}
The error is uniform in the sets $D_1,\ldots,D_r$, the gaps between observations inside each cluster and the positions of the clusters.
\end{lemma}

\begin{proof}
Recall that $
\widehat{\mathcal P}^{m}:=\lambda^{-m}\mathcal P^m$ and $ \Pi g:=\eta\mu(g).$
By \eqref{ref:normalised-spectral-decomposition}, 
$$
\left\|\widehat{\mathcal P}^{m}-\Pi\right\|_{L^\infty\to L^\infty}\leq C\vartheta^m\quad \text{and}\quad \sup_{m\geq0}\left\|\widehat{\mathcal P}^{m}\right\|_{L^\infty\to L^\infty}<\infty.
$$
For each $\ell\in\{1,\ldots,c\}$, define the operator associated with the $\ell$-th block by:
\begin{itemize}
    \item $
\mathcal B_\ell g:=\mathbbm 1_{D_{p_{\ell-1}+1}}\widehat{\mathcal P}^{\,t_{p_{\ell-1}+2}-t_{p_{\ell-1}+1}}\left(\mathbbm 1_{D_{p_{\ell-1}+2}}\cdots\widehat{\mathcal P}^{\,t_{p_\ell}-t_{p_\ell-1}}\left( \mathbbm 1_{D_{p_\ell}}g\right)\cdots\right)$ if $p_\ell-p_{\ell-1}\geq2;$  and
\item $
\mathcal B_\ell g:=\mathbbm 1_{D_{p_\ell}}g$ when  $p_\ell-p_{\ell-1}=1.$
\end{itemize}
We claim that
\begin{align}
\pi_\ell=\mu(\mathcal B_\ell\eta).
\label{ref:block-stationary-probability}
\end{align}
Indeed, the Markov property for the stationary $Q$-process gives
$$
\pi_\ell=\int_M\mathbbm 1_{D_{p_{\ell-1}+1}}\mathcal Q^{t_{p_{\ell-1}+2}-t_{p_{\ell-1}+1}}\left(\mathbbm 1_{D_{p_{\ell-1}+2}}\cdots\mathcal Q^{t_{p_\ell}-t_{p_\ell-1}}\left(\mathbbm 1_{D_{p_\ell}}\right)\right)\,\d\nu.
$$
Recall that, for every bounded measurable function $h$, $
\mathcal Q^m h=\frac{1}{\eta}\widehat{\mathcal P}^{m}(\eta h).$
Consequently, two consecutive applications of the $Q$-transition satisfy
$$
\mathcal Q^{m_1}\left(\mathbbm 1_D\mathcal Q^{m_2}h
\right)=\frac{1}{\eta}\widehat{\mathcal P}^{m_1}\left(\eta\mathbbm 1_D\frac{1}{\eta}\widehat{\mathcal P}^{m_2}(\eta h)\right)=\frac{1}{\eta}\widehat{\mathcal P}^{m_1}\left(\mathbbm 1_D\widehat{\mathcal P}^{m_2}(\eta h)\right).
$$
Applying the above identity successively to all the transitions inside the cluster gives
$$
\pi_\ell=\int_M\frac{1}{\eta}\mathcal B_\ell\eta\,\d\nu = \int_M \mathcal B_\ell \eta \d \mu = \mu(\mathcal B_\ell \eta).$$
If the cluster contains only one selected time, the same identity follows directly from
$$
\pi_\ell=\nu(D_{p_\ell})=\mu\left(\eta\mathbbm 1_{D_{p_\ell}}\right)=\mu(\mathcal B_\ell\eta).
$$

Define the gaps outside the blocks by
$$d_0:=t_1,\ d_\ell:=t_{p_\ell+1}-t_{p_\ell}\ \text{for }1\leq\ell<c,\ d_c:=N-t_r.
$$
By assumption, $d_\ell\geq q\ \text{for every }0\leq\ell\leq c.$
The Markov property gives
\begin{align}
&\mathbb P_x\left(X_{t_a}\in D_a
\text{ for every }1\leq a\leq r\mid\tau>N\right) =\frac{\widehat{\mathcal P}^{\,d_0}\mathcal B_1\widehat{\mathcal P}^{\,d_1}\mathcal B_2\cdots\widehat{\mathcal P}^{\,d_{c-1}}\mathcal B_c\widehat{\mathcal P}^{\,d_c}\mathbbm 1_M(x)
}{
\widehat{\mathcal P}^{N}\mathbbm 1_M(x)
}.
\label{ref:block-markov-representation}
\end{align}
We replace each of the $c+1$ operators $
\widehat{\mathcal P}^{\,d_0},
\ldots,
\widehat{\mathcal P}^{\,d_c}$
in \eqref{ref:block-markov-representation} by $\Pi$. Since $d_\ell\geq q$, each replacement produces an error bounded by $C\vartheta^q$. Moreover, the operators $\mathcal B_1,\ldots,\mathcal B_c$ are uniformly bounded because each block contains at most $r$ observations and $
\sup_{m\geq0}\|\widehat{\mathcal P}^{m}\|_{L^\infty\to L^\infty}<\infty.$ A telescoping expansion over the $c+1$ long gaps (as done in the proof of Lemma \ref{lem:mix}) therefore gives 
$$
\begin{aligned}
\left[\widehat{\mathcal P}^{\,d_0}\mathcal B_1\widehat{\mathcal P}^{\,d_1}\cdots\widehat{\mathcal P}^{\,d_{c-1}}\mathcal B_c\widehat{\mathcal P}^{\,d_c}\mathbbm 1_M\right](x)=\left[\Pi\mathcal B_1\Pi\cdots\Pi\mathcal B_c\Pi\mathbbm 1_M\right](x)+\mathcal O_{r}(\vartheta^q).
\end{aligned}
$$
Since $\mu(\mathbbm 1_M)=1$, we have $
\Pi\mathbbm 1_M=\eta.
$
Using \eqref{ref:block-stationary-probability} repeatedly, we obtain
$$
\Pi\mathcal B_1\Pi\cdots\Pi\mathcal B_c\Pi\mathbbm 1_M(x)=\eta(x)\prod_{\ell=1}^c\pi_\ell.
$$
Finally, \eqref{ref:normalised-spectral-decomposition} gives $
\widehat{\mathcal P}^{N}\mathbbm 1_M(x)=\eta(x)+\mathcal O_x(\vartheta^N).$
Since $\eta(x)>0$ and $N\geq q$, division in \eqref{ref:block-markov-representation} implies
$$\mathbb P_x \left(X_{t_a}\in D_a \text{ for every }1\leq a\leq r\mid\tau>N\right) = \prod_{\ell=1}^c\pi_\ell + \mathcal O_{x,r}(\vartheta^q),$$
which finishes the proof.
\end{proof}

\section{Poisson point-process convergence}\label{sec:ppp}

In this section, we prove Theorem~\ref{thm:ppp}. Recall that
$$
N_n:=\sum_{j=0}^{n-1}\delta_{\frac{f_\beta\circ X_j}{B_n}}\in \mathcal M_p((0,\infty]).
$$
We first record three standard results that will be used to pass from the convergence of mixed factorial moments to the convergence of random point measures. For a finite set $A$, we denote its cardinality by $|A|$.

\begin{theorem}[{\cite[Theorem~30.2]{Billingsley95}}] \label{thm:methodofmoments}
Let $Y,Y_1,Y_2,\ldots$ be real-valued random variables. Suppose that the distribution of $Y$ is uniquely determined by its moments, that $Y_n$ has moments of all orders for every $n\geq1$, and that
$\mathbb E\left[Y_n^r\right]\xrightarrow[]{n\to\infty} \mathbb E\left[Y^r\right]$ for every $r\in\mathbb N$. Then $Y_n\xrightarrow[d]{n\to\infty}Y.$
\end{theorem}

\begin{theorem} [{\cite[Theorem~29.4]{Billingsley95}}] 
\label{thm:cramerwold}
Let $Z_n$ and $Z$ be random vectors in $\mathbb R^k$. Then $Z_n\xrightarrow[d]{n\to\infty}Z$
if and only if, for every $u_1,\ldots,u_k\in\mathbb R$, $\sum_{j=1}^ku_jZ_{n,j}\xrightarrow[d]{n\to\infty}\sum_{j=1}^ku_jZ_j.$
\end{theorem}

\begin{theorem}[{\cite[Theorem~4.11]{Kallenberg2017}}] \label{thm:kallenberg-random-measures}
Let $N,N_1,N_2,\ldots$ be random point measures on $(0,\infty]$. Suppose that, for every $p\in\mathbb N$ and every collection of pairwise disjoint intervals
$J_\ell=(a_\ell,b_\ell],$
with $0<a_\ell<b_\ell\leq\infty$, we have
$$
\left(N_n(J_1),\ldots,N_n(J_p)\right)\xrightarrow[d]{n\to\infty}\left(N(J_1),\ldots,N(J_p)\right).
$$
Suppose also that $N(\partial J)=0\ \text{almost surely}$ for every interval $J=(a,b]$ with $0<a<b\leq\infty$. Then
$$
N_n\xrightarrow[d]{n\to\infty}N\ \text{in}\ \mathcal M_p((0,\infty]),
$$
endowed with the vague topology.
\end{theorem}

\begin{remark}
    Theorem~\ref{thm:kallenberg-random-measures} is a direct consequence of the implication
$\textnormal{(ii)}\Rightarrow\textnormal{(i)}$ in
\cite[Theorem~4.11]{Kallenberg2017}, applied to the dissecting semiring of intervals $\{\varnothing\}\cup\{(a,b]:0<a<b\leq\infty\}$.
\end{remark}
 
The proof of Theorem~\ref{thm:ppp} is based on the factorial-moment method, following the classical method of moments~\cite[Theorem~30.2]{Billingsley95}. To estimate the mixed factorial moments, we decompose the selected observation times into well-separated configurations and configurations containing clusters of nearby times, as in
\cite{DolgopyatFayadSixu2022,Auer25} (see also
\cite{AuerLiu2026,Auer2026Quenched}). In the present conditioned setting, the well-separated configurations are controlled by conditional mixing of all orders, whereas the contribution of each cluster is estimated under the stationary $Q$-process using the short-return estimate~\eqref{eq:shortreturn}.

\begin{proof}[Proof of Theorem~\ref{thm:ppp}]
Let $I_1,\ldots,I_k$ be pairwise disjoint intervals of the form $I_j=(a_j,b_j]\subset(0,\infty]$, with $a_j>0$, and use the convention $b_j^{-\alpha}=0$ when $b_j=\infty$.
Define $\mathbb P_{x,n}(\cdot):=\mathbb P_x(\cdot\mid\tau>n)$ and $a:=\min_{1\leq j\leq k}a_j$. For each $j\in\{1,\ldots,k\}$, set 
$$
A_{n,j}:=\{f_\beta/B_n\in I_j\}\quad \text{and} \quad U_n:= \{f_\beta > aB_n\}.
$$
Thus,
$$
N_n(I_j)=\sum_{i=0}^{n-1}\mathbbm 1_{A_{n,j}}(X_i)\quad \text{and}\quad A_{n,j}\subset U_n\ \text{for every }j\in\{1,\ldots,k\}.
$$ 

We divide the proof into seven steps.

\begin{step}[1] \label{step:ppp-1}
We show that 
\begin{enumerate}[label = (\roman*)]
    \item \label{it:ppp-1-i} $n\nu(A_{n,j})\to\upsilon_\alpha(I_j)$ for every $j\in\{1,\ldots,k\}$ as $n\to\infty$;
    \item \label{it:ppp-1-ii} $n\nu(U_n)\to a^{-\alpha}$ as $n\to\infty$; and
    \item \label{it:ppp-1-iii} $\sup_{z\in M}\mathcal P(z,U_n)\to0$ as $n\to\infty.$
\end{enumerate} 
\end{step}
\begin{stepproof}[Proof of Step~\ref{step:ppp-1}]
We show \ref{it:ppp-1-i}. Since $I_j=(a_j,b_j]$ and from Proposition \ref{prop:tail-asymptotics} we have that $n\nu(f_\beta>uB_n)\to u^{-\alpha}$ for every $u>0$, we have
$$
n\nu(A_{n,j})=n\nu(f_\beta>a_jB_n)-n\nu(f_\beta>b_jB_n)\xrightarrow[]{n\to\infty}a_j^{-\alpha}-b_j^{-\alpha}=\upsilon_\alpha(I_j).
$$
To obtain \ref{it:ppp-1-ii}, observe that
\begin{align*}
n\nu(U_n)&=n\nu(f_\beta>aB_n)\xrightarrow[]{n\to\infty}a^{-\alpha}.
\end{align*}
We now verify \ref{it:ppp-1-iii}.
Since $U_n$ decreases to the empty set, continuity from above gives
$\rho(U_n)\to0.$ Lemma~\ref{lem:uniform-small-targets} therefore yields
$$
\sup_{z\in M}\mathcal P(z,U_n)\xrightarrow[]{n\to\infty}0,
$$
as desired.
\end{stepproof}

\begin{step}[2] \label{step:ppp-2}
We show that there exists $C>0$ and $\vartheta\in(0,1)$ such that
$$
\sup_{y\in U_n}\mathcal Q^m(y,U_n)\leq\nu(U_n)+C\vartheta^{m-1}\sup_{z\in M}\mathcal P(z,U_n)
$$
for every $m\geq1$ and all sufficiently large $n$. Recall from \eqref{thm:Qprocess} that
$$
\mathcal Q^mg(y):=\frac{\mathcal P^m(\eta g)(y)}{\lambda^m\eta(y)}.
$$
\end{step}
\begin{stepproof}[Proof of Step~\ref{step:ppp-2}]
Fix $\vartheta\in(\lambda_0/\lambda,1)$. Recall that $\nu$ is stationary for $\mathcal Q$:
\begin{align}
\int_M\mathcal Qg\,d\nu&=\frac{1}{\lambda}\mu\bigl(\mathcal P(\eta g)\bigr)=\mu(\eta g)=\nu(g).\nonumber
\end{align}
Since $\eta(x_0)>0$, there exist a neighborhood $V$ of $x_0$ and $\eta_*>0$ such that $\eta(y)\geq\eta_*$ for every $y\in V$. For all sufficiently large $n$, $U_n\subset V$. Let $\varphi_n:=\mathcal P(\eta\mathbbm 1_{U_n})$. By Step~\ref{step:ppp-1} and the identity $\mathcal P^* \mu=\lambda\mu$, we have that
\begin{align*}
\|\varphi_n\|_\infty\leq \|\eta\|_\infty\sup_{z\in M}\mathcal P(z,U_n)\quad  \text{and}\quad\mu(\varphi_n) =\lambda\nu(U_n).
\end{align*}
For $m\geq2$, the spectral decomposition in \eqref{2.1} gives
\begin{align}
\mathcal P^m(\eta\mathbbm 1_{U_n})&=\mathcal P^{m-1}\varphi_n=\lambda^{m-1}\eta\mu(\varphi_n)+\mathcal O\bigl(\lambda^{m-1}\vartheta^{m-1}\|\varphi_n\|_\infty\bigr)\nonumber\\
&=\lambda^m\eta\nu(U_n)+\mathcal O\bigl(\lambda^{m-1}\vartheta^{m-1}\sup_{z\in M}\mathcal P(z,U_n)\bigr).\nonumber
\end{align}
Dividing by $\lambda^m\eta(y)$ and using $\eta(y)\geq\eta_*$ on $U_n$, we obtain
\begin{align}
\mathcal Q^m(y,U_n)=\nu(U_n)+\mathcal O\bigl(\vartheta^{m-1}\sup_{z\in M}\mathcal P(z,U_n)\bigr)\label{eq:shortreturn}
\end{align}
uniformly in $y\in U_n$. For $m=1$,
$$
\mathcal Q(y,U_n)=\frac{\mathcal P(\eta\mathbbm 1_{U_n})(y)}{\lambda\eta(y)}\leq\frac{\|\eta\|_\infty}{\lambda\eta_*}\sup_{z\in M}\mathcal P(z,U_n).
$$
Increasing the constant if necessary, we conclude that
\begin{align}
\sup_{y\in U_n}\mathcal Q^m(y,U_n)\leq\nu(U_n)+C\vartheta^{m-1}\sup_{z\in M}\mathcal P(z,U_n)\label{ref:q-short-return}
\end{align}
for every $m\geq1$, completing Step~\ref{step:ppp-2}.
\end{stepproof}
\begin{step}[3] \label{step:ppp-3}
Fix $r_1,\ldots,r_k\in\mathbb N_0$, set $r:=r_1+\cdots+r_k$, and choose $q_n=\lceil (\log n)^2\rceil$ hence ${n^r\vartheta^{q_n}\xrightarrow[]{n\to\infty}0}$. We show that
\begin{align}
\sum_{\substack{0\leq i<q_n\\\text{or }n-q_n<i<n}}\mathbb P_{x,n}(X_i\in U_n)\xrightarrow[]{n\to\infty}0.\label{eq:limite}
\end{align}
In particular, defining
$$
\widetilde N_n(I_j):=\sum_{i=q_n}^{n-q_n}\mathbbm 1_{A_{n,j}}(X_i),
$$
the fact that $A_{n,j}\subset U_n$ together with \eqref{eq:limite} yields
\begin{align}
\mathbb P_{x,n}\left(\bigl(N_n(I_1),\ldots,N_n(I_k)\bigr)\neq\bigl(\widetilde N_n(I_1),\ldots,\widetilde N_n(I_k)\bigr)\right)\xrightarrow[]{n\to\infty}0.\label{ref:truncated-counts}
\end{align}
\end{step}
\begin{stepproof}[Proof of Step~\ref{step:ppp-3}]
For $1\leq i\leq q_n$, the Markov property and the spectral decomposition, followed by the computation from Step~2 at the fixed point $x$, give
\begin{align}
\mathbb P_{x,n}(X_i\in U_n)&=\frac{\mathcal P^i\left(\mathbbm 1_{U_n}\mathcal P^{n-i}\mathbbm 1_M\right)(x)}{\mathcal P^n\mathbbm 1_M(x)}=\frac{\lambda^{-i}\mathcal P^i\left(\mathbbm 1_{U_n}\left(\eta+\mathcal O(\vartheta^{n-i})\right)\right)(x)}{\eta(x)+\mathcal O_x(\vartheta^n)}\nonumber\\
&=\mathcal Q^i(x,U_n)+\mathcal O_x(\vartheta^{n-i})\nonumber\\
&\leq\nu(U_n)+C_x\vartheta^{i-1}\sup_{z\in M}\mathcal P(z,U_n)+\mathcal O_x(\vartheta^{n-i}).\label{eq:low}
\end{align}
Here we used $\eta(x)>0$ in the last two lines.

 For $1\leq t\leq q_n$, applying the spectral decomposition to the initial gap $n-t$ gives
\begin{align}
\mathbb P_{x,n}(X_{n-t}\in U_n)&=\frac{\mathcal P^{n-t}\bigl(\mathbbm 1_{U_n}\mathcal P^t\mathbbm 1_M\bigr)(x)}{\mathcal P^n\mathbbm 1_M(x)} =\frac{\lambda^{-t} \lambda^{-(n-t)}\mathcal P^{n-t}\bigl(\mathbbm 1_{U_n}\mathcal P^t\mathbbm 1_M\bigr)(x)}{\lambda^{-n}\mathcal P^n\mathbbm 1_M(x)}\nonumber\\
&= \frac{ \mu(\mathbbm 1_{U_n} \lambda^{-t} \mathcal P^t \mathbbm 1_M)\eta(x) + \mathcal O(\vartheta^{n-t}) }{\eta(x) + \mathcal O(\vartheta^{n}  )}= \mu\bigl(\mathbbm 1_{U_n} \eta  + \mathbbm 1_{U_n} R^t\mathbbm 1_M  \bigr)+\mathcal O_x(\vartheta^{n-t})\nonumber\\
&\leq C\nu(U_n)+\mathcal O_x(\vartheta^{n-t}).\label{eq:high}
\end{align}
For the last inequality, we used that $R^t\mathbbm 1_M$ is uniformly bounded in $L^\infty(M,\mu)$ and that $\mu(U_n)\leq\eta_*^{-1}\nu(U_n)$ for $n$ large enough. From \eqref{eq:low}, \eqref{eq:high} and Step~\ref{step:ppp-1}, we obtain
\begin{align*}
\sum_{\substack{0\leq i<q_n\\\text{or }n-q_n<i<n}}\mathbb P_{x,n}(X_i\in U_n)&\leq C_x\bigl(q_n\nu(U_n)+\sup_{z\in M}\mathcal P(z,U_n)\bigr)+\smallO(1)\xrightarrow[]{n\to\infty}0.
\end{align*}
The term $i=0$ vanishes for all sufficiently large $n$ because
$f_\beta(x)/B_n\to0.$
Here we use the convention $f_\beta(x_0)=1$. This completes Step~\ref{step:ppp-3}.
\end{stepproof}

Let $\mathscr L:=\{(j,\ell):1\leq j\leq k,\ 1\leq\ell\leq r_j\}$ and $T_n:=\{q_n,\ldots,n-q_n\}$. We refer to $\mathscr L$ as the set of labels and $T_n$ as the set of times. An element $\mathrm{t}\in\operatorname{Inj}(\mathscr L,T_n)$ assigns to each label $(j,\ell)$ a time $t_{j,\ell}:=\mathrm{t}(j,\ell)\in T_n$, with no time being chosen twice. Define 
$$
\mathscr G_n:=\left\{\mathrm{t}\in\operatorname{Inj}(\mathscr L,T_n):|t_{j,\ell}-t_{j',\ell'}|>q_n\text{ whenever }(j,\ell)\neq(j',\ell')\right\}.
$$
i.e.~the set of injective maps whose selected times are mutually separated by more than $q_n$.

Given $z,r\in \mathbb N_0$, recall the notation 
$$
(z)_r:= \begin{cases}
z(z-1)(z-2)\ldots (z-r+1),&\ \text{if }r\in \mathbb N\\
1,&\ \text{if }r=0.
\end{cases}
$$

\begin{step}[4] \label{step:ppp-4}
We show that
$$
\mathbb E_{x,n}\left[\prod_{j=1}^k\bigl(\widetilde N_n(I_j)\bigr)_{r_j}\right]=S_n^{\mathrm{sep}}+S_n^{\mathrm{cl}},
$$
where
\begin{align}
S_n^{\mathrm{sep}}&:=\sum_{\mathrm{t}\in\mathscr G_n}\mathbb P_{x,n}\left(X_{t_{j,\ell}}\in A_{n,j}\text{ for every }(j,\ell)\in\mathscr L\right),\quad \text{and} \quad\nonumber\\
S_n^{\mathrm{cl}}&:=\sum_{\mathrm{t}\in\operatorname{Inj}(\mathscr L,T_n)\setminus\mathscr G_n}\mathbb P_{x,n}\left(X_{t_{j,\ell}}\in A_{n,j}\text{ for every }(j,\ell)\in\mathscr L\right).\nonumber
\end{align}
\end{step}
\begin{stepproof}[Proof of Step~\ref{step:ppp-4}]
To see this, first observe that
$$
\bigl(\widetilde N_n(I_j)\bigr)_{r_j}=\sum_{\substack{t_{j,1},\ldots,t_{j,r_j}\in T_n\\t_{j,1},\ldots,t_{j,r_j}\text{ distinct}}}\prod_{\ell=1}^{r_j}\mathbbm 1_{A_{n,j}}(X_{t_{j,\ell}}) \ \text{for every }j\in\{1,\ldots,k\}.
$$
Indeed, if $\widetilde N_n(I_j)=z$, then there are exactly $z$ observation indices $i\in T_n$ for which $X_i\in A_{n,j}$. The right-hand side counts the ordered choices of $r_j$ distinct indices among these $z$ indices. The number of such choices is $z(z-1)\cdots(z-r_j+1)=(z)_{r_j},$ which proves the identity.

When the products over $j=1,\ldots,k$ are expanded, a term in which the same time is chosen for two different intervals is zero, since $A_{n,j}\cap A_{n,j'}=\varnothing$ whenever $j\neq j'$. Therefore, only injective choices of times contribute, and
$$
\mathbb E_{x,n}\left[\prod_{j=1}^k\bigl(\widetilde N_n(I_j)\bigr)_{r_j}\right]=\sum_{\mathrm{t}\in\operatorname{Inj}(\mathscr L,T_n)}\mathbb P_{x,n}\left(X_{t_{j,\ell}}\in A_{n,j}\text{ for every }(j,\ell)\in\mathscr L\right).
$$
The set $\operatorname{Inj}(\mathscr L,T_n)$ is the disjoint union of $\mathscr G_n$ and $\operatorname{Inj}(\mathscr L,T_n)\setminus\mathscr G_n$. Splitting the last sum according to this disjoint union gives the required decomposition and completes Step~\ref{step:ppp-4}.
\end{stepproof}

\begin{step}[5] \label{step:ppp-5}
We show that
$$
S_n^{\mathrm{sep}}\xrightarrow[]{n\to\infty}\prod_{j=1}^k\upsilon_\alpha(I_j)^{r_j}.
$$
\end{step}
\begin{stepproof}[Proof of Step~\ref{step:ppp-5}]
The case $r=0$ is trivial; we therefore assume that $r\geq 1$.  Since $|\mathscr L|=r_1+\cdots+r_k=r$, the total number of injective maps from $\mathscr L$ to $T_n$ is
$$
\left|\operatorname{Inj}(\mathscr L,T_n)\right|=(|T_n|)_r=n^r+\mathcal O_r(q_nn^{r-1}),
$$
where we have used that $|T_n|=n-2q_n+1=n+\mathcal O(q_n)$. 

We next estimate the number of injective maps which do not belong to $\mathscr G_n$. Every such map contains two distinct labels $(j,\ell)$ and $(j',\ell')$ whose corresponding observation indices satisfy
$$
|t_{j,\ell}-t_{j',\ell'}|\leq q_n.
$$
If $r=1$, then $\mathscr G_n=\operatorname{Inj}(\mathscr L,T_n)$, so the complement is empty. We may therefore assume that $r\geq2$ when estimating the number of maps outside $\mathscr G_n$.
\begin{itemize}
\item There are at most $\binom{r}{2}$ choices for the pair of distinct labels, since $\mathscr L$ contains $r$ labels.
\item After fixing these labels, there are at most $|T_n|$ choices for the observation index $t_{j,\ell}\in T_n$ associated with the first label.
\item Once $t_{j,\ell}$ has been chosen, the second observation index must belong to
$$
T_n\cap\{t_{j,\ell}-q_n,\ldots,t_{j,\ell}+q_n\},
$$
which contains at most $2q_n+1$ elements.
\item After choosing these two observation indices, we disregard the injectivity restriction and allow each of the remaining $r-2$ labels to be assigned any element of $T_n$. This gives at most $|T_n|^{r-2}$ choices.
\end{itemize}
It follows that
$$
\left|\operatorname{Inj}(\mathscr L,T_n)\setminus\mathscr G_n\right|\leq\binom{r}{2}|T_n|(2q_n+1)|T_n|^{r-2}\leq C_rq_nn^{r-1}.
$$
Subtracting these maps from the total number of injective maps gives
$$
|\mathscr G_n|= \left|\operatorname{Inj}(\mathscr L,T_n)\right| - \left|\operatorname{Inj}(\mathscr L,T_n)\setminus \mathscr G_n\right| =n^r+\mathcal O_r(q_nn^{r-1}).
$$
For every $\mathrm{t}\in\mathscr G_n$, the selected times belong to $T_n$ and are mutually separated by more than $q_n$. Lemma~\ref{lem:mix}, applied to the corresponding indicator functions, gives
$$
\mathbb P_{x,n}\left(X_{t_{j,\ell}}\in A_{n,j}\text{ for every }(j,\ell)\in\mathscr L\right)=\prod_{j=1}^k\nu(A_{n,j})^{r_j}+\mathcal O_{x,r}(\vartheta^{q_n}).
$$
Hence,
\begin{align*}
S_n^{\mathrm{sep}}&=|\mathscr G_n|\prod_{j=1}^k\nu(A_{n,j})^{r_j}+\mathcal O_{x,r}\bigl(n^r\vartheta^{q_n}\bigr)\nonumber\\
&=\frac{|\mathscr G_n|}{n^r}\prod_{j=1}^k\bigl(n\nu(A_{n,j})\bigr)^{r_j}+\mathcal O_{x,r}\bigl(n^r\vartheta^{q_n}\bigr)\xrightarrow[]{n\to\infty}\prod_{j=1}^k\upsilon_\alpha(I_j)^{r_j}.\nonumber
\end{align*}
Indeed, $|\mathscr G_n|/n^r\to1$, $n\nu(A_{n,j})\to\upsilon_\alpha(I_j)$, and $n^r\vartheta^{q_n}\to0$. This completes Step~\ref{step:ppp-5}.
\end{stepproof}

\begin{step}[6]\label{step:ppp-6}
We show that $S_n^{\mathrm{cl}}\xrightarrow[]{n\to\infty}0.$ 
Consequently, combining this limit with Steps~\ref{step:ppp-4} and~\ref{step:ppp-5}, we obtain
$$
\mathbb E_{x,n}\left[\prod_{j=1}^k\bigl(\widetilde N_n(I_j)\bigr)_{r_j}\right]\xrightarrow[]{n\to\infty}\prod_{j=1}^k\upsilon_\alpha(I_j)^{r_j}.
$$
\end{step}
\begin{stepproof}[Proof of Step~\ref{step:ppp-6}]
The case $r=0$ is trivial. If $r=1$, then $\operatorname{Inj}(\mathscr L,T_n)\setminus\mathscr G_n=\varnothing,$
and hence $S_n^{\mathrm{cl}}=0$. We may therefore assume that $r\geq2$.

We divide the proof into three parts. We first decompose each tuple contributing to $S_n^{\mathrm{cl}}$ into clusters. We then estimate the contribution of a single cluster. Finally, we use the long gaps between the clusters to factorise their joint probability and count all possible tuples.

Fix $\mathrm{t}\in\operatorname{Inj}(\mathscr L,T_n)\setminus\mathscr G_n$
and write its selected times in increasing order as $t_1<t_2<\cdots<t_r.$
For every $a\in\{1,\ldots,r\}$, there is a unique label $(j_a,\ell_a)\in\mathscr L$ such that $\mathrm{t}(j_a,\ell_a)=t_a.$
Thus, the event associated with $\mathrm{t}$ is
$$
X_{t_a}\in A_{n,j_a}\ \text{for every }a\in\{1,\ldots,r\}.
$$
Since $\mathrm{t}\notin\mathscr G_n$, there exist two distinct selected times whose distance is at most $q_n$. After ordering the selected times, this implies that
$$
t_{a+1}-t_a\leq q_n\ \text{for at least one }a\in\{1,\ldots,r-1\}.
$$

Let us decompose $\{t_1,\ldots,t_r\}$ into clusters such that consecutive elements of each cluster are at distance at most $q_n$. There exist unique indices
\begin{itemize}
\item $0=p_0<p_1<\cdots<p_c=r$;
\item $t_{a+1}-t_a\leq q_n$ for $p_{\ell-1}+1\leq a<p_\ell$;
\item $t_{p_\ell+1}-t_{p_\ell}>q_n$ for $1\leq\ell<c$.
\end{itemize}
Observe that $c$ is the number of clusters. The $\ell$-th cluster is $\mathcal C_\ell:=\{t_{p_{\ell-1}+1},\ldots,t_{p_\ell}\},$
and its size is $s_\ell:=p_\ell-p_{\ell-1}.$
The cluster sizes satisfy
\begin{align}
s_1+\cdots+s_c=r \quad \text{and} \quad r-c=\sum_{\ell=1}^c(s_\ell-1)\geq1.\label{ref:number-internal-gaps}
\end{align}
Indeed, $s_\ell-1$ is the number of internal gaps in the $\ell$-th cluster. Since $\mathrm{t}\notin\mathscr G_n$, at least one cluster has size at least two, which proves the final inequality in \eqref{ref:number-internal-gaps}.

We next estimate the contribution of a single cluster to the sum defining $S_n^{\mathrm{cl}}$. Consider a cluster of size $s\geq2$, with interval labels $j_1,\ldots,j_s$ and successive internal gaps
$$
1\leq m_1,\ldots,m_{s-1}\leq q_n.
$$
Set
$$
\pi_n(j_1,\ldots,j_s;m_1,\ldots,m_{s-1}):=\mathbb Q_\nu\left(X_0\in A_{n,j_1},X_{m_1}\in A_{n,j_2},\ldots,X_{m_1+\cdots+m_{s-1}}\in A_{n,j_s}\right).
$$
Since $A_{n,j}\subset U_n$ for every $j$, the Markov property gives
$$
\pi_n(j_1,\ldots,j_s;m_1,\ldots,m_{s-1})\leq\nu(U_n)\prod_{b=1}^{s-1}\sup_{y\in U_n}\mathcal Q^{m_b}(y,U_n).
$$
Applying \eqref{ref:q-short-return} to every factor, we obtain
\begin{align}
\pi_n(j_1,\ldots,j_s; m_1,\ldots,m_{s-1})&\leq\nu(U_n)\prod_{b=1}^{s-1}\left(\nu(U_n)+C\vartheta^{m_b-1}\sup_{z\in M}\mathcal P(z,U_n)\right).\label{ref:single-cluster-probability}
\end{align}
Define $\varepsilon_n:=q_n\nu(U_n)+\frac{C}{1-\vartheta}\sup_{z\in M}\mathcal P(z,U_n).$ Summing \eqref{ref:single-cluster-probability} over all possible internal gaps gives
\begin{align}
\sum_{m_1,\ldots,m_{s-1}=1}^{q_n}\pi_n(j_1,\ldots,j_s;m_1,\ldots,m_{s-1})&\leq\nu(U_n)\varepsilon_n^{s-1}.\label{ref:single-cluster-sum}
\end{align}
Indeed,
$$
\sum_{m=1}^{q_n}\left(\nu(U_n)+C\vartheta^{m-1}\sup_{z\in M}\mathcal P(z,U_n)\right)\leq\varepsilon_n.
$$
For a cluster of size $s=1$ there are no internal gaps and its stationary $Q$-probability is bounded by
$$
\nu(A_{n,j})\leq\nu(U_n)=\nu(U_n)\varepsilon_n^{s-1}.
$$
Thus, \eqref{ref:single-cluster-sum} remains valid for $s=1$, with the sum over internal gaps interpreted as having one term.

Since $q_n=\lceil(\log n)^2\rceil$ and $\nu(U_n)=\mathcal O(n^{-1})$, we have $q_n\nu(U_n)\to0.$ Step~\ref{step:ppp-1}~\ref{it:ppp-1-iii} also gives $\sup_{z\in M}\mathcal P(z,U_n)\to0.$ Therefore, $\varepsilon_n\to0.$

We now use the estimate for each individual cluster to bound the joint contribution of all $c$ clusters in the tuple. For each $\ell\in\{1,\ldots,c\}$, define
$$
\pi_{\ell,n}:=\mathbb Q_\nu\left(X_{t_{p_{\ell-1}+b}-t_{p_{\ell-1}+1}}\in A_{n,j_{p_{\ell-1}+b}}\text{ for every }1\leq b\leq s_\ell\right).
$$
Thus, $\pi_{\ell,n}$ is the stationary $Q$-probability of the visits in the $\ell$-th cluster after shifting the first selected time of that cluster to zero.

Since every selected time belongs to $T_n=\{q_n,\ldots,n-q_n\}$, we have $t_1\geq q_n$
and $n-t_r\geq q_n.$
Moreover, the maximality of the cluster decomposition gives $t_{p_\ell+1}-t_{p_\ell}>q_n$
for every $1\leq\ell<c$. Applying Lemma~\ref{lem:cluster-mixing} to the $c$ clusters, with $D_a=A_{n,j_a}$, $N=n$ and $q=q_n$, gives
\begin{align}
\mathbb P_{x,n}\left(X_{t_a}\in A_{n,j_a}\text{ for every }1\leq a\leq r\right)&=\prod_{\ell=1}^c\pi_{\ell,n}+\mathcal O_{x,r}(\vartheta^{q_n}).\label{ref:cluster-factorisation}
\end{align}

Let us first sum the product term in \eqref{ref:cluster-factorisation}. Fix the number of clusters $c$, their sizes $s_1,\ldots,s_c$ and the temporal order of the labels. A tuple is determined by choosing the initial time of each cluster and the internal gaps within each cluster.

There are at most $n^c$ choices for the initial times of the clusters. The restrictions that the clusters must be ordered and separated by more than $q_n$ only reduce this number. For each cluster, the sum over its internal gaps is controlled by \eqref{ref:single-cluster-sum}. Therefore, the sum of the product terms is bounded by
\begin{align}
\sum_{\substack{\mathrm{t}\in\operatorname{Inj}(\mathscr L,T_n)\setminus\mathscr G_n\\\mathrm{t}\text{ has }c\text{ clusters of sizes }s_1,\ldots,s_c}}\prod_{\ell=1}^c\pi_{\ell,n} \leq C_{x,r}n^c\prod_{\ell=1}^c\left(\nu(U_n)\varepsilon_n^{s_\ell-1}\right)=C_{x,r}\bigl(n\nu(U_n)\bigr)^c\varepsilon_n^{r-c}.\label{ref:cluster-main-count}
\end{align}
Here we used \eqref{ref:number-internal-gaps}.

We next sum the error term in \eqref{ref:cluster-factorisation}. There are at most $\left|\operatorname{Inj}(\mathscr L,T_n)\right|\leq n^r$
injective choices of the $r$ selected times. Since the error associated with each choice is bounded by $\mathcal O_{x,r}(\vartheta^{q_n})$, the total absolute error is bounded by
\begin{align}
&\sum_{\mathrm{t}\in\operatorname{Inj}(\mathscr L,T_n)\setminus\mathscr G_n}\left|\mathbb P_{x,n}\left(X_{t_a}\in A_{n,j_a}\text{ for every }1\leq a\leq r\right)-\prod_{\ell=1}^c\pi_{\ell,n}\right|\leq C_{x,r}n^r\vartheta^{q_n}.\label{ref:cluster-error-count}
\end{align}
For fixed $r$ and $c$, there are at most $r!$ possible temporal orders of the labels and $\binom{r-1}{c-1}$
possible decompositions of $r$ into $c$ positive cluster sizes. These numbers do not depend on $n$ and can be absorbed into $C_{x,r}$.

Summing \eqref{ref:cluster-main-count} over $c\in\{1,\ldots,r-1\}$ and adding \eqref{ref:cluster-error-count}, we obtain
\begin{align}
S_n^{\mathrm{cl}}&\leq C_{x,r}\sum_{c=1}^{r-1}\bigl(n\nu(U_n)\bigr)^c\varepsilon_n^{r-c}+C_{x,r}n^r\vartheta^{q_n}.\label{ref:close-tuples-bound}
\end{align}
Since $n\nu(U_n)=\mathcal O(1),\ \varepsilon_n\to0,\ r-c\geq1,\ n^r\vartheta^{q_n}\to0,$
every term on the right-hand side of \eqref{ref:close-tuples-bound} converges to zero. Hence, $S_n^{\mathrm{cl}}\xrightarrow[]{n\to\infty}0.$
Finally, Step~\ref{step:ppp-4} gives
$$
\mathbb E_{x,n}\left[\prod_{j=1}^k\bigl(\widetilde N_n(I_j)\bigr)_{r_j}\right]=S_n^{\mathrm{sep}}+S_n^{\mathrm{cl}},
$$
while Step~\ref{step:ppp-5} gives
$$
S_n^{\mathrm{sep}}\to\prod_{j=1}^k\upsilon_\alpha(I_j)^{r_j},
$$
concluding Step~\ref{step:ppp-6}.
\end{stepproof}

\begin{step}[7]\label{step:ppp-7}
We show that the joint factorial moments converge to those of $(Y_1,\ldots,Y_k)$ and conclude the Poisson point-process limit.
\end{step}
\begin{stepproof}[Proof of Step~\ref{step:ppp-7}]
Combining Steps~\ref{step:ppp-4}, \ref{step:ppp-5} and \ref{step:ppp-6} and using the factorial moments of independent Poisson random variables, for arbitrary $r_1,\ldots,r_k\in\mathbb N_0$, we obtain
$$
\mathbb E_{x,n}\left[\prod_{j=1}^k\bigl(\widetilde N_n(I_j)\bigr)_{r_j}\right]\xrightarrow[]{n\to\infty}\prod_{j=1}^k\upsilon_\alpha(I_j)^{r_j}=\mathbb E\left[\prod_{j=1}^k(Y_j)_{r_j}\right].
$$
Thus, all mixed factorial moments of
$\left(\widetilde N_n(I_1),\ldots,\widetilde N_n(I_k)\right)$
converge to the corresponding mixed factorial moments of
$(Y_1,\ldots,Y_k).$

Since the falling factorials $(z)_r$ form a basis for the space of polynomials~\cite[\S6.1, equation~(6.10)]{GrahamKnuthPatashnik1994}, convergence of all mixed factorial moments implies convergence of all mixed ordinary moments, i.e.
$$
\mathbb E_{x,n}\left[\prod_{j=1}^k\bigl(\widetilde N_n(I_j)\bigr)^{r_j}\right]\xrightarrow[]{n\to\infty}\mathbb E\left[\prod_{j=1}^kY_j^{r_j}\right].
$$
In particular, for every $u_1,\ldots,u_k\in\mathbb R$ and every $m\in\mathbb N$,
$$
\mathbb E_{x,n}\left[\left(\sum_{j=1}^ku_j\widetilde N_n(I_j)\right)^m\right]\xrightarrow[]{n\to\infty}\mathbb E\left[\left(\sum_{j=1}^ku_jY_j\right)^m\right].
$$
Moreover, since $Y_1,\ldots,Y_k$ are independent Poisson random variables,
$$
\mathbb E\left[\exp\left(s\sum_{j=1}^ku_jY_j\right)\right]=\exp\left(\sum_{j=1}^k\upsilon_\alpha(I_j)\left(e^{su_j}-1\right)\right)<\infty
$$
for every $s\in\mathbb R$. Hence, the distribution of $\sum_{j=1}^ku_jY_j$ is determined by its moments. From Theorem \ref{thm:methodofmoments}
$$
\sum_{j=1}^ku_j\widetilde N_n(I_j)\xrightarrow[d]{n\to\infty}\sum_{j=1}^ku_jY_j
$$
for every $u_1,\ldots,u_k\in\mathbb R$. Theorem~\ref{thm:cramerwold}  then yields
\begin{align}
\left(\widetilde N_n(I_1),\ldots,\widetilde N_n(I_k)\right)\xrightarrow[d]{n\to\infty}\left(Y_1,\ldots,Y_k\right).\label{eq:jointconversition}
\end{align}
By \eqref{ref:truncated-counts}, the same convergence holds with $N_n(I_j)$ in place of $\widetilde N_n(I_j)$.

For every interval $J=(a,b]$ with $0<a<b\leq\infty$, the boundary $\partial J$ contains at most two points. Since $\upsilon_\alpha$ has no atoms,
$$
N(\partial J)\sim\operatorname{Poi}\left(\upsilon_\alpha(\partial J)\right)=\operatorname{Poi}(0).
$$
Hence,
$N(\partial J)=0$
almost surely. Combining this fact with the joint convergence in \eqref{eq:jointconversition}, Theorem~\ref{thm:kallenberg-random-measures} gives
$$
N_n\xrightarrow[d]{n\to\infty} N\ \text{in }\mathcal M_p((0,\infty]),\  N\sim\operatorname{PRM}(\upsilon_\alpha).
$$
This concludes Step~\ref{step:ppp-7}.
\end{stepproof}
The proof of the theorem is now complete.
\end{proof}

\section{Conditional stable laws for \texorpdfstring{$0<\alpha<2$}{0<α<2}}
\label{sec:stable-laws}
The point-process limit analysed in the previous section does not immediately imply convergence of the partial sums. This is a consequence of the identity function on $E := (0,\infty]$ not being compactly supported. We overcome this obstacle by truncating the identity function near zero and infinity following~\cite{Resnick1987,DavisHsing1995}.

Let us make this obstruction more evident. Recall that
$$
N_n(\d y):=\sum_{i=0}^{n-1}\delta_{\frac{f_\beta\circ X_i}{B_n}}(\d y).
$$
Hence, for every $n\in \mathbb N$,
$$
\frac{1}{B_n}\sum_{i=0}^{n-1}f_\beta\circ X_i=\int_{(0,\infty]}yN_n(\d y)=N_n(\mathrm{id}),
$$
where $\mathrm{id}(y):=y$. The function $\mathrm{id}$ is not a compactly supported continuous function on $(0,\infty]$. Therefore, while Theorem \ref{thm:ppp} guarantees
$$
\mathbb P_x\bigl(N_n\in\cdot\mid\tau>n\bigr)\xRightarrow[]{n\to\infty}\mathrm{Prob}\bigl(N\in\cdot\bigr),
$$
with $N\sim\operatorname{PRM}(\upsilon_\alpha)$, the continuous mapping theorem cannot be applied directly to the map
$$
\xi\mapsto\xi(\mathrm{id})=\int_{(0,\infty]}y\,\xi(\d y).
$$
Consequently, the point-process convergence does not by itself imply that
$$
\mathbb P_x\bigl(N_n(\mathrm{id})\in\cdot\mid\tau>n\bigr)\xRightarrow[]{n\to\infty}\mathrm{Prob}\bigl(N(\mathrm{id})\in\cdot\bigr).
$$
Let us instead truncate the identity function near zero and infinity: fix $0<\varepsilon<1<R$. Then
\begin{align}
    N_n(\mathrm{id})=\int_{(0,\varepsilon]}y\,N_n(\d y)+\int_{(\varepsilon,R]}y\,N_n(\d y)+\int_{(R,\infty]}y\,N_n(\d y).\label{eq:integrand}
\end{align}
The middle integrand $y\mathbbm 1_{(\varepsilon,R]}(y)$ now is supported away from zero and infinity. Therefore, since $N(\{\e,R\})=0$ the point-process convergence can be used to prove
$$
\int_{(\varepsilon,R]}y\,N_n(\d y)\xrightarrow[d]{n\to\infty}\int_{(\varepsilon,R]}y\,N(\d y).
$$
It remains to control the extremal contributions from $(0,\varepsilon]$ and $(R,\infty]$. The required argument depends on the value of $\alpha$:
\begin{itemize}
\item If $0<\alpha<1$, no centring is required. We show that the small-jump contribution vanishes in probability as $\varepsilon\to0$ and that the large-jump contribution vanishes in probability as $R\to\infty$. This yields
$$
N_n(\mathrm{id})\xrightarrow[d]{n\to\infty}\int_{(0,\infty]}y\,N(\d y).
$$
\item If $\alpha=1$, we use the truncated centring
$$
A_n:=n\nu\left(f_\beta\mathbbm 1_{\{f_\beta\leq B_n\}}\right).
$$
The small jumps must be compensated by their deterministic mean. After controlling the compensated small jumps and the large jumps, the limit is the compensated Poisson integral
$$
\lim_{\varepsilon\to0}\left(\int_{(\varepsilon,\infty]}y\,N(\d y)-\int_{(\varepsilon,1]}y\,\upsilon_1(\d y)\right).
$$
\item If $1<\alpha<2$, the observable $f_\beta$ is integrable with respect to $\nu$, and we take
$$
A_n:=n\nu(f_\beta).
$$
We then control the centred small-jump contribution and the centred large-jump contribution. The limit is the compensated Poisson integral
$$
\lim_{\varepsilon\to0}\left(\int_{(\varepsilon,\infty]}y\,N(\d y)-\int_{(\varepsilon,\infty]}y\,\upsilon_\alpha(\d y)\right).
$$
\end{itemize}

\subsection{The case \texorpdfstring{$0<\alpha<1$}{0<α<1}}

\begin{proof}[Proof of Theorem~\ref{thm:main-stable-laws}~\ref{it:stab-law-i}]
Write $\mathbb P_{x,n}(\cdot):=\mathbb P_x(\cdot\mid\tau>n)$. By Theorem~\ref{thm:ppp} and the truncation argument \eqref{eq:integrand}, it remains only to control the small and large jumps.

We first control the small jumps.
From \eqref{eq:mixingallorders}, for every $1\leq i<n$,
$$
\mathbb E_{x,n}\left[\frac{1}{B_n}f_\beta\circ X_i\cdot \mathbbm 1_{\{f_\beta\leq\varepsilon B_n\}}(X_i)\right]\leq \frac{1}{B_n}\nu(f_\beta\cdot \mathbbm 1_{\{f_\beta\leq\varepsilon B_n\} })+\widetilde{C}_x\vartheta^{\min\{i,n-i\}}\varepsilon.
$$
Summing over $i$ gives
\begin{align*}
\mathbb E_{x,n}\left[\int_{(0,\varepsilon]}y\,N_n(\d y)\right] &= \mathbb E_x\left[ \left.\sum_{i=0}^{n-1} \frac{f_\beta\circ X_i}{B_n} \mathbbm 1_{\left\{\frac{f_\beta (X_i)}{B_n}\leq \e\right\}}\,\right|\, \tau >n\right]\\
&\leq
\frac{n}{B_n}\nu(f_\beta\mathbbm 1_{\{f_\beta\leq\varepsilon B_n\}})
+C_x\varepsilon+\frac{f_\beta(x)}{B_n},
\end{align*}
where $C_x = 2\widetilde C_x/(1-\vartheta).$ From \eqref{eq:tail}, $\nu(f_\beta>t)\sim C(x_0)t^{-\alpha}.$
Since $0<\alpha<1$, integration of this asymptotic gives 
$$
\int_0^t\nu(f_\beta>u)\,\d u\sim\frac{C(x_0)}{1-\alpha}t^{1-\alpha}\sim
\frac{t\nu(f_\beta>t)}{1-\alpha}.
$$
Using that
$$
\nu\left(f_\beta\mathbbm 1_{\{f_\beta\leq t\}}\right) =  \nu\left(\int_0^t \mathbbm 1_{\{u<f_\beta\}} \d u - t \mathbbm 1_{\{t<f_\beta\}} \right)=\int_0^t\nu(f_\beta>u)\,\d u-t\nu(f_\beta>t),
$$
we obtain
\begin{align}
\nu\left(f_\beta\mathbbm 1_{\{f_\beta\leq t\}}\right)\sim\frac{\alpha}{1-\alpha}
t\nu(f_\beta>t).\label{eq:assymp}
\end{align}
Moreover, for each $\e>0$ fixed, we have that $\varepsilon B_n\to\infty$. Setting $t=\varepsilon B_n$ in \eqref{eq:assymp} we obtain
$$
\nu\left(
f_\beta\mathbbm 1_{\{f_\beta\leq\varepsilon B_n\}}\right)\sim\frac{\alpha}{1-\alpha}\varepsilon B_n\nu(f_\beta>\varepsilon B_n).
$$
Hence,
$$
\frac{n}{B_n}\nu\left(f_\beta\mathbbm 1_{\{f_\beta\leq\varepsilon B_n\}}\right)\sim\frac{\alpha}{1-\alpha}\varepsilon n\nu(f_\beta>\varepsilon B_n).$$
By Proposition~\ref{prop:tail-asymptotics}, $
n\nu(f_\beta>uB_n)\xrightarrow[]{n\to\infty}u^{-\alpha}$
for every $u>0$, so taking $u=\varepsilon$ gives
$$
n\nu(f_\beta>\varepsilon B_n)\xrightarrow[]{n\to\infty}\varepsilon^{-\alpha}.
$$
Consequently,
\begin{align}
\frac{n}{B_n}\nu\left(f_\beta\mathbbm 1_{\{f_\beta\leq\varepsilon B_n\}}\right)\xrightarrow[]{n\to\infty}\frac{\alpha}{1-\alpha}\varepsilon\varepsilon^{-\alpha}=\frac{\alpha}{1-\alpha}\varepsilon^{1-\alpha}.\label{eq:limepsalpha}
\end{align}
Therefore, for every $\delta>0$, Markov's inequality gives
\begin{align*}
\mathbb P_{x,n}\left(\int_{(0,\varepsilon]}y\,N_n(\d y)>\delta\right)&\leq\frac{1}{\delta}\mathbb E_{x,n}\left[\int_{(0,\varepsilon]}y\,N_n(\d y)\right]\\
&\leq\frac{1}{\delta}\left(\frac{n}{B_n}\nu\left(f_\beta\mathbbm 1_{\{f_\beta\leq\varepsilon B_n\}}\right)+C_x\varepsilon+\frac{f_\beta(x)}{B_n}\right).
\end{align*}
Taking the upper limit as $n\to\infty$ and using \eqref{eq:limepsalpha}, we obtain
$$
\limsup_{n\to\infty}\mathbb P_{x,n}\left(\int_{(0,\varepsilon]}y\,N_n(\d y)>\delta\right)\leq
\frac{1}{\delta}\left(
\frac{\alpha}{1-\alpha}\varepsilon^{1-\alpha}+C_x\varepsilon\right).
$$
Since $0<\alpha<1$, both terms on the right-hand side converge to zero as $\varepsilon\to0$. Hence,
\begin{align}
\lim_{\varepsilon\to0}\limsup_{n\to\infty}\mathbb P_{x,n}\left(\int_{(0,\varepsilon]}y\,N_n(\d y)>\delta\right)
&=0.
\label{ref:small-jumps-alpha-less-one}
\end{align}

For the large jumps, we apply Theorem~\ref{thm:ppp} with $k=1$ and $I_1=(R,\infty]$. Since this interval is bounded away from zero, the theorem gives
$$
N_n((R,\infty])\xrightarrow[d]{n\to\infty}N((R,\infty]),
$$
where $N((R,\infty])\sim\operatorname{Poi}\left(\upsilon_\alpha((R,\infty])\right)=\operatorname{Poi}(R^{-\alpha}),$ where $N((R,\infty])$ is Poisson with parameter $R^{-\alpha}$. Moreover,
$$
\left\{
\int_{(R,\infty]}y\,N_n(\d y)>0
\right\}
=
\{N_n((R,\infty])\geq1\}.
$$
Hence, for every fixed $R>0$,
\begin{align*}
\lim_{n\to\infty}\mathbb P_{x,n}\left(\int_{(R,\infty]}y\,N_n(\d y)>0\right)
&=\lim_{n\to\infty} \mathbb P_{x,n}(N_n((R,\infty))\geq 1 )
\\&=\mathrm{Prob}\left(
N((R,\infty])\geq1
\right)
=
1-e^{-R^{-\alpha}}.
\end{align*}
Since $1-e^{-R^{-\alpha}}\to0$ as $R\to\infty$, we conclude that
\begin{align}
\lim_{R\to\infty}\limsup_{n\to\infty}\mathbb P_{x,n}\left(\int_{(R,\infty]}y\,N_n(\d y)>0\right)
&=0.
\label{ref:large-jumps-alpha-less-one}
\end{align}
Combining Theorem~\ref{thm:ppp}, \eqref{ref:small-jumps-alpha-less-one} and \eqref{ref:large-jumps-alpha-less-one} with the truncation argument above gives
$$
N_n(\mathrm{id})\xrightarrow[d]{n\to\infty}N(\mathrm{id})
=
Z_\alpha.
$$
Since
$$
\int_{(0,1]}y\,\upsilon_\alpha(\d y)=\frac{\alpha}{1-\alpha}<\infty
\quad \text{and}\quad  
\upsilon_\alpha((1,\infty])=1,
$$
the random variable $Z_\alpha$ is almost surely finite. Finally, applying the formula for the Laplace transform of a Poisson point-process~\cite[Theorem~24.14]{Klenke2020} to the function $y\mapsto ty$ yields
\begin{align*}
\mathbb E\left[e^{-tZ_\alpha}\right]
&=
\exp\left(\int_{(0,\infty]}\left(e^{-ty}-1\right)\upsilon_\alpha(\d y)\right)=\exp\left(-\int_{(0,\infty]}\left(1-e^{-ty}\right)\alpha y^{-\alpha-1} \d y\right)\\
&=\exp\left(-\Gamma(1-\alpha)t^\alpha\right).
\end{align*}
From~\cite[Theorem~14.10]{Sato1999}, the equation above is the Laplace transform of a positive strictly $\alpha$-stable random variable. In particular, $Z_\alpha$ is totally skewed to the right.
\end{proof}

\subsection{A conditional second-moment bound}

We first establish a conditional second-moment estimate for bounded centred observables. We then apply it to the compensated small jumps when $1\leq\alpha<2$. The same estimate will also be used in the proof of Theorem~\ref{thm:conditional-clt-l2}.

\begin{lemma}
\label{lem:conditional-sum-second-moment}
Assume Hypothesis~\ref{hyp:H}. For every $x\in M\setminus Z$, there exists $C_x>0$ such that, for every bounded measurable function $g:M\to\mathbb R$ satisfying $\nu(g)=0$ and every $n\geq1$,
\begin{align}
\mathbb E_{x,n}\left[\left(\sum_{i=0}^{n-1}g(X_i)\right)^2\right]
&\leq C_xn\left(\nu(g^2)+\mu(g^2)\right)+C_x\|g\|_\infty^2.
\label{ref:conditional-sum-second-moment}
\end{align}
The constant $C_x$ is independent of $g$ and $n$.
\end{lemma}

\begin{proof}
We start by recording some estimates. From Theorem~\ref{thm:spectralgap} and Lemma~\ref{lem:l2-compactness}, we may choose a common $\vartheta\in(0,1)$ such that
\begin{align}
\|R^kf\|_\infty\leq C\vartheta^k\|f\|_\infty\quad \text{and}\quad \|R^kf\|_{L^2(\mu)}\leq C\vartheta^k\|f\|_{L^2(\mu)}
\label{eq:second-moment-spectral-bounds}
\end{align}
for every $k\geq1$. Since $\widehat{\mathcal P}^{k}=\Pi+R^k$, we may enlarge $C$ so that
\begin{align}
\sup_{k\geq0}\max\left\{\|\widehat{\mathcal P}^{k}\|_{L^\infty(\rho)\to L^\infty(\rho)},\|\widehat{\mathcal P}^{k}\|_{L^2(\mu)\to L^2(\mu)}
\right\}\leq C.\label{eq:second-moment-uniform-powers}
\end{align}
Since $\widehat{\mathcal P}^{m-k}\mathbbm 1_M=\eta+R^{m-k}\mathbbm 1_M$ and $\Pi(\eta g)=\eta\nu(g)=0$, for every $m\geq1$ we have
\begin{align}
\sum_{k=1}^{m-1}\widehat{\mathcal P}^{k}\left(g\widehat{\mathcal P}^{m-k}\mathbbm 1_M\right)
&=\sum_{k=1}^{m-1}R^k(\eta g)+\sum_{k=1}^{m-1}\widehat{\mathcal P}^{k}\left(gR^{m-k}\mathbbm 1_M\right),
\label{eq:second-moment-future-decomposition}
\end{align}
 Applying \eqref{eq:second-moment-spectral-bounds} and \eqref{eq:second-moment-uniform-powers}, together with ${\|gR^{m-k}\mathbbm 1_M\|_{L^2(\mu)}\leq\|g\|_{L^2(\mu)}\|R^{m-k}\mathbbm 1_M\|_\infty}$, gives
\begin{align}
\left\|\sum_{k=1}^{m-1}\widehat{\mathcal P}^{k}\left(g\widehat{\mathcal P}^{m-k}\mathbbm 1_M\right)\right\|_\infty
&\leq C\|g\|_\infty\sum_{k=1}^{m-1}\left(\vartheta^k+\vartheta^{m-k}\right)\leq C\|g\|_\infty,\ \text{and}\nonumber\\
\left\|\sum_{k=1}^{m-1}\widehat{\mathcal P}^{k}\left(g\widehat{\mathcal P}^{m-k}\mathbbm 1_M\right)\right\|_{L^2(\mu)}
&\leq C\|g\|_{L^2(\mu)}\sum_{k=1}^{m-1}\left(\vartheta^k+\vartheta^{m-k}\right)
\leq C\|g\|_{L^2(\mu)}.
\label{eq:second-moment-future-bounds}
\end{align}
Consequently, Cauchy--Schwarz, \eqref{eq:second-moment-uniform-powers} and \eqref{eq:second-moment-future-bounds} imply
\begin{align}
\left|\mu\left(g^2\widehat{\mathcal P}^{m}\mathbbm 1_M
+2g\sum_{k=1}^{m-1}\widehat{\mathcal P}^{k}\left(g\widehat{\mathcal P}^{m-k}\mathbbm 1_M\right)\right)\right|
&\leq C\mu(g^2),\ \text{and }\nonumber\\
\left\|g^2\widehat{\mathcal P}^{m}\mathbbm 1_M
+2g\sum_{k=1}^{m-1}\widehat{\mathcal P}^{k}\left(g\widehat{\mathcal P}^{m-k}\mathbbm 1_M\right)\right\|_\infty
&\leq C\|g\|_\infty^2.
\label{eq:second-moment-integrand-bounds}
\end{align}

We now expand the left-hand side of \eqref{ref:conditional-sum-second-moment} as
\begin{align}
\widehat{\mathcal P}^{n}\mathbbm 1_M(x)\,\mathbb E_{x,n}&\left[\left(\sum_{i=0}^{n-1}g(X_i)\right)^2\right]= \widehat{\mathcal P
}^{n}\left(\left(\sum_{i=0}^{n-1}g(X_i)\right)^2\right)(x) \\
&=\sum_{i=0}^{n-1}\widehat{\mathcal P}^{i}\left(
g^2\widehat{\mathcal P}^{n-i}\mathbbm 1_M
+2g\sum_{k=1}^{n-i-1}\widehat{\mathcal P}^{k}\left(g\widehat{\mathcal P}^{n-i-k}\mathbbm 1_M\right)
\right)(x).
\label{eq:second-moment-expansion}
\end{align}
From \eqref{eq:second-moment-integrand-bounds}, the term $i=0$ is bounded in absolute value by $C\|g\|_\infty^2$. For $i\geq1$, applying $\widehat{\mathcal P}^{i}=\Pi+R^i$ and using \eqref{eq:second-moment-spectral-bounds} and \eqref{eq:second-moment-integrand-bounds} yields
\begin{align}
&\left|\widehat{\mathcal P}^{i}\left(
g^2\widehat{\mathcal P}^{n-i}\mathbbm 1_M
+2g\sum_{k=1}^{n-i-1}\widehat{\mathcal P}^{k}\left(g\widehat{\mathcal P}^{n-i-k}\mathbbm 1_M\right)
\right)(x)\right|\leq C\eta(x)\mu(g^2)+C\vartheta^i\|g\|_\infty^2.
\label{eq:second-moment-summand-bound}
\end{align}
Summing \eqref{eq:second-moment-summand-bound} in \eqref{eq:second-moment-expansion} therefore gives
\begin{align}
\widehat{\mathcal P}^{n}\mathbbm 1_M(x)\,
\mathbb E_{x,n}\left[\left(\sum_{i=0}^{n-1}g(X_i)\right)^2\right]
&\leq Cn\eta(x)\mu(g^2)+C\|g\|_\infty^2.
\label{eq:second-moment-unnormalised}
\end{align}
Finally, positivity and $\widehat{\mathcal P}\eta=\eta$ imply
\begin{align}
\widehat{\mathcal P}^{n}\mathbbm 1_M(x)
&\geq\widehat{\mathcal P}^{n}\left(\frac{\eta}{\|\eta\|_\infty}\right)(x)
=\frac{\eta(x)}{\|\eta\|_\infty}>0.
\label{eq:second-moment-denominator}
\end{align}
Dividing \eqref{eq:second-moment-unnormalised} by this denominator, we conclude that
\begin{align}
\mathbb E_{x,n}\left[\left(\sum_{i=0}^{n-1}g(X_i)\right)^2\right]
&\leq Cn\mu(g^2)+C_x\|g\|_\infty^2,
\label{eq:second-moment-final-bound}
\end{align}
which implies \eqref{ref:conditional-sum-second-moment}.
\end{proof}

We now apply this estimate to the compensated small jumps.

\begin{lemma}
\label{lem:conditional-second-moment}
Assume that $1\leq\alpha<2$ and let $\upsilon_n(\d y):=n\nu(f_\beta/B_n\in\d y)$. For every $x\in M\setminus Z$ and $\varepsilon>0$,
\begin{align}
\mathbb E_{x,n}\left[\left(\int_{(0,\varepsilon]}y(N_n-\upsilon_n)(\d y)\right)^2\right]
&\leq \frac{C_xn}{B_n^2}\nu\left(f_\beta^2\mathbbm 1_{\{f_\beta\leq\varepsilon B_n\}}\right)+C_x\varepsilon^2+\smallO(1),
\label{ref:conditional-small-jump-second-moment}
\end{align}
where $C_x$ is independent of $n$ and $\varepsilon$, and $\smallO(1)$ tends to zero as $n\to\infty$ for fixed $\varepsilon$.
\end{lemma}

\begin{proof}
Define
$$
g:=\frac{1}{B_n}\left(f_\beta\mathbbm 1_{\{f_\beta\leq\varepsilon B_n\}}-\nu\left(f_\beta\mathbbm 1_{\{f_\beta\leq\varepsilon B_n\}}\right)\right).
$$
Then $\nu(g)=0$, $\|g\|_\infty\leq\varepsilon$ and, by the definitions of $N_n$ and $\upsilon_n$,
\begin{align}
\int_{(0,\varepsilon]}y(N_n-\upsilon_n)(\d y)
&=\sum_{i=0}^{n-1}g(X_i).
\label{eq:small-jump-centred-sum}
\end{align}
Choose a neighbourhood $V$ of $x_0$ and $\eta_*>0$ such that $\eta\geq\eta_*$ on $V$. Since $\nu=\eta\mu$ and $f_\beta$ is bounded on $M\setminus V$,
\begin{align}
\mu\left(f_\beta^2\mathbbm 1_{\{f_\beta\leq\varepsilon B_n\}}\right)
&\leq\eta_*^{-1}\nu\left(f_\beta^2\mathbbm 1_{\{f_\beta\leq\varepsilon B_n\}}\right)+C,
\label{eq:small-jump-moment-comparison}
\end{align}
where $C$ is independent of $n$ and $\varepsilon$. The variance identity under $\nu$ gives
\begin{align}
\nu(g^2)&\leq\frac{1}{B_n^2}\nu\left(f_\beta^2\mathbbm 1_{\{f_\beta\leq\varepsilon B_n\}}\right).
\label{eq:small-jump-nu-bound}
\end{align}
Using $|a-b|^2\leq2|a|^2+2|b|^2$, Cauchy--Schwarz and \eqref{eq:small-jump-moment-comparison}, we also obtain
\begin{align}
\mu(g^2)\leq\frac{2}{B_n^2}\mu\left(f_\beta^2\mathbbm 1_{\{f_\beta\leq\varepsilon B_n\}}\right)
+\frac{2}{B_n^2}\nu\left(f_\beta^2\mathbbm 1_{\{f_\beta\leq\varepsilon B_n\}}\right)\leq\frac{C}{B_n^2}\nu\left(f_\beta^2\mathbbm 1_{\{f_\beta\leq\varepsilon B_n\}}\right)+\frac{C}{B_n^2}.
\label{eq:small-jump-mu-bound}
\end{align}
Applying Lemma~\ref{lem:conditional-sum-second-moment} to $g$ and using \eqref{eq:small-jump-centred-sum}, \eqref{eq:small-jump-nu-bound} and \eqref{eq:small-jump-mu-bound}, we conclude that
\begin{align}
\mathbb E_{x,n}\left[\left(\int_{(0,\varepsilon]}y(N_n-\upsilon_n)(\d y)\right)^2\right]
&\leq C_xn\left(\nu(g^2)+\mu(g^2)\right)+C_x\|g\|_\infty^2\nonumber\\
&\leq\frac{C_xn}{B_n^2}\nu\left(f_\beta^2\mathbbm 1_{\{f_\beta\leq\varepsilon B_n\}}\right)
+C_x\varepsilon^2+\frac{C_xn}{B_n^2}.
\label{eq:small-jump-second-moment-bound}
\end{align}
Since $\alpha<2$, we have $n/B_n^2\to0$. Thus, the last term in \eqref{eq:small-jump-second-moment-bound} tends to zero, proving \eqref{ref:conditional-small-jump-second-moment}.
\end{proof}

\subsection{The case \texorpdfstring{$\alpha=1$}{α=1}}

\begin{proof}[Proof of Theorem~\ref{thm:main-stable-laws}~\ref{it:stab-law-ii}]
Write $\mathbb P_{x,n}(\cdot):=\mathbb P_x(\cdot\mid\tau>n)$ and recall that $\upsilon_n(A)
:=
n\nu\left(\frac{f_\beta}{B_n}\in A\right).$
By the definition of $A_n$,
$$\frac{A_n}{B_n}=\int_{(0,1]}y\,\upsilon_n(\d y).$$
Fix $0<\varepsilon<1<R$. We have the decomposition
$$
\begin{aligned}
\frac{1}{B_n}\left(\sum_{i=0}^{n-1}f_\beta\circ X_i-A_n\right)
&=
\int_{(0,\varepsilon]}y\left(N_n-\upsilon_n\right)(\d y)\\
&\phantom{=}+\int_{(\varepsilon,R]}y\,N_n(\d y)-\int_{(\varepsilon,1]}y\,\upsilon_n(\d y)+\int_{(R,\infty]}y\,N_n(\d y).
\end{aligned}
$$
Combining Theorem~\ref{thm:ppp}, the convergence $\upsilon_n\to\upsilon_1$ on sets bounded away from zero, and the fact that $N(\{\e,R\})=0$, we obtain that
\begin{align}
\int_{(\varepsilon,R]}y\,N_n(\d y)-\int_{(\varepsilon,1]}y\,\upsilon_n(\d y)&\xrightarrow[n\to\infty]{d}\int_{(\varepsilon,R]}y\,N(\d y)-\int_{(\varepsilon,1]}y\,\upsilon_1(\d y).
\label{ref:middle-jumps-alpha-one}
\end{align}
We next control the small jumps. 
Lemma~\ref{lem:conditional-second-moment} gives 
$$
\mathbb E_{x,n}\left[\left(\int_{(0,\varepsilon]}y\left(N_n-\upsilon_n\right)(\d y)\right)^2\right]\leq \frac{C_x}{B_n^2}n\nu({f_\beta}^2 \mathbbm 1_{\{f_\beta\leq\varepsilon B_n\}})+C_x\varepsilon^2+\smallO(1).
$$
From \eqref{eq:tail}, $\nu(f_\beta>t)\sim C(x_0)t^{-1}.$
Using
$$
\nu\left(f_\beta^2\mathbbm 1_{\{f_\beta\leq t\}}\right)=2\int_0^t u\nu(f_\beta>u)\,\d u-t^2\nu(f_\beta>t),
$$
we therefore obtain $\nu\left(
f_\beta^2\mathbbm 1_{\{f_\beta\leq t\}}\right)\sim t^2\nu(f_\beta>t).$ Taking $t=\varepsilon B_n$ gives
$$
\begin{aligned}
\frac{n}{B_n^2}\nu\left(f_\beta^2\mathbbm 1_{\{f_\beta\leq\varepsilon B_n\}}\right)\sim\varepsilon^2 n\nu(f_\beta>\varepsilon B_n)\to\varepsilon.
\end{aligned}
$$
Therefore, Chebyshev's inequality yields, for every $\delta>0$,
$$
\begin{aligned}
\limsup_{n\to\infty}\mathbb P_{x,n}\left(\left|\int_{(0,\varepsilon]}y\left(N_n-\upsilon_n\right)(\d y)\right|>\delta\right)\leq \frac{C_x}{\delta^2}\left(\varepsilon+\varepsilon^2\right).
\end{aligned}
$$
Letting $\varepsilon\to0$, we obtain
\begin{align}
\lim_{\varepsilon\to0}\limsup_{n\to\infty}\mathbb P_{x,n}\left(\left|\int_{(0,\varepsilon]}y\left(N_n-\upsilon_n\right)(\d y)\right|>\delta\right)
&=0.
\label{ref:small-jumps-alpha-one}
\end{align}
For the limiting Poisson random measure,
$$
\mathbb E\left[\left|\int_{(0,\varepsilon]}y\left(N-\upsilon_1\right)(\d y)\right|^2\right]=\int_{(0,\varepsilon]}y^2\,\upsilon_1(\d y)=\varepsilon.
$$
Hence, the compensated integrals over $(\varepsilon,1]$ converge in $L^2$ as $\varepsilon\to0$. Adding the almost surely finite integral over $(1,\infty]$ gives convergence in probability of the expressions defining $Z_1$.

For the large jumps, Theorem~\ref{thm:ppp} applied to $(R,\infty]$ gives
$$
N_n((R,\infty])\xrightarrow[d]{n\to\infty}N((R,\infty]),\ \text{where }N((R,\infty])\sim\operatorname{Poi}(R^{-1}).
$$
Since
$$
\left\{\int_{(R,\infty]}y\,N_n(\d y)>0\right\}=\{N_n((R,\infty])\geq1\},$$
we obtain
\begin{align}
\lim_{R\to\infty}\limsup_{n\to\infty}\mathbb P_{x,n}\left(\int_{(R,\infty]}y\,N_n(\d y)>0\right)
&=0.
\label{ref:large-jumps-alpha-one}
\end{align}
Combining \eqref{ref:middle-jumps-alpha-one}, \eqref{ref:small-jumps-alpha-one} and \eqref{ref:large-jumps-alpha-one}, and then letting $\varepsilon\to0$ and $R\to\infty$, gives
$$
\frac{1}{B_n}\left(\sum_{i=0}^{n-1}f_\beta\circ X_i-A_n\right)\xrightarrow[d]{n\to\infty}Z_1.
$$
Finally, the characteristic functional of the compensated Poisson random measure gives
$$
\mathbb E\left[e^{itZ_1}\right]
=
\exp\left(\int_0^\infty\left(e^{ity}-1-ity\mathbbm 1_{(0,1]}(y)\right)y^{-2}\,\d y\right).
$$
The corresponding Lévy measure is supported on $(0,\infty)$ and has density $y^{-2}$. Therefore, by~\cite[Theorem~14.10]{Sato1999}, $Z_1$ has a $1$-stable distribution totally skewed to the right.
\end{proof}

\subsection{The case \texorpdfstring{$1<\alpha<2$}{1<α<2}}

\begin{proof}[Proof of Theorem~\ref{thm:main-stable-laws}~\ref{it:stab-law-iii}]
Write $\mathbb P_{x,n}(\cdot):=\mathbb P_x(\cdot\mid\tau>n)$ and define the deterministic measure
$$
\upsilon_n(A):=n\nu\left(\frac{f_\beta}{B_n}\in A\right).
$$
Since $f_\beta\in L^1(M,\nu)$, we have that
$$
\frac{n\nu(f_\beta)}{B_n}=\int_{(0,\infty]}y\,\upsilon_n(\d y).
$$
Hence
$$
\frac{1}{B_n}\left(\sum_{i=0}^{n-1}f_\beta\circ X_i-n\nu(f_\beta)\right)=\int_{(0,\infty]}y\left(N_n-\upsilon_n
\right)(\d y).
$$
Fix $0<\varepsilon<1<R$. From Theorem~\ref{thm:ppp}, the convergence $\upsilon_n\to\upsilon_\alpha$ on sets bounded away from zero and since $N(\{\e,R\})=0$  we obtain that
\begin{align}
\int_{(\varepsilon,R]}y\left(N_n-\upsilon_n\right)(\d y)
&\xrightarrow[d]{n\to\infty}\int_{(\varepsilon,R]}y\left(N-\upsilon_\alpha\right)(\d y).
\label{ref:middle-jumps-alpha-greater-one}
\end{align}
We next control the small jumps. 
From Lemma \ref{lem:conditional-second-moment} we have that
\begin{align*}
\mathbb E_{x,n}\left[\left(\int_{(0,\varepsilon]}y\left(N_n-\upsilon_n\right)(\d y)\right)^2\right]
&\leq
C_xn\nu\left(\frac{1}{B_n^2} f_\beta^2\mathbbm 1_{\{f_\beta\leq\varepsilon B_n\}}\right)+C_x\varepsilon^2+\smallO(1).
\end{align*}
From \eqref{eq:tail}, $\nu(f_\beta>t)\sim C(x_0)t^{-\alpha}$. Since $\alpha<2$, integration of the tail asymptotic gives
$$
\nu\left(f_\beta^2\mathbbm 1_{\{f_\beta\leq t\}}\right)\sim\frac{\alpha}{2-\alpha}t^2\nu(f_\beta>t).
$$
Taking $t=\varepsilon B_n$, we obtain
$$
n\nu\left(\frac{1}{B_n^2} f_\beta^2\mathbbm 1_{\{f_\beta\leq\varepsilon B_n\}}\right)=\frac{n}{B_n^2}\nu\left(f_\beta^2\mathbbm 1_{\{f_\beta\leq\varepsilon B_n\}}\right)\xrightarrow[]{n\to\infty}\frac{\alpha}{2-\alpha}\varepsilon^{2-\alpha}.
$$
Therefore, Chebyshev's inequality gives, for every $\delta>0$,
\begin{align}
\lim_{\varepsilon\to0}\limsup_{n\to\infty}\mathbb P_{x,n}\left(\left|\int_{(0,\varepsilon]}y\left(N_n-\upsilon_n\right)(\d y)\right|>\delta\right)
&=0.
\label{ref:small-jumps-alpha-greater-one}
\end{align}
For the limiting Poisson random measure,
$$
\mathbb E\left[\left|\int_{(0,\varepsilon]}y
\left(N-\upsilon_\alpha\right)(\d y)\right|^2\right]
=
\int_{(0,\varepsilon]}y^2\,\upsilon_\alpha(\d y)
=
\frac{\alpha}{2-\alpha}\varepsilon^{2-\alpha}.
$$
Thus, the compensated integrals over $(\varepsilon,1]$ converge in $L^2$ as $\varepsilon\to0$. Adding the compensated integral over $(1,\infty]$ gives convergence in probability of the expressions defining $Z_\alpha$. It remains to control the large jumps. Theorem~\ref{thm:ppp} gives
$$
N_n((R,\infty])
\xrightarrow[d]{n\to\infty}N((R,\infty]),
$$
where $N((R,\infty])\sim\operatorname{Poi}(R^{-\alpha})$. Hence,
$$
\lim_{R\to\infty}\limsup_{n\to\infty}\mathbb P_{x,n}\left(\int_{(R,\infty]}y\,N_n(\d y)>0\right)=0.
$$
Since $\alpha>1$, integration of \eqref{eq:tail} gives
$$
\nu\left(f_\beta\mathbbm 1_{\{f_\beta>t\}}\right)\sim\frac{\alpha}{\alpha-1}t\nu(f_\beta>t).
$$
Therefore,
$$
\int_{(R,\infty]}y\,\upsilon_n(\d y)=\frac{n}{B_n}
\nu\left(f_\beta\mathbbm 1_{\{f_\beta>RB_n\}}\right)\xrightarrow[]{n\to\infty}\frac{\alpha}{\alpha-1}R^{1-\alpha}.
$$
Since $R^{1-\alpha}\to0$, we conclude that, for every $\delta>0$,
\begin{align}
\lim_{R\to\infty}\limsup_{n\to\infty}\mathbb P_{x,n}\left(
\left|\int_{(R,\infty]}y\left(N_n-\upsilon_n\right)(\d y)\right|>\delta\right)
&=0.
\label{ref:large-jumps-alpha-greater-one}
\end{align}
Combining \eqref{ref:middle-jumps-alpha-greater-one}, \eqref{ref:small-jumps-alpha-greater-one} and \eqref{ref:large-jumps-alpha-greater-one} gives
$$
\frac{1}{B_n}\left(\sum_{i=0}^{n-1}f_\beta\circ X_i-n\nu(f_\beta)\right)\xrightarrow[d]{n\to\infty}Z_\alpha.
$$
Finally, the characteristic functional of the compensated Poisson random measure gives
$$
\mathbb E\left[e^{itZ_\alpha}\right]=\exp\left(\int_0^\infty
\left(
e^{ity}-1-ity
\right)
\alpha y^{-\alpha-1}\,\d y
\right).
$$
The corresponding Lévy measure is supported on $(0,\infty)$ and has density $\alpha y^{-\alpha-1}$. Therefore, by~\cite[Theorem~14.10]{Sato1999}, $Z_\alpha$ has a $\alpha$-stable distribution totally skewed to the right.
\end{proof}

\section{The Gaussian boundary case, conditional CLT and large deviations}
\label{sec:gaussian-laws}

 The proofs in this section use the Nagaev--Guivarc'h spectral method and perturbation theory for a simple isolated eigenvalue (see~\cite{Kato1995,HennionHerve2001}).  Similar spectral central limit theorems were proved for sequential
dynamical systems in~\cite[Theorems~2.6--2.7 and Corollaries~2.8--2.9]{DemersLiverani2025}. Their argument uses twisted transfer operators for time-dependent dynamics and observables. Our proof follows the same general idea, but for a fixed killed transition operator, row-dependent observables $g_n$, and conditioning on survival up to time $n$.

\begin{lemma}
\label{lem:conditional-triangular-clt}
Let $g_n:M\to\mathbb R$ be bounded measurable functions satisfying $\nu(g_n)=0$, and let $a_n\to\infty$. Define
$$
\sigma_n^2:=\nu(g_n^2)+2\sum_{m=1}^\infty\mu\left(g_n \frac{1}{\lambda^m}\mathcal P^m(\eta g_n)\right).
$$
Assume that
\begin{enumerate}[label = (\roman*)]
    \item \label{it:tri-ass-i} $\|g_n\|_\infty/a_n\to0$ as $n\to\infty$;
    \item \label{it:tri-ass-ii} $n \left(\nu(g_n^2)+\mu(g_n^2)\right)/a_n^2 = \mathcal O(1)$; and
    \item \label{it:tri-ass-iii} $n\sigma_n^2/a_n^2 \to \sigma^2$ as $n\to\infty$.
\end{enumerate}
 Then, for every $x\in M\setminus Z$, under the probability measures $\mathbb P_x(\cdot\mid\tau>n)$,
$$
\frac{1}{a_n}\sum_{j=0}^{n-1}g_n\circ X_j\xrightarrow[d]{n\to\infty}\mathcal N(0,\sigma^2).
$$
In the case that $\sigma = 0$,  $\mathcal N(0,0)$ should be understood as  $\delta_0.$
\end{lemma}
\begin{proof}
Observe that if $\nu(g)=\mu(\eta g)=0$, then $\lambda^{-m}\mathcal P^m(\eta g)=R^m(\eta g)$ for every $m\in\mathbb N$. For $t\in\mathbb R$, define $\widehat{\mathcal P}_{n,t}u:=\widehat{\mathcal P}\left(e^{it g_n/a_n}u\right).$
Hence, choosing representatives, for every $f\in L^2(M,\mu)$ and for $\mu$-almost every $x\in M$,
\begin{align}
\left|\widehat{\mathcal P}_{n,t} f (x)  - \widehat{\mathcal P}f(x)\right| =  \left|\widehat{\mathcal P} [ (e^{i t g_n/a_n} -1) f] (x)\right|  \leq \left\|e^{i t \frac{g_n}{a_n}} -1\right\|_\infty  \widehat{\mathcal P}(|f|) (x).\label{eq:ineq}
\end{align}
From \eqref{eq:ineq} and Lemma \ref{lem:l2-compactness} we obtain that on every compact set $K\subset\mathbb R$,
$$
\sup_{t\in K}\left(\left\|\widehat{\mathcal P}_{n,t}-\widehat{\mathcal P}\right\|_{L^\infty\to L^\infty}+\left\|\widehat{\mathcal P}_{n,t}-\widehat{\mathcal P}\right\|_{L^2(\mu)\to L^2(\mu)}\right)\leq C_K\frac{\|g_n\|_\infty}{a_n}\xrightarrow[]{n\to\infty}0.
$$
Recall from Lemma~\ref{lem:l2-compactness} that $\widehat{\mathcal P}$ is compact on both $L^\infty(M,\mu)$ and $L^2(M,\mu)$. We apply analytic perturbation theory to $\widehat{\mathcal P}_{n,t}$ separately on $L^\infty(M,\mu)$ and $L^2(\mu)$. By \eqref{eq:ineq}, for all sufficiently large $n$, the operator $\widehat{\mathcal P}_{n,t}$, acting on either $L^\infty(M,\mu)$ or $L^2(\mu)$, has a unique simple eigenvalue near $1$ and a corresponding rank-one spectral projection, both depending analytically on $t$ (see~\cite[Ch.~VII, \S1, Theorems~1.7--1.8]{Kato1995}). Since $L^\infty(M,\mu)$ is continuously and densely embedded in $L^2(\mu)$, the action of $\widehat{\mathcal P}_{n,t}$ on $L^\infty(M,\mu)$ is the same whether the operator is considered on $L^\infty(M,\mu)$ or on $L^2(\mu)$. Moreover, in both spaces, the remainder of the spectrum is separated from $1$. Hence, the result in~\cite[Lemma~A.1]{BaladiTsujii2008} implies that the eigenvalues and eigenspaces obtained in the two spaces coincide. We denote the common eigenvalue by $\Lambda_n(t)$ and choose the corresponding eigenfunction $\eta_{n,t}$ so that
\begin{align}
\widehat{\mathcal P}_{n,t}\eta_{n,t}=\Lambda_n(t)\eta_{n,t},\quad\mu(\eta_{n,t})=1.\label{eq:eigenvalueeq}
\end{align}
We denote the corresponding spectral projection by $\Pi_{n,t}$. The projections obtained in the two spaces agree on $L^\infty(M,\mu)$. Define
$$
\mathcal R_{n,t}:=\widehat{\mathcal P}_{n,t}\left(\mathrm{Id}-\Pi_{n,t}\right).
$$
Then, for every $m\geq1$, $\widehat{\mathcal P}_{n,t}^{m}=\Lambda_n(t)^m\Pi_{n,t}+\mathcal R_{n,t}^{m},$
and there exist $\chi\in(0,1)$ and $C>0$ such that
\begin{align*}
\sup_{t\in K}\left(\left\|\mathcal R_{n,t}^{m}\right\|_{L^\infty\to L^\infty}+\left\|\mathcal R_{n,t}^{m}\right\|_{L^2(\mu)\to L^2(\mu)}\right)\leq C\chi^m
\end{align*}
for every $m\geq1$ and all sufficiently large $n$. Moreover,
$$
\Lambda_n(0)=1,\quad\eta_{n,0}=\eta,\quad\Pi_{n,0}=\Pi,
$$
and
$$
\sup_{t\in K}\left(\left\|\Pi_{n,t}-\Pi\right\|_{L^\infty\to L^\infty}+\left\|\Pi_{n,t}-\Pi\right\|_{L^2(\mu)\to L^2(\mu)}\right)\xrightarrow[]{n\to\infty}0.
$$
 We divide the remainder of the proof into five steps. Since the argument is technical, we keep track of all constants explicitly for clarity.

\begin{step}[1]\label{step:tri-clt-1}
We compute the first two derivatives of $\Lambda_n(t)$ at zero. More precisely, we show that
$$
\Lambda_n'(0)=0\quad \text{and}\quad\eta_{n,0}'=\frac{i}{a_n}\sum_{m=1}^\infty R^m(\eta g_n),
$$
and
$$
\Lambda_n''(0)=-\frac{1}{a_n^2}\left(\nu(g_n^2)+2\sum_{m=1}^\infty\mu\left(g_nR^m(\eta g_n)\right)\right).
$$
\end{step}
\begin{stepproof}[Proof of Step~\ref{step:tri-clt-1}]
The first two derivatives of the perturbed operator are
\begin{align}
\widehat{\mathcal P}_{n,0}'u&=\frac{i}{a_n}\widehat{\mathcal P}(g_nu)\quad \text{and}\quad\widehat{\mathcal P}_{n,0}''u=-\frac{1}{a_n^2}\widehat{\mathcal P}(g_n^2u).\label{eq:pertdif}
\end{align}
Differentiating \eqref{eq:eigenvalueeq} at $t=0$ and applying $\mu$ gives
\begin{align}
\Lambda_n'(0)&=\mu\left(\widehat{\mathcal P}_{n,0}'\eta\right)=\frac{i}{a_n}\mu(\eta g_n)=\frac{i}{a_n}\nu(g_n)=0.\label{eq:lambdaprime0}
\end{align}
Moreover, differentiating the normalisation $\mu(\eta_{n,t})=1$ at $t=0$ gives $\mu(\eta_{n,0}')=0.$
The differentiated eigenvalue equation therefore becomes
\begin{align}
\left(\mathrm{Id}-\widehat{\mathcal P}\right)\eta_{n,0}'&=\frac{i}{a_n}\widehat{\mathcal P}(\eta g_n).\label{eq:eta0n}
\end{align}
Since $\mu(\eta g_n)=\nu(g_n)=0$ and $\mu(\eta_{n,0}')=0$, both $\eta g_n$ and $\eta_{n,0}'$ lie in $\ker\mu=\ker\Pi$. On this subspace,
\begin{align}
\left.\widehat{\mathcal P}\right|_{\ker\mu}&=R\quad \text{and}\quad\left[\left.\left(\mathrm{Id}-\widehat{\mathcal P}\right)\right|_{\ker\mu}\right]^{-1}=\left[\left.\left(\mathrm{Id}-R\right)\right|_{\ker\mu}\right]^{-1}=\sum_{m=0}^\infty R^m.\label{eq:powerr}
\end{align}
Let $C_R>0$ and $\kappa\in(0,1)$ be the constants in \eqref{ref:l2-spectral-estimate}. Then
\begin{align}
\left\|\sum_{m=0}^\infty R^mf\right\|_{L^2(\mu)}&\leq\left(1+\frac{C_R\kappa}{1-\kappa}\right)\|f\|_{L^2(\mu)}=C_{\mathrm{inv}}\|f\|_{L^2(\mu)},\label{eq:inverse-bound}
\end{align}
where $C_{\mathrm{inv}}:=1+C_R\kappa/(1-\kappa).$
From \eqref{eq:eta0n} and \eqref{eq:powerr},
\begin{align}
\eta_{n,0}'&=\frac{i}{a_n}\sum_{m=1}^\infty R^m(\eta g_n).\label{eq:etaprime}
\end{align}
Differentiating \eqref{eq:eigenvalueeq} twice and applying $\mu$, using \eqref{eq:pertdif} and \eqref{eq:lambdaprime0}, gives
\begin{align*}
\Lambda_n''(0)&=\mu\left(\widehat{\mathcal P}_{n,0}''\eta\right)+2\mu\left(\widehat{\mathcal P}_{n,0}'\eta_{n,0}'\right)=-\frac{1}{a_n^2}\mu\left(\widehat{\mathcal P}(g_n^2\eta)\right)+\frac{2i}{a_n}\mu\left(\widehat{\mathcal P}(g_n\eta_{n,0}')\right).
\end{align*}
From \eqref{eq:etaprime},
\begin{align*}
\Lambda_n''(0)&=-\frac{1}{a_n^2}\left(\nu(g_n^2)+2\sum_{m=1}^\infty\mu\left(g_nR^m(\eta g_n)\right)\right).
\end{align*}
This completes Step~\ref{step:tri-clt-1}.
\end{stepproof}
\begin{step}[2]\label{step:tri-clt-2}
We prove that, for every compact set $K\subset\mathbb R$,
\begin{align}
\Lambda_n(t)&=1-\frac{t^2}{2a_n^2}\left(\nu(g_n^2)+2\sum_{m=1}^\infty\mu\left(g_nR^m(\eta g_n)\right)\right)+\mathcal E_n(t)\label{eq:eigenvalue-expansion}
\end{align}
for every $t\in K$, where there exists $C_{4,K}>0$ such that
\begin{align}
\sup_{t\in K}|\mathcal E_n(t)|&\leq\frac{C_{4,K}\|g_n\|_\infty}{a_n^3}\left(\nu(g_n^2)+\mu(g_n^2)\right).\label{eq:eigenvalue-remainder}
\end{align}
\end{step}
\begin{stepproof}[Proof of Step~\ref{step:tri-clt-2}]
We first control the third-order remainder of $\Lambda_n(t)$. Fix a compact set $K\subset\mathbb R$ and set
$$
T_K:=\sup_{t\in K}|t|,\quad \text{and}\quad C_{\mathrm{op}}:=\left\|\widehat{\mathcal P}\right\|_{L^2(\mu)\to L^2(\mu)}.
$$
The convergence of peripheral eigenpair gives a constant $C_{\mathrm{sp},K}>0$ such that
\begin{align*}
\sup_{t\in K}\left(\|\eta_{n,t}\|_{L^2(\mu)}+\|\Pi_{n,t}\|_{L^2(\mu)\to L^2(\mu)}\right)&\leq C_{\mathrm{sp},K}
\end{align*}
for all sufficiently large $n$. These constants remain fixed throughout Step~\ref{step:tri-clt-2}.
For every $t\in K$, Taylor's expansion gives
\begin{align}
\left|e^{it g_n/a_n}-1\right|&\leq\frac{T_K|g_n|}{a_n},\label{eq:exponential-estimates1}\\
\left|e^{it g_n/a_n}-1-\frac{it g_n}{a_n}\right|&\leq\frac{T_K^2g_n^2}{2a_n^2},\label{eq:exponential-estimates2}\\
\left|e^{it g_n/a_n}-1-\frac{it g_n}{a_n}+\frac{t^2g_n^2}{2a_n^2}\right|&\leq\frac{T_K^3\|g_n\|_\infty g_n^2}{6a_n^3}.\label{eq:exponential-estimates3}
\end{align}
Since $\mu(\eta_{n,t})=\mu(\eta)=1$, we have $\eta_{n,t}-\eta\in\ker\mu.$ Moreover,
\begin{align}
\widehat{\mathcal P}\eta_{n,t}=\eta+R\left(\eta_{n,t}-\eta\right).\label{eq:PEtaPrime}
\end{align}
Substituting \eqref{eq:PEtaPrime} into \eqref{eq:eigenvalueeq} and rearranging gives
\begin{align}
\left(\mathrm{Id}-R\right)\left(\eta_{n,t}-\eta\right)&=\widehat{\mathcal P}\left(\left(e^{it g_n/a_n}-1\right)\eta_{n,t}\right)-\left(\Lambda_n(t)-1\right)\eta_{n,t}.\label{eq:eigenfunction-difference}
\end{align}
Applying $\mu$ to \eqref{eq:eigenfunction-difference} and using that $(\mathrm{Id} - R)(\eta_{n,t} -\eta) =(\mathrm{Id} - \widehat{\mathcal P})(\eta_{n,t} -\eta) $ and \eqref{eq:PEtaPrime} we obtain
\begin{align}
\Lambda_n(t)-1&=\mu\left(\left(e^{it g_n/a_n}-1\right)\eta\right)+\mu\left(\left(e^{it g_n/a_n}-1\right)\left(\eta_{n,t}-\eta\right)\right).\label{eq:eigenvalue-difference}
\end{align}
From $\nu(g_n)=0$, \eqref{eq:exponential-estimates1}, \eqref{eq:exponential-estimates2} and Cauchy--Schwarz in \eqref{eq:eigenvalue-difference}, we obtain
\begin{align}
\left|\Lambda_n(t)-1\right|&\leq\frac{T_K^2}{2a_n^2}\nu(g_n^2)+\frac{T_K}{a_n}\mu(g_n^2)^{1/2}\left\|\eta_{n,t}-\eta\right\|_{L^2(\mu)}.\label{eq:first-eigenvalue-bound}
\end{align}
The right-hand side of \eqref{eq:eigenfunction-difference} lies in $\ker\mu$. Applying the inverse in \eqref{eq:inverse-bound}, splitting $\eta_{n,t}=\eta+(\eta_{n,t}-\eta),$
and using
$$
\|\eta g_n\|_{L^2(\mu)}\leq\|\eta\|_\infty^{1/2}\nu(g_n^2)^{1/2},
$$
we obtain
\begin{align}
\left\|\eta_{n,t}-\eta\right\|_{L^2(\mu)}\leq&\frac{C_{\mathrm{inv}}C_{\mathrm{op}}T_K\|\eta\|_\infty^{1/2}}{a_n}\nu(g_n^2)^{1/2}+\frac{C_{\mathrm{inv}}C_{\mathrm{op}}T_K\|g_n\|_\infty}{a_n}\left\|\eta_{n,t}-\eta\right\|_{L^2(\mu)}\nonumber\\
&+C_{\mathrm{inv}}C_{\mathrm{sp},K}\left|\Lambda_n(t)-1\right|.\label{eq:first-eigenfunction-bound}
\end{align}
Choose $N_{1,K}$ such that, for every $n\geq N_{1,K}$,
\begin{align}
\frac{C_{\mathrm{inv}}T_K\left(C_{\mathrm{op}}+C_{\mathrm{sp},K}\right)\|g_n\|_\infty}{a_n}&\leq\frac12\quad \text{and}\quad\frac{\|g_n\|_\infty}{a_n}\leq1.\label{eq:absorption-choice}
\end{align}
Substituting \eqref{eq:first-eigenvalue-bound} into \eqref{eq:first-eigenfunction-bound}, using \eqref{eq:absorption-choice} and moving the term containing $\|\eta_{n,t}-\eta\|_{L^2(\mu)}$ to the left-hand side, we obtain
\begin{align}
\sup_{t\in K}\left\|\eta_{n,t}-\eta\right\|_{L^2(\mu)}&\leq\frac{C_{1,K}}{a_n}\left(\nu(g_n^2)+\mu(g_n^2)\right)^{1/2},\label{eq:first-eigenfunction-final}
\end{align}
where
$$
C_{1,K}:=2C_{\mathrm{inv}}\left(C_{\mathrm{op}}T_K\|\eta\|_\infty^{1/2}+\frac{C_{\mathrm{sp},K}T_K^2}{2}\right).
$$
Inserting \eqref{eq:first-eigenfunction-final} into \eqref{eq:first-eigenvalue-bound} gives
\begin{align}
\sup_{t\in K}\left|\Lambda_n(t)-1\right|&\leq\frac{C_{2,K}}{a_n^2}\left(\nu(g_n^2)+\mu(g_n^2)\right),\label{eq:first-eigenvalue-final}
\end{align}
where
$$
C_{2,K}:=\frac{T_K^2}{2}+T_KC_{1,K}.
$$
We next estimate the second-order remainder of the eigenfunction. Equation~\eqref{eq:etaprime} gives
$$
\left(\mathrm{Id}-R\right)t\eta_{n,0}'=\frac{it}{a_n}\widehat{\mathcal P}(\eta g_n).
$$
Subtracting this identity from \eqref{eq:eigenfunction-difference} gives
\begin{align*}
\left(\mathrm{Id}-R\right)\left(\eta_{n,t}-\eta-t\eta_{n,0}'\right)&=\widehat{\mathcal P}\left(\left(e^{it g_n/a_n}-1-\frac{it g_n}{a_n}\right)\eta\right)+\widehat{\mathcal P}\left(\left(e^{it g_n/a_n}-1\right)\left(\eta_{n,t}-\eta\right)\right)\nonumber\\
&\quad-\left(\Lambda_n(t)-1\right)\eta_{n,t}.
\end{align*}
Applying \eqref{eq:inverse-bound}, \eqref{eq:exponential-estimates1}, \eqref{eq:exponential-estimates2}, \eqref{eq:first-eigenfunction-final} and \eqref{eq:first-eigenvalue-final} gives
\begin{align}
\sup_{t\in K}\left\|\eta_{n,t}-\eta-t\eta_{n,0}'\right\|_{L^2(\mu)}&\leq\frac{C_{3,K}\|g_n\|_\infty}{a_n^2}\left(\nu(g_n^2)+\mu(g_n^2)\right)^{1/2},\label{eq:second-order-eigenfunction-bound}
\end{align}
where
$$
C_{3,K}:=C_{\mathrm{inv}}\left(\frac{C_{\mathrm{op}}T_K^2\|\eta\|_\infty^{1/2}}{2}+C_{\mathrm{op}}T_KC_{1,K}+\sqrt2C_{\mathrm{sp},K}C_{2,K}\right).
$$
Here we used
$$
\nu(g_n^2)+\mu(g_n^2)\leq\sqrt2\|g_n\|_\infty\left(\nu(g_n^2)+\mu(g_n^2)\right)^{1/2}.
$$
It remains to expand the eigenvalue $\Lambda_n(t)$. From \eqref{eq:eigenvalue-difference}, \eqref{eq:exponential-estimates3} and $\nu(g_n)=0$,
\begin{align}
\mu\left(\left(e^{it g_n/a_n}-1\right)\eta\right)&=-\frac{t^2}{2a_n^2}\nu(g_n^2)+\mathcal E_{n,1}(t),\label{eq:first-eigenvalue-expansion-term}
\end{align}
where
\begin{align}
\sup_{t\in K}\left|\mathcal E_{n,1}(t)\right|&\leq\frac{T_K^3\|g_n\|_\infty}{6a_n^3}\nu(g_n^2).\label{eq:first-eigenvalue-expansion-error}
\end{align}
For the second term in \eqref{eq:eigenvalue-difference}, write
\begin{align}
\mu\left(\left(e^{it g_n/a_n}-1\right)\left(\eta_{n,t}-\eta\right)\right)=\frac{it}{a_n}\mu\left(g_nt\eta_{n,0}'\right)+\mathcal E_{n,2}(t)+\mathcal E_{n,3}(t), \label{eq:second-eigenvalue-expansion-term}
\end{align}
where
$$
\mathcal E_{n,2}(t):=\mu\left(\left(e^{it g_n/a_n}-1-\frac{it g_n}{a_n}\right)t\eta_{n,0}'\right)\ \text{and }\mathcal E_{n,3}(t):=\mu\left(\left(e^{it g_n/a_n}-1\right)\left(\eta_{n,t}-\eta-t\eta_{n,0}'\right)\right).
$$
From \eqref{eq:etaprime} and \eqref{ref:l2-spectral-estimate},
\begin{align}
\|\eta_{n,0}'\|_{L^2(\mu)}&\leq\frac{C_0}{a_n}\nu(g_n^2)^{1/2},\label{eq:eigenfunction-derivative-bound}
\end{align}
where $ C_0:=C_R\kappa\|\eta\|_\infty^{1/2} /(1-\kappa).$
Using \eqref{eq:exponential-estimates1}, \eqref{eq:exponential-estimates2}, \eqref{eq:second-order-eigenfunction-bound} and \eqref{eq:eigenfunction-derivative-bound}, we obtain
\begin{align}
\sup_{t\in K}\left(\left|\mathcal E_{n,2}(t)\right|+\left|\mathcal E_{n,3}(t)\right|\right)&\leq\frac{\left(T_K^3C_0/4+T_KC_{3,K}\right)\|g_n\|_\infty}{a_n^3}\left(\nu(g_n^2)+\mu(g_n^2)\right).\label{eq:second-eigenvalue-expansion-errors}
\end{align}
Finally, \eqref{eq:etaprime} gives
\begin{align}
\frac{it}{a_n}\mu\left(g_nt\eta_{n,0}'\right)&=-\frac{t^2}{a_n^2}\sum_{m=1}^\infty\mu\left(g_nR^m(\eta g_n)\right).\label{eq:correlation-expansion-term}
\end{align}
Combining \eqref{eq:first-eigenvalue-expansion-term}, \eqref{eq:second-eigenvalue-expansion-term} and \eqref{eq:correlation-expansion-term} proves \eqref{eq:eigenvalue-expansion}. Equations~\eqref{eq:first-eigenvalue-expansion-error} and \eqref{eq:second-eigenvalue-expansion-errors} give \eqref{eq:eigenvalue-remainder} with
$$
C_{K}:=\frac{T_K^3}{6}+\frac{T_K^3C_0}{4}+T_KC_{3,K}.
$$
This completes Step~\ref{step:tri-clt-2}.
\end{stepproof}

\begin{step}[3] \label{step:tri-clt-3}
We show that
$$
\Lambda_n(t)^n\xrightarrow[]{n\to\infty}\exp\left(-\frac{\sigma^2t^2}{2}\right).
$$
\end{step}
\begin{stepproof}[Proof of Step~\ref{step:tri-clt-3}]
By the definition of $\sigma_n^2$, Step~\ref{step:tri-clt-2} gives
$$
\Lambda_n(t)=1-\frac{t^2\sigma_n^2}{2a_n^2}+\mathcal E_n(t).
$$
Assumptions~\ref{it:tri-ass-i} and \ref{it:tri-ass-ii}, and \eqref{eq:eigenvalue-remainder} imply that
$$
n|\mathcal E_n(t)|\leq C_t\frac{\|g_n\|_\infty}{a_n}\frac{n}{a_n^2}\left(\nu(g_n^2)+\mu(g_n^2)\right)\xrightarrow[]{n\to\infty}0.
$$
Since \ref{it:tri-ass-iii} gives $n\sigma_n^2/a_n^2\xrightarrow[]{n\to\infty}\sigma^2,$ we obtain
$$
n\left(\Lambda_n(t)-1\right)\xrightarrow[]{n\to\infty}-\frac{\sigma^2t^2}{2}.
$$
In particular, $\Lambda_n(t)-1=\mathcal O(n^{-1})$, and hence $n\left|\Lambda_n(t)-1\right|^2\xrightarrow[]{n\to\infty}0.$
Using the expansion of the logarithm at $1$, we obtain
$$
n\log\Lambda_n(t)=n\left(\Lambda_n(t)-1\right)+\mathcal O\left(n\left|\Lambda_n(t)-1\right|^2\right)\xrightarrow[]{n\to\infty}-\frac{\sigma^2t^2}{2}.
$$
Exponentiating proves Step~\ref{step:tri-clt-3}. Observe that if $\sigma =0$ then, the above computation shows that $\Lambda_n(t) \xrightarrow[]{n\to\infty} 1.$
\end{stepproof}

\begin{step}[4]\label{step:tri-clt-4}
We show that
$$
\mathbb E_{x,n}\left[\exp\left(\frac{it}{a_n}\sum_{j=1}^ng_n\circ X_j\right)\right]\xrightarrow[]{n\to\infty}\exp\left(-\frac{\sigma^2t^2}{2}\right).
$$
\end{step}
\begin{stepproof}[Proof of Step~\ref{step:tri-clt-4}]
By the Markov property,
$$
\mathbb E_{x,n}\left[\exp\left(\frac{it}{a_n}\sum_{j=1}^ng_n\circ X_j\right)\right]=\frac{\widehat{\mathcal P}_{n,t}^{n}\mathbbm 1_M(x)}{\widehat{\mathcal P}^{n}\mathbbm 1_M(x)}.
$$
Using the spectral decompositions of the numerator and denominator, together with the strong Feller property to evaluate the resulting continuous representatives at $x$, we obtain
$$
\widehat{\mathcal P}_{n,t}^{n}\mathbbm 1_M(x)=\Lambda_n(t)^n\Pi_{n,t}\mathbbm 1_M(x)+\mathcal R_{n,t}^{n}\mathbbm 1_M(x)
$$
and
$$
\widehat{\mathcal P}^{n}\mathbbm 1_M(x)=\eta(x)+R^n\mathbbm 1_M(x).
$$
Moreover,
$$
\Pi_{n,t}\mathbbm 1_M(x)\xrightarrow[]{n\to\infty}\Pi\mathbbm 1_M(x)=\eta(x),
$$
while
$$
\mathcal R_{n,t}^{n}\mathbbm 1_M(x)\xrightarrow[]{n\to\infty}0\quad \text{and}\quad R^n\mathbbm 1_M(x)\xrightarrow[]{n\to\infty}0.
$$
Since $\eta(x)>0$ for $x\in M\setminus Z$ and using Step~\ref{step:tri-clt-3} we conclude Step~\ref{step:tri-clt-4}.
\end{stepproof}

\begin{step}[5]\label{step:tri-clt-5}
We replace the sum from $j=1$ to $j=n$ by the sum from $j=0$ to $j=n-1$ and conclude that
$$
\frac{1}{a_n}\sum_{j=0}^{n-1}g_n\circ X_j\xrightarrow[d]{n\to\infty}\mathcal N(0,\sigma^2).
$$
\end{step}
\begin{stepproof}[Proof of Step~\ref{step:tri-clt-5}]
We have
$$
\frac{1}{a_n}\sum_{j=0}^{n-1}g_n\circ X_j-\frac{1}{a_n}\sum_{j=1}^ng_n\circ X_j=\frac{g_n\circ X_0-g_n\circ X_n}{a_n}.
$$
Consequently,
$$
\left|\frac{1}{a_n}\sum_{j=0}^{n-1}g_n\circ X_j-\frac{1}{a_n}\sum_{j=1}^ng_n\circ X_j\right|\leq\frac{2\|g_n\|_\infty}{a_n}\xrightarrow[]{n\to\infty}0
$$
by assumption \ref{it:tri-ass-i}, i.e.~$\|g_n\|_\infty/a_n\to0$ as $n\to\infty$. Step~\ref{step:tri-clt-4} and Lévy's continuity theorem~\cite[Theorem~26.3 and Corollary~1]{Billingsley95} complete the proof.
\end{stepproof}
\end{proof}

\subsection{The boundary case \texorpdfstring{$\alpha=2$}{α=2}}

\begin{proof}[Proof of Theorem~\ref{thm:main-stable-laws}~\ref{it:stab-law-iv}]
Recall that $B_n:= \sqrt{C(x_0) n \log n}$ and let $C:=C(x_0)$. Since $\nu(f_\beta>t)\sim Ct^{-2},$
we have that $f_\beta\in L^1(M,\nu)$, although $f_\beta\notin L^2(M,\nu)$. Integration of the truncated tail, applying Fubini, gives
\begin{align}
\nu\left(f_\beta^2\mathbbm 1_{\{f_\beta\leq t\}}\right)&=2\int_0^t s\nu(f_\beta>s)\,\d s-t^2\nu(f_\beta>t)\sim2C\log t.\label{ref:truncated-second-moment-alpha-two}
\end{align}
Choose 
$$u_n:=\frac{B_n}{(\log n)^{1/4}}\ \text{and define }g_n:=f_\beta\mathbbm 1_{\{f_\beta\leq u_n\}} -\nu(f_\beta\mathbbm 1_{\{f_\beta\leq u_n\}}).
$$

We divide the proof into 5 steps.

\begin{step}[1]\label{step:bdy-1}
We verify assumption~\ref{it:tri-ass-i} of Lemma~\ref{lem:conditional-triangular-clt} applied to the sequence $(g_n)$ with $a_n:=B_n$.
\end{step}
\begin{stepproof}[Proof of Step~\ref{step:bdy-1}]
Since $f_\beta\in L^1(M,\nu)$, we have that $\nu(f_\beta\mathbbm 1_{\{f_\beta\leq u_n\}})\xrightarrow[]{n\to\infty}\nu(f_\beta).$
Therefore,
$$
\frac{\|g_n\|_\infty}{B_n}\leq\frac{u_n+\nu(f_\beta\mathbbm 1_{\{f_\beta\leq u_n\}})}{B_n}=\frac{1}{(\log n)^{1/4}}+\frac{\nu(f_\beta\mathbbm 1_{\{f_\beta\leq u_n\}})}{B_n}\xrightarrow[]{n\to\infty}0.
$$
This verifies \ref{it:tri-ass-i} of Lemma~\ref{lem:conditional-triangular-clt}.
\end{stepproof}

\begin{step}[2]\label{step:bdy-2}
We verify assumption \ref{it:tri-ass-ii} of Lemma~\ref{lem:conditional-triangular-clt} applied to the sequence $(g_n)$ with $a_n:=B_n$.
\end{step}
\begin{stepproof}[Proof of Step~\ref{step:bdy-2}]
By the definition of $g_n$,
$$
\nu(g_n^2)=\nu\left(f_\beta^2\mathbbm 1_{\{f_\beta\leq u_n\}}\right)-\nu\left(f_\beta\mathbbm 1_{\{f_\beta\leq u_n\}}\right)^2.
$$
Since $f_\beta\in L^1(M,\nu)$, the second term is $\mathcal O(1)$. Equation~\eqref{ref:truncated-second-moment-alpha-two} and $\log u_n\sim\frac12\log n$ therefore give
\begin{align}
\nu(g_n^2)&\sim2C\log u_n\sim C\log n.\label{eq:vngrow}
\end{align}
Since $\eta$ is bounded away from zero in a neighbourhood of $x_0$ and $f_\beta$ is bounded outside this neighbourhood, $\mu(g_n^2)=\mathcal O(\log n).$
Recalling that $B_n^2=Cn\log n$, we obtain
$$
\frac{n}{B_n^2}\left(\nu(g_n^2)+\mu(g_n^2)\right)=\frac{1}{C\log n}\left(\nu(g_n^2)+\mu(g_n^2)\right)=\mathcal O(1).
$$
This verifies \ref{it:tri-ass-ii} of Lemma~\ref{lem:conditional-triangular-clt}.
\end{stepproof}

\begin{step}[3]\label{step:bdy-3}
We verify assumption \ref{it:tri-ass-iii} of Lemma~\ref{lem:conditional-triangular-clt} and conclude that
\begin{equation}
\frac{1}{B_n}\sum_{i=0}^{n-1}g_n\circ X_i\xrightarrow[d]{n\to\infty}\mathcal N(0,1).\label{ref:truncated-clt-alpha-two}
\end{equation}
\end{step}
\begin{stepproof}[Proof of Step~\ref{step:bdy-3}]
Recall that $R:=\widehat{\mathcal P}-\Pi,\ \widehat{\mathcal P}:=\lambda^{-1}\mathcal P,\ \Pi g:=\eta\mu(g),$
so that $\widehat{\mathcal P}^{m}=\Pi+R^m$ for every $m\geq1$. From Lemma \ref{lem:l2-compactness} the operator $R$ is compact and  there exist $\widetilde C>0$ and $\kappa\in(0,1)$ such that
\begin{align}
\|R^mf\|_{L^2(\mu)}&\leq\widetilde C\kappa^m\|f\|_{L^2(\mu)}\ \text{for every }f\in L^2(M,\mu)\text{ and }m\geq1.\label{eq:rmgap}
\end{align}
By the Cauchy--Schwarz inequality and \eqref{eq:rmgap},
$$
\begin{aligned}
\left|\mu\left(g_nR^m(\eta g_n)\right)\right|&\leq\|g_n\|_{L^2(\mu)}\left\|R^m(\eta g_n)\right\|_{L^2(\mu)}\leq\widetilde C\kappa^m\|g_n\|_{L^2(\mu)}\|\eta g_n\|_{L^2(\mu)}\leq\widetilde C\|\eta\|_\infty\kappa^m\mu(g_n^2).
\end{aligned}
$$
Since $\mu(g_n^2)=\mathcal O(\log n)$, this gives $\left|\mu\left(g_nR^m(\eta g_n)\right)\right|=\mathcal O(\kappa^m\log n).$ In particular, for every fixed $n$,
$$
\sum_{m=1}^\infty\left|\mu\left(g_nR^m(\eta g_n)\right)\right|<\infty.
$$
To obtain a bound which is negligible compared with $\log n$, we use the compactness of $R$ on $L^2(\mu)$ (see Lemma~\ref{lem:l2-compactness}). Since $R$ is $L^2(\mu)$ compact, so is $R^m$ for every fixed $m\geq1$.

We claim that
$$
\frac{\eta g_n}{\sqrt{\log n}}\xrightarrow[n\to\infty]{w}0\ \text{in }L^2(M,\mu).
$$
Indeed,
\begin{align}
\left\|\frac{\eta g_n}{\sqrt{\log n}}\right\|_{L^2(\mu)}^2&\leq\frac{\|\eta\|_\infty^2}{\log n}\mu(g_n^2)=\mathcal O(1).\label{eq:O1}
\end{align}
Moreover, since $u_n\to\infty$, $\nu(f_\beta\mathbbm 1_{\{f_\beta\leq u_n\}})\to\nu(f_\beta)$ and $f_\beta$ is finite everywhere,
$$
g_n(x)\xrightarrow[]{n\to\infty}f_\beta(x)-\nu(f_\beta)
$$
for every $x\in M$. Consequently,
\begin{align}
\frac{\eta(x)g_n(x)}{\sqrt{\log n}}&\xrightarrow[]{n\to\infty}0\ \text{for every }x\in M.\label{eq:0000}
\end{align}
From~\cite[Theorem~13.44, p.~207]{HewittStromberg1965}, every bounded sequence in $L^p(M,\mu)$, with $1<p<\infty$, which converges $\mu$-almost everywhere also converges weakly to the same limit. Hence from \eqref{eq:O1} and  \eqref{eq:0000} we obtain that
$$
\frac{\eta g_n}{\sqrt{\log n}}\xrightarrow[w]{n\to\infty}0\ \text{in }L^2(M,\mu).
$$
Since $R^m$ is compact on $L^2(\mu)$, weak convergence of its argument implies strong convergence of its image. In other words, for any $m\geq1$,
\begin{align}
R^m\left(\frac{\eta g_n}{\sqrt{\log n}}\right)\xrightarrow[]{n\to\infty}0\ \text{in }L^2(\mu).\label{eq:L2to0}
\end{align}
Since $g_n/\sqrt{\log n}$ is bounded in $L^2(\mu)$, from Cauchy-Schwarz and \eqref{eq:L2to0} we obtain that

\begin{align}
\frac{1}{\log n}\left|\mu\left(g_nR^m(\eta g_n)\right)\right|&=\left|\mu\left(\frac{g_n}{\sqrt{\log n}}R^m\left(\frac{\eta g_n}{\sqrt{\log n}}\right)\right)\right|\xrightarrow[]{n\to\infty}0\label{eq:indoa0}
\end{align}
for every fixed $m\geq1$. Finally, for every $L\geq1$,
$$
\begin{aligned}
\frac{1}{\log n}\left|\sum_{m=1}^\infty\mu\left(g_nR^m(\eta g_n)\right)\right|&\leq\sum_{m=1}^L\frac{\left|\mu\left(g_nR^m(\eta g_n)\right)\right|}{\log n}+\widetilde C\|\eta\|_\infty\frac{\mu(g_n^2)}{\log n}\sum_{m=L+1}^\infty\kappa^m.
\end{aligned}
$$
Taking the upper limit as $n\to\infty$ and using \eqref{eq:indoa0} we obtain that
$$
\limsup_{n\to\infty}\frac{1}{\log n}\left|\sum_{m=1}^\infty\mu\left(g_nR^m(\eta g_n)\right)\right|\leq C\frac{\kappa^{L+1}}{1-\kappa}.
$$
Letting $L\to\infty$, we conclude that
\begin{align}
\frac{1}{\log n}\sum_{m=1}^\infty\mu\left(g_nR^m(\eta g_n)\right)\xrightarrow[]{n\to\infty}0.\label{eq:conv}
\end{align}
It follows from \eqref{eq:vngrow} and \eqref{eq:conv} that 
$$
\sigma_n^2:=\nu(g_n^2)+2\sum_{m=1}^\infty\mu\left(g_nR^m(\eta g_n)\right)\sim C\log n.
$$
Therefore,
$$
\frac{n\sigma_n^2}{B_n^2}=\frac{\sigma_n^2}{C\log n} \xrightarrow[]{n\to\infty}1.
$$
This verifies \ref{it:tri-ass-iii} of Lemma~\ref{lem:conditional-triangular-clt} with $\sigma^2=1$. Lemma~\ref{lem:conditional-triangular-clt} now concludes Step~\ref{step:bdy-3}.
\end{stepproof}
\begin{step}[4]\label{step:bdy-4}
We show that the conditional probability of observing a value of $f_\beta\circ X_j$ larger than the truncation level $u_n$ converges to zero, namely
$$
\mathbb P_{x,n}\left(f_\beta\circ X_j>u_n\text{ for some }0\leq j<n\right)\xrightarrow[]{n\to\infty}0.
$$
\end{step}
\begin{stepproof}[Proof of Step~\ref{step:bdy-4}]
Set $A_n:=\{f_\beta>u_n\}.$
Since
$$
n\nu(A_n)=n\nu(f_\beta>u_n)\sim\frac{Cn}{u_n^2}=\frac{1}{\sqrt{\log n}}\xrightarrow[]{n\to\infty}0,
$$
the sets $A_n$ are rare on the time scale $n$. Moreover, since $\eta$ is bounded away from zero in a neighbourhood of $x_0$ and $A_n$ is contained in this neighbourhood for all sufficiently large $n$,
$$
n\mu(A_n)\leq\eta_*^{-1}n\nu(A_n)\xrightarrow[]{n\to\infty}0.
$$
Observe that for $1\leq j<n$
\begin{align}
\mathbb P_{x,n}[X_j \in A_n] =& \frac{\mathbb P_x[X_j \in A_n\ \text{and} \ X_n\in M]}{\mathbb P_x[X_n \in M]}= \frac{\lambda^{-j}\mathcal P^{j}( \mathbbm 1_{A_n} \lambda^{-(n-j)}\mathcal P^{n-j}(\cdot, M))(x)}{\lambda^{-n}\mathcal P^n (x, M)}\nonumber\\
=&\frac{\eta(x) \nu(A_n) + \mathcal O(\vartheta^{\min\{j, n-j\}-1}\sup_{z\in M}\mathcal P(z,A_n) )}{\eta(x) + \mathcal O(\vartheta^n)}\nonumber\\
=&  \nu(A_n) + \mathcal O_x\left(\vartheta^{\min\{j, n-j\}} \sup_{z\in M}\mathcal P(z,A_n)  \right).\label{eq:o1}
\end{align}
Hence, from \eqref{eq:o1} we obtain that
\begin{align*}
\mathbb P_{x,n}\left(f_\beta\circ X_j>u_n\text{ for some }0\leq j<n\right) \leq& \mathbb P_{x,n}[X_0 \in A_n] + \sum_{j=1}^{n-1}\mathbb P_{x,n}[X_j \in A_n]\\
\leq&  \mathbb P_{x,n}[X_0 \in A_n]  + n \nu(A_n) + \mathcal O_x ( \sup_{z\in M} \mathcal P(z,A_n) )\\
&\xrightarrow[]{n\to\infty} 0.
\end{align*}
Above,  we have used that $\sup_{z\in M}\mathcal P(z,A_n)\xrightarrow[]{n\to\infty}0$ which follows from Lemma \ref{lem:uniform-small-targets} and the term $j=0$ vanishes for all sufficiently large $n$ because $f_\beta(x)<\infty$ for each $x\in M$. This concludes Step~\ref{step:bdy-4}.
\end{stepproof}

\begin{step}[5] \label{step:bdy-5}
We pass from the centred truncated observable $g_n$ back to the original centred observable $f_\beta-\nu(f_\beta)$ and conclude.
\end{step}
\begin{stepproof}[Proof of Step~\ref{step:bdy-5}]
The difference between the original and truncated centred sums is
$$
\begin{aligned}
\frac{1}{B_n}\left(\sum_{i=0}^{n-1}f_\beta\circ X_i-n\nu(f_\beta)\right)-\frac{1}{B_n}\sum_{i=0}^{n-1}g_n\circ X_i=\frac{1}{B_n}\sum_{i=0}^{n-1}f_\beta\mathbbm 1_{\{f_\beta>u_n\}}\circ X_i-\frac{n}{B_n}\nu\left(f_\beta\mathbbm 1_{\{f_\beta>u_n\}}\right).
\end{aligned}
$$
Integration of the tail gives
$$
\nu\left(f_\beta\mathbbm 1_{\{f_\beta>t\}}\right)=t\nu(f_\beta>t)+\int_t^\infty\nu(f_\beta>s)\,\d s\sim\frac{2C}{t}.
$$
Thus,
$$
\frac{n}{B_n}\nu\left(f_\beta\mathbbm 1_{\{f_\beta>u_n\}}\right)\sim\frac{2Cn}{u_nB_n}=\frac{2}{(\log n)^{3/4}}\xrightarrow[]{n\to\infty}0.
$$
By Step~\ref{step:bdy-4},
$$
\frac{1}{B_n}\sum_{i=0}^{n-1}f_\beta\mathbbm 1_{\{f_\beta>u_n\}}\circ X_i\xrightarrow[]{n\to\infty} 0\ \text{in probability under }\mathbb P_{x,n}.
$$
 Combining the above limit, \eqref{ref:truncated-clt-alpha-two} and Slutsky's theorem~\cite[Lemma~2.8]{vanDerVaart1998} proves the result.
\end{stepproof}
This concludes the proof of the theorem.
\end{proof}

\subsection{Conditional central limit theorem for bounded observables}

We first apply the spectral method to a fixed bounded observable $g$. The following theorem gives the conditional central limit theorem in this setting and identifies the asymptotic variance $\sigma_g^2$ through the spectral remainder.

\begin{theorem}
\label{thm:conditional-clt-bounded}
Let $g:M\to\mathbb R$ be bounded and measurable, and set $\overline g:=g-\nu(g)$. Then, for every $x\in M\setminus Z$, under the probability measures $\mathbb P_x(\cdot\mid\tau>n)$,
$$
\frac{1}{\sqrt n}\sum_{j=0}^{n-1}\overline g\circ X_j\xrightarrow[d]{n\to\infty}\mathcal N(0,\sigma_g^2),
$$
where
\begin{align}
\sigma_g^2:=\nu(\overline g^2)+2\sum_{m=1}^{\infty}\mu\left(\overline g\, \frac{1}{\lambda^m}\mathcal P^m(\eta\overline g)\right).\label{ref:conditional-asymptotic-variance}
\end{align}
The series in \eqref{ref:conditional-asymptotic-variance} converges absolutely and $\sigma_g^2\geq0$. When $\sigma_g^2=0$, the limiting distribution is $\delta_0$.
\end{theorem}

\begin{proof}[Proof of Theorem~\ref{thm:conditional-clt-bounded}]
We first prove that the series in \eqref{ref:conditional-asymptotic-variance} converges absolutely. Since $\nu(\overline g)=0$,
$\widehat{\mathcal P}^{m}(\eta\overline g)=R^m(\eta\overline g).$ From Lemma \ref{lem:l2-compactness},
$$
\begin{aligned}
\left|\mu\left(\overline g\,R^m(\eta\overline g)\right)\right|&\leq\|\overline g\|_{L^2(\mu)}\left\|R^m(\eta\overline g)\right\|_{L^2(\mu)}\leq C\kappa^m\|\overline g\|_{L^2(\mu)}\|\eta\overline g\|_{L^2(\mu)}.
\end{aligned}
$$
Since $\overline g$ and $\eta$ are bounded, the expression on the right-hand side is summable in $m$. Hence,
$$
\sum_{m=1}^\infty\left|\mu\left(\overline g\,R^m(\eta\overline g)\right)\right|<\infty.
$$
We next prove that $\sigma_g^2\geq0$. Since $\nu$ is stationary for the $Q$-process,
$$
\begin{aligned}
\mathbb E_\nu^{\mathcal Q}\left[\overline g(X_0)\overline g(X_m)\right]&=\int_M\overline g\,\mathcal Q^m\overline g\,\d\nu=\mu\left(\overline g\,\widehat{\mathcal P}^{m}(\eta\overline g)\right)=\mu\left(\overline g\,R^m(\eta\overline g)\right).
\end{aligned}
$$
Consequently,
$$
\begin{aligned}
\frac{1}{n}\mathbb E_\nu^{\mathcal Q}\left[\left(\sum_{j=0}^{n-1}\overline g\circ X_j\right)^2\right]&=\nu(\overline g^2)+2\sum_{m=1}^{n-1}\left(1-\frac{m}{n}\right)\mu\left(\overline g\,R^m(\eta\overline g)\right)\\
&\xrightarrow[]{n\to\infty}\nu(\overline g^2)+2\sum_{m=1}^\infty\mu\left(\overline g\,R^m(\eta\overline g)\right)=\sigma_g^2.
\end{aligned}
$$
The convergence follows from the absolute convergence proved above. Since the expression on the left-hand side is non-negative for every $n$, we obtain that $\sigma_g^2\geq0$.

We may conclude applying Lemma~\ref{lem:conditional-triangular-clt} with $g_n:=\overline g$ and $a_n:=\sqrt n.$
By the definition of $\overline g$, $\nu(g_n)=\nu(\overline g)=0.$ Assumption~\ref{it:tri-ass-i} follows from
$$
\frac{\|g_n\|_\infty}{a_n}=\frac{\|\overline g\|_\infty}{\sqrt n}\xrightarrow[]{n\to\infty}0.
$$
For assumption~\ref{it:tri-ass-ii}, we have
$$
\frac{n}{a_n^2}\left(\nu(g_n^2)+\mu(g_n^2)\right)=\nu(\overline g^2)+\mu(\overline g^2)=\mathcal O(1).
$$
Finally, the quantity $\sigma_n^2$ appearing in Lemma~\ref{lem:conditional-triangular-clt} is independent of $n$ and satisfies
$$
\sigma_n^2=\nu(\overline g^2)+2\sum_{m=1}^\infty\mu\left(\overline g\,R^m(\eta\overline g)\right)=\sigma_g^2.
$$
Thus,
${n\sigma_n^2}/{a_n^2}=\sigma_g^2,$ which verifies assumption~\ref{it:tri-ass-iii} with $\sigma^2=\sigma_g^2$. This finishes the proof of the theorem.
\end{proof}

\subsection{Conditional central limit theorem for \texorpdfstring{$L^2$}{L2} observables}
\label{sec:l2-clt} 
Building on the previous proof of a conditional central limit theorem for bounded observables, we now extend it to observables in $L^2$ and prove Theorem~\ref{thm:conditional-clt-l2}.

\begin{proof}[Proof of Theorem~\ref{thm:conditional-clt-l2}]
Since $\eta$ is bounded and $\nu=\eta\mu$, we have
$$
\nu(g^2)=\mu(\eta g^2)\leq\|\eta\|_\infty\mu(g^2)<\infty.
$$
In particular, $\nu(g)$ is well defined and $g\in L^2(M,\nu)$.
We first verify that the asymptotic variance is well defined. By \eqref{ref:l2-spectral-estimate},
$$
\begin{aligned}
\left|\mu\left(\overline g\,R^m(\eta\overline g)\right)\right|&\leq\|\overline g\|_{L^2(\mu)}\left\|R^m(\eta\overline g)\right\|_{L^2(\mu)}\leq C\kappa^m\|\overline g\|_{L^2(\mu)}\|\eta\overline g\|_{L^2(\mu)}\leq C\kappa^m\|\eta\|_\infty \|\overline g\|_{L^2(\mu)}^2.
\end{aligned}
$$
Thus, the series defining $\sigma_g^2$ converges absolutely.
For $K>0$, define the bounded truncation
$$
g_K(x):=\max\{ -K, \min\{g(x), K\}\}
$$
and set $\overline g_K:=g_K-\nu(g_K).$ Since $g_K\to g$ in $L^2(M,\mu)$ and $\nu=\eta\mu$ with $\eta$ bounded,
$$
g_K\to g\ \text{in }L^2(M,\mu)\text{ and }L^2(M,\nu).
$$
Hence $\overline g_K\to\overline g\ \text{in }L^2(M,\mu)\text{ and }L^2(M,\nu).$

For centred functions $u,v\in L^2(M,\mu)$, define the bilinear form
$$
\mathcal B(u,v):=\nu(uv)+\sum_{m=1}^{\infty}\left[\mu\left(uR^m(\eta v)\right)+\mu\left(vR^m(\eta u)\right)\right].
$$
By \eqref{ref:l2-spectral-estimate} and the boundedness of $\eta$,
\begin{align*}
|\mathcal B(u,v)|&\leq\|\eta\|_\infty\|u\|_{L^2(\mu)}\|v\|_{L^2(\mu)}+2C\|\eta\|_\infty\sum_{m=1}{\infty}\kappa^m\|u\|_{L^2(\mu)}\|v\|_{L^2(\mu)}\leq C\|u\|_{L^2(\mu)}\|v\|_{L^2(\mu)}.
\end{align*}
Since
$$
\sigma_{g_K}^2=\mathcal B(\overline g_K,\overline g_K)\quad \text{and}\quad  \sigma_g^2=\mathcal B(\overline g,\overline g),
$$
we obtain
$$
\begin{aligned}
\left|\sigma_{g_K}^2-\sigma_g^2\right|&\leq C\|\overline g_K-\overline g\|_{L^2(\mu)}\left(\|\overline g_K\|_{L^2(\mu)}+\|\overline g\|_{L^2(\mu)}\right)\to0.
\end{aligned}
$$
Theorem~\ref{thm:conditional-clt-bounded} gives, for every fixed $K>0$,
\begin{align}
\frac{1}{\sqrt n}\sum_{j=0}^{n-1}\overline g_K\circ X_j\xrightarrow[d]{n\to\infty}\mathcal N(0,\sigma_{g_K}^2)\label{ref:clt-bounded-truncation}
\end{align}
under $\mathbb P_x(\cdot\mid\tau>n)$.

It remains to prove that the contribution of $g-g_K$ becomes negligible when $K\to\infty$. We first control observations larger than the central limit theorem scale. Fix $\varepsilon>0$ and define
$$
A_{n,\varepsilon}:=\left\{|g|>\varepsilon\sqrt n\right\}.
$$
Since $g$ is finite-valued, $A_{n,\varepsilon}$ decreases to the empty set. From Lemma~\ref{lem:uniform-small-targets} we have that
$$
\sup_{z\in M}\mathcal P(z,A_{n,\varepsilon})\xrightarrow[]{n\to\infty}0.
$$
Moreover, $n\nu(A_{n,\varepsilon})\leq\frac{1}{\varepsilon^2}\nu\left(g^2\mathbbm 1_{A_{n,\varepsilon}}\right)\xrightarrow[]{n\to\infty}0$ and
$$
n\mu(A_{n,\varepsilon})\leq\frac{1}{\varepsilon^2}\mu\left(g^2\mathbbm1_{A_{n,\varepsilon}}\right)\xrightarrow[]{n\to\infty} 0.
$$
We claim that
\begin{align}
\mathbb P_{x,n}\left(X_j\in A_{n,\varepsilon}\text{ for some }0\leq j<n\right)\xrightarrow[]{n\to\infty}0.\label{ref:conditional-large-observation}
\end{align}
Indeed, for $1\leq j\leq n/2$, the Markov property and the spectral decomposition give
$$
\mathbb P_{x,n}\left(X_j\in A_{n,\varepsilon}\right)=\mathcal Q^j(x,A_{n,\varepsilon})+\mathcal O_x(\vartheta^{n-j}).
$$
 Since
$$
\|\widehat{\mathcal P} \left( \eta\mathbbm 1_{A_{n,\varepsilon}} \right)\|_\infty\leq\frac{\|\eta\|_\infty}{\lambda}\sup_{z\in M} \mathcal P(z,A_{n,\varepsilon})\quad \text{and}\quad \mu(\widehat{\mathcal P} \left( \eta\mathbbm 1_{A_{n,\varepsilon}} \right))=\nu(A_{n,\varepsilon}),
$$
we obtain
$$
\mathcal Q^j(x,A_{n,\varepsilon})=\nu(A_{n,\varepsilon})+\mathcal O_x\left(\vartheta^{j-1}\sup_{z\in M} \mathcal P(z,A_{n,\varepsilon})\right).
$$
Consequently,
\begin{align}
\sum_{j=1}^{\lfloor n/2\rfloor}\mathbb P_{x,n}\left(X_j\in A_{n,\varepsilon}\right)\leq n\nu(A_{n,\varepsilon})+C_x\sup_{z\in M} \mathcal P(z,A_{n,\varepsilon})+\smallO(1).\label{eq:n/2}
\end{align}
For $n/2<j<n$, applying the spectral decomposition to the initial gap $j$ gives
\begin{align*}
\mathbb P_{x,n}\left(X_j\in A_{n,\varepsilon}\right)&=\frac{\widehat{\mathcal P}^{j}\left(\mathbbm 1_{A_{n,\varepsilon}}\widehat{\mathcal P}^{n-j}\mathbbm 1_M\right)(x)}{\widehat{\mathcal P}^{n}\mathbbm 1_M(x)}\leq C_x\mu(A_{n,\varepsilon})+C_x\vartheta^j.
\end{align*}
It follows that
\begin{align}
\sum_{j=\lfloor n/2\rfloor+1}^{n-1}\mathbb P_{x,n}\left(X_j\in A_{n,\varepsilon}\right)\leq C_xn\mu(A_{n,\varepsilon})+\smallO(1).\label{eq:n/2n}
\end{align}
The term $j=0$ vanishes for all sufficiently large $n$ because $X_0=x$ and $g(x)$ is finite. Combining \eqref{eq:n/2} and \eqref{eq:n/2n} we obtain that  \eqref{ref:conditional-large-observation} holds.

To simplify notation and improve readability, define the bounded centred function
$$
r_{n,K,\varepsilon}:=(g-g_K)\mathbbm 1_{\{|g|\leq\varepsilon\sqrt n\}}-\nu\left((g-g_K)\mathbbm 1_{\{|g|\leq\varepsilon\sqrt n\}}\right).
$$
We have
$$
\nu(r_{n,K,\varepsilon}^2)\leq\nu((g-g_K)^2)\quad \text{and}\quad \mu(r_{n,K,\varepsilon}^2)\leq 2\mu((g-g_K)^2)+2\nu((g-g_K)^2).
$$
Moreover, $\|r_{n,K,\varepsilon}\|_\infty\leq2\left(\varepsilon\sqrt n+K\right).$ Applying Lemma~\ref{lem:conditional-sum-second-moment} to $r_{n,K,\varepsilon}$ and dividing by $n$ gives
\begin{equation}
\mathbb E_{x,n}\left[\left(\frac{1}{\sqrt n}\sum_{j=0}^{n-1}r_{n,K,\varepsilon}(X_j)\right)^2\right]\leq C_x\left(\mu((g-g_K)^2)+\nu((g-g_K)^2)\right)+C_x\varepsilon^2+\frac{C_xK^2}{n}.\label{eq:L2}
\end{equation}
Applying Chebyshev's inequality to \eqref{eq:L2}, we obtain that for every $\delta>0$,
\begin{align}
\limsup_{n\to\infty}\mathbb P_{x,n}\left(\left|\frac{1}{\sqrt n}\sum_{j=0}^{n-1}r_{n,K,\varepsilon}(X_j)\right|>\delta\right)\leq\frac{C_x}{\delta^2}\left(\mu((g-g_K)^2)+\nu((g-g_K)^2)+\varepsilon^2\right).\label{eq:rnk}
\end{align}
Let $E_{n,\varepsilon}:=\left\{|g(X_j)|>\varepsilon\sqrt n\text{ for some }0\leq j<n\right\}.$
By \eqref{ref:conditional-large-observation},
$\mathbb P_{x,n}(E_{n,\varepsilon})\to0$. Moreover,
\begin{align}
\frac{1}{\sqrt n}\sum_{j=0}^{n-1}\left((g-g_K)(X_j)-\nu(g-g_K)\right)=&\frac{1}{\sqrt n}\sum_{j=0}^{n-1}r_{n,K,\varepsilon}(X_j) \nonumber \\
 &-\sqrt n\,\nu\left((g-g_K)\mathbbm 1_{\{|g|>\varepsilon\sqrt n\}}\right)\label{eq:r-nu}
\end{align}
on $\Omega\setminus E_{n,\e}$. The deterministic last term converges to zero. Indeed,
\begin{align}
\sqrt n\left|\nu\left((g-g_K)\mathbbm 1_{\{|g|>\varepsilon\sqrt n\}}\right)\right|&\leq\sqrt n\,\nu\left(|g|\mathbbm 1_{\{|g|>\varepsilon\sqrt n\}}\right)\nonumber\\
&\leq\frac1\varepsilon\nu\left(g^2\mathbbm 1_{A_{n,\varepsilon}}\right)\xrightarrow[]{n\to\infty}0,\label{eq:deterministic}
\end{align}
since $g\in L^2(M,\nu)$ and $A_{n,\e}$ are nested sets satisfying $\cap_{n\in\mathbb N}A_{n,\e} = \varnothing$.
Recalling $\overline g-\overline g_K=(g-g_K)-\nu(g-g_K)$, equations
\eqref{eq:r-nu}--\eqref{eq:deterministic} imply that, for all
sufficiently large $n$,
\begin{align*}
\mathbb P_{x,n}\left(\left|\frac{1}{\sqrt n}\sum_{j=0}^{n-1}\left(\overline g-\overline g_K\right)(X_j)\right|>\delta\right)\leq\mathbb P_{x,n}(E_{n,\varepsilon})+\mathbb P_{x,n}\left(\left|\frac{1}{\sqrt n}\sum_{j=0}^{n-1}r_{n,K,\varepsilon}(X_j)\right|>\frac{\delta}{2}\right).
\end{align*}
Therefore, by \eqref{ref:conditional-large-observation} and
\eqref{eq:rnk},
$$
\limsup_{n\to\infty}\mathbb P_{x,n}\left(\left|\frac{1}{\sqrt n}\sum_{j=0}^{n-1}\left(\overline g-\overline g_K\right)(X_j)\right|>\delta\right)\leq\frac{4C_x}{\delta^2}\left(\mu((g-g_K)^2)+\nu((g-g_K)^2)+\varepsilon^2\right).
$$
Since $(g-g_K)\to0$ in both $L^2(M,\mu)$ and $L^2(M,\nu)$, first
letting $K\to\infty$ and then $\varepsilon\to0$ gives
\begin{align}
\lim_{K\to\infty}\limsup_{n\to\infty}\mathbb P_{x,n}\left(\left|\frac{1}{\sqrt n}\sum_{j=0}^{n-1}\left(\overline g-\overline g_K\right)(X_j)\right|>\delta\right)=0.\label{ref:l2-truncation-negligible}
\end{align}
For every fixed $K$, \eqref{ref:clt-bounded-truncation} holds. We use the following standard approximation result for weak convergence (see~\cite[Theorem~3.2]{Billingsley1999}): if $Y_{n,K}\xrightarrow[d]{n\to\infty}Y_K$ for every fixed $K$, $Y_K\xrightarrow[d]{K\to\infty}Y$, and for every $\delta>0$,
$$
\lim_{K\to\infty}\limsup_{n\to\infty}\mathbb{P}(|Y_{n,K}-Y_n|>\delta)=0,
$$ 
then $Y_n\xrightarrow[d]{n\to\infty}Y.$ Applied to our setting, combining  $\mathcal{N}(0,\sigma_{g_K}^2)\xrightarrow[d]{K\to\infty}\mathcal{N}(0,\sigma_g^2)$ and \eqref{ref:l2-truncation-negligible} we conclude that
$$
\frac{1}{\sqrt n}\sum_{j=0}^{n-1}\bar g(X_j)\xrightarrow[d]{n\to\infty}\mathcal{N}(0,\sigma_g^2),
$$
which finishes the proof.
\end{proof}

\subsection{Conditional large deviations for bounded observables}
\label{sec:exponential-concentration}

We dedicate this section to the proof of Theorem~\ref{thm:conditional-exponential-concentration}, which uses the spectral method for large deviations based on exponentially twisted transition operators (see~\cite{HennionHerve2001}).

\begin{proof}[Proof of Theorem~\ref{thm:conditional-exponential-concentration}]
Set $\overline h:=h-\nu(h)$. For $t\in\mathbb R$, define the twisted operators
$$
\mathcal P_t f:=\mathcal P(e^{t\overline h}f)\quad \text{and}\quad\widehat{\mathcal P}_t f:=\lambda^{-1}\mathcal P_t f=\widehat{\mathcal P}(e^{t\overline h}f)
$$
for $f\in L^\infty(M,\rho)$. Since $\overline h$ is bounded, the family $t\mapsto\widehat{\mathcal P}_t$ is analytic in operator norm, and
\begin{align}
\left\|\widehat{\mathcal P}_t-\widehat{\mathcal P}\right\|_{L^\infty\to L^\infty}&\leq\left\|\widehat{\mathcal P}\right\|_{L^\infty\to L^\infty}e^{|t|\|\overline h\|_\infty}|t|\|\overline h\|_\infty\xrightarrow[]{t\to0}0.\label{eq:concentration-operator-perturbation}
\end{align}
The spectral stability of the isolated simple eigenvalue $1$ of $\widehat{\mathcal P}$ (see~\cite[Ch.~IV, Theorem~3.16]{Kato1995}), together with analytic perturbation theory, therefore gives $t_0>0$, $\chi\in(0,1)$ and analytic families $\Lambda(t)$ and $\Pi_t$ such that, for every $|t|\leq t_0$ and $n\geq1$,
\begin{align}
\widehat{\mathcal P}_t^n&=\Lambda(t)^n\Pi_t+\mathcal R_t^n\quad \text{and}\quad \left\|\mathcal R_t^n\right\|_{L^\infty\to L^\infty}\leq C\chi^n,\label{eq:concentration-spectral-decomposition}
\end{align}
where $\Lambda(0)=1$ and $\Pi_0=\Pi$. Reducing $t_0$ if necessary, we may assume that $\Lambda(t)>0$ and $\Lambda(t)>\chi$ for every $|t|\leq t_0$.

Choose an analytic family of eigenfunctions $\eta_t$ satisfying $\eta_0=\eta$ and $\mu(\eta_t)=1$, and set $\Psi(t):=\log\Lambda(t)$. Differentiating $\widehat{\mathcal P}_t\eta_t=\Lambda(t)\eta_t$
at $t=0$ and applying $\mu$ gives
$$
\Lambda'(0)=\mu\left(\widehat{\mathcal P}(\overline h\eta)\right)=\mu(\overline h\eta)=\nu(\overline h)=0.
$$
Consequently, $\Psi(0)=\Psi'(0)=0$. Since $\Psi$ is analytic near zero, there exists $K_h>0$ such that
\begin{align}
\Psi(t)&\leq K_h t^2\ \text{for every }|t|\leq t_0.\label{eq:concentration-quadratic-bound}
\end{align}

We first consider the sums from $1$ to $n$. By the Markov property,
\begin{align}
\mathbb E_{x,n}\left[\exp\left(t\sum_{j=1}^n\overline h\circ X_j\right)\right]&=\frac{\widehat{\mathcal P}_t^n\mathbbm 1_M(x)}{\widehat{\mathcal P}^{\,n}\mathbbm 1_M(x)}.\label{eq:concentration-moment-representation}
\end{align}
Using \eqref{eq:concentration-spectral-decomposition} and the decomposition of $\widehat{\mathcal P}^{\,n}$, together with the strong Feller property to evaluate the resulting continuous representatives at $x$, we obtain
$$
\frac{\widehat{\mathcal P}_t^n\mathbbm 1_M(x)}{\widehat{\mathcal P}^{\,n}\mathbbm 1_M(x)}=\frac{\Lambda(t)^n\Pi_t\mathbbm 1_M(x)+\mathcal R_t^n\mathbbm 1_M(x)}{\eta(x)+\mathcal O_x(\vartheta^n)}.
$$
Since $\eta(x)>0$, $\Lambda(t)>\chi$ and the projections $\Pi_t$ are uniformly bounded for $|t|\leq t_0$, equations \eqref{eq:concentration-quadratic-bound} and \eqref{eq:concentration-moment-representation} give a constant $C_x^{(0)}>0$ such that
\begin{align}
\mathbb E_{x,n}\left[\exp\left(t\sum_{j=1}^n\overline h\circ X_j\right)\right]&\leq C_x^{(0)}e^{K_hnt^2}\label{ref:conditional-exponential-moment-bound}
\end{align}
for every $n\geq1$ and every $|t|\leq t_0$.

Fix $\varepsilon>0$ and set
\begin{align}
s(h,\varepsilon)&:=\min\left\{t_0,\frac{\varepsilon}{4K_h}\right\}.\label{eq:concentration-parameter}
\end{align}
By Markov's inequality and \eqref{ref:conditional-exponential-moment-bound},
\begin{align}
\mathbb P_{x,n}\left(\frac1n\sum_{j=1}^n\overline h\circ X_j>\frac{\varepsilon}{2}\right)&=\mathbb P_{x,n}\left(\exp\left(s(h,\varepsilon)\sum_{j=1}^n\overline h\circ X_j\right)>e^{ns(h,\varepsilon)\varepsilon/2}\right)\nonumber\\
&\leq e^{-ns(h,\varepsilon)\varepsilon/2}\mathbb E_{x,n}\left[\exp\left(s(h,\varepsilon)\sum_{j=1}^n\overline h\circ X_j\right)\right]\nonumber\\
&\leq C_x^{(0)}\exp\left(-n\left(\frac{s(h,\varepsilon)\varepsilon}{2}-K_hs(h,\varepsilon)^2\right)\right)\leq C_x^{(0)}e^{-s(h,\varepsilon)\varepsilon n/4},\label{eq:concentration-upper-tail}
\end{align}
where the last inequality follows from $s(h,\varepsilon)\leq\varepsilon/(4K_h)$ in \eqref{eq:concentration-parameter}. Applying the same argument with $t=-s(h,\varepsilon)$ and combining the two tails gives
\begin{align}
\mathbb P_{x,n}\left(\left|\frac1n\sum_{j=1}^n\overline h\circ X_j\right|>\frac{\varepsilon}{2}\right)&\leq 2C_x^{(0)}e^{-s(h,\varepsilon)\varepsilon n/4}.\label{eq:concentration-two-sided}
\end{align}

We now return to the sums from $0$ to $n-1$. Since
\begin{align}
\left|\sum_{j=0}^{n-1}\overline h\circ X_j-\sum_{j=1}^n\overline h\circ X_j\right|&\leq2\|\overline h\|_\infty,\label{eq:concentration-boundary-terms}
\end{align}
for $n\geq4\|\overline h\|_\infty/\varepsilon$ we have
\begin{align}
\left\{\left|\frac1n\sum_{j=0}^{n-1}\overline h\circ X_j\right|>\varepsilon\right\}&\subset\left\{\left|\frac1n\sum_{j=1}^n\overline h\circ X_j\right|>\frac{\varepsilon}{2}\right\}.\label{eq:concentration-event-inclusion}
\end{align}
Set $t_h:=4K_ht_0^2$ and $
\gamma(\varepsilon,h):=\min\left\{t_h/{4},\varepsilon^2/(16K_h)\right\}.$ By \eqref{eq:concentration-parameter},
\begin{align}
\gamma(\varepsilon,h)&\leq\min\left\{\frac{t_0\varepsilon}{4},\frac{\varepsilon^2}{16K_h}\right\}=\frac{s(h,\varepsilon)\varepsilon}{4}.\label{eq:concentration-rate-comparison}
\end{align}
Combining \eqref{eq:concentration-two-sided}, \eqref{eq:concentration-event-inclusion} and \eqref{eq:concentration-rate-comparison}, we obtain, for $n\geq4\|\overline h\|_\infty/\varepsilon$,
\begin{align}
\mathbb P_{x,n}\left(\left|\frac1n\sum_{j=0}^{n-1}\overline h\circ X_j\right|>\varepsilon\right)&\leq 2C_x^{(0)}e^{-\gamma(\varepsilon,h)n}.\label{eq:concentration-large-n}
\end{align}
If $n<4\|\overline h\|_\infty/\varepsilon$, then \eqref{eq:concentration-rate-comparison} gives $\gamma(\varepsilon,h)n<s(h,\varepsilon)\|\overline h\|_\infty\leq t_0\|\overline h\|_\infty.$
Hence,
\begin{align}
\mathbb P_{x,n}\left(\left|\frac1n\sum_{j=0}^{n-1}\overline h\circ X_j\right|>\varepsilon\right)&\leq1\leq e^{t_0\|\overline h\|_\infty}e^{-\gamma(\varepsilon,h)n}.\label{eq:concentration-small-n}
\end{align}
Equations \eqref{eq:concentration-large-n} and \eqref{eq:concentration-small-n} therefore prove \eqref{ref:conditional-exponential-concentration} for every $n\geq1$, with $C_x:=\max\left\{2C_x^{(0)},e^{t_0\|\overline h\|_\infty}\right\},$
which is independent of $\varepsilon$. The proof is completed by setting $a_h := \e^2/(16 K_h)$.
\end{proof}

\section{A conditional Poisson law}
\label{sec:conditional-poisson}

The proof of Theorem~\ref{thm:conditional-poisson} relies on the Nagaev--Guivarc'h method. The argument below is a simpler adaptation of the celebrated result of Keller--Liverani \cite[Theorem~2.1]{KellerLiverani2009} under the conditions of Hypothesis \eqref{hyp:H}.

\begin{proof}[Proof of Theorem~\ref{thm:conditional-poisson}]
Fix $x\in M\setminus(Z\cup\{x_1\})$. Since $\eta$ is continuous, $\eta(x_1)>0$ and $U_n$ shrinks to $x_1$, for all sufficiently large $n$ we have
\begin{align}
U_n\subset M\setminus Z
\quad \text{and}\quad
\mu(U_n)\leq C\nu(U_n).
\label{eq:poisson-measure-comparison}
\end{align}
Moreover, $n\nu(U_n)\to t$ implies $\nu(U_n)\to0$ as $n \to \infty$. Lemma~\ref{lem:uniform-small-targets} therefore gives
\begin{align}
\sup_{z\in M}\mathcal P(z,U_n)\xrightarrow[]{n\to\infty}0.
\label{eq:un0}
\end{align}

Fix $s\in\mathbb R$ and consider the twisted operator $\widehat{\mathcal P}_{n,s}f
:=\widehat{\mathcal P}\left(e^{is\mathbbm 1_{U_n}}f\right)$
on $\mathcal C^0(M;\mathbb C)$. Theorem~\ref{thm:spectralgap}~\ref{it:sg-it1} implies that $\widehat{\mathcal P}_{n,s}$ is strong Feller. Since
$e^{is\mathbbm 1_{U_n}}-1=(e^{is}-1)\mathbbm 1_{U_n}$, we have
\begin{align}
\left\|\widehat{\mathcal P}_{n,s}-\widehat{\mathcal P}\right\|_{\mathcal C^0\to\mathcal C^0}
&\leq\frac{|e^{is}-1|}{\lambda}
\sup_{z\in M}\mathcal P(z,U_n)
\xrightarrow[]{n\to\infty}0.
\label{eq:en2}
\end{align}
By Theorem~\ref{thm:spectralgap}, the eigenvalue $1$ of $\widehat{\mathcal P}$ is simple and isolated, and the remaining spectrum lies strictly inside the unit disk. Spectral perturbation theory~\cite[Ch.~IV, Theorem~3.16]{Kato1995} gives a simple eigenvalue $\lambda_{n,s}\to1$ and an eigenfunction $\eta_{n,s}$ such that
\begin{align}
\widehat{\mathcal P}_{n,s}\eta_{n,s}
&=\lambda_{n,s}\eta_{n,s},
\quad \mu(\eta_{n,s})=1,
\quad \|\eta_{n,s}-\eta\|_\infty\xrightarrow[]{n\to\infty}0.
\label{eq:poisson-perturbed-eigenfunction}
\end{align}
The corresponding spectral projections $\Pi_{n,s}$ of $\widehat{\mathcal P}_{n,s}$ converge in operator norm to $\Pi$, and the powers 
\begin{align}\max\left\{\left\|(\mathcal P_{n,s} - \Pi_{n,s})^k\right\|_{L^\infty(\rho)\to L^\infty(\rho) } , \left\|(\mathcal P - \Pi)^k\right\|_{L^\infty(\rho)\to L^\infty(\rho) }\right\}\xrightarrow[]{k\to\infty}0\label{eq:expfast}
\end{align}
exponentially fast.
Applying $\mu$ to the eigenvalue equation in \eqref{eq:poisson-perturbed-eigenfunction}, and using $\widehat{\mathcal P}^*\mu=\mu$, gives
\begin{align}
\lambda_{n,s}=\mu\left(e^{is\mathbbm 1_{U_n}}\eta_{n,s}\right)=1+(e^{is}-1)\mu\left(\mathbbm 1_{U_n}\eta_{n,s}\right).
\label{eq:poisson-eigenvalue-identity}
\end{align}
Since $\nu=\eta\mu$, equations \eqref{eq:poisson-measure-comparison} and \eqref{eq:poisson-perturbed-eigenfunction} imply
\begin{align}
\left|\mu\left(\mathbbm 1_{U_n}\eta_{n,s}\right)-\nu(U_n)\right|
&\leq\mu(U_n)\|\eta_{n,s}-\eta\|_\infty\leq C\nu(U_n)\|\eta_{n,s}-\eta\|_\infty
=\smallO(\nu(U_n)).
\label{eq:poisson-eigenfunction-error}
\end{align}
Substituting \eqref{eq:poisson-eigenfunction-error} into \eqref{eq:poisson-eigenvalue-identity}, we obtain
\begin{align}
\lambda_{n,s}
&=1+(e^{is}-1)\nu(U_n)+\smallO(\nu(U_n)).
\label{eq:poisson-eigenvalue-expansion}
\end{align}
Since $n\nu(U_n)\to t$, this yields $n(\lambda_{n,s}-1)\to t(e^{is}-1)$ and $n|\lambda_{n,s}-1|^2\to0$. Expanding the logarithm near $1$ therefore gives
\begin{align*}
\log(\lambda_{n,s}^{n-1})  &= (n-1) \log(1 + (\lambda_{n,s} -1)) = (n-1) (\lambda_{n,s} -1) +\mathcal O(  (n-1) |\lambda_{n,s} -1|^2)   \\ 
&= (e^{is} -1) (n-1) \nu(U_n) + \smallO( (n-1)\nu(U_n))+\mathcal O(  (n-1) |\lambda_{n,s} -1|^2).
\end{align*}
Hence,
\begin{align}
\lambda_{n,s}^{n-1}
&\xrightarrow[]{n\to\infty}\exp\left(t(e^{is}-1)\right).
\label{eq:poisson-eigenvalue-power}
\end{align}

It remains to identify the conditional characteristic function. By the Markov property,
\begin{align}
\mathbb E_{x,n}\left[
\exp\left(is\sum_{j=0}^{n-1}\mathbbm 1_{U_n}(X_j)\right)
\right]
&=
\frac{e^{is\mathbbm 1_{U_n}(x)}\widehat{\mathcal P}_{n,s}^{n-1}\widehat{\mathcal P}\mathbbm 1_M(x)}{\widehat{\mathcal P}^{n}\mathbbm 1_M(x)}.
\label{eq:poisson-characteristic-function}
\end{align}
The convergence of the spectral projections $\Pi_{n,s}\to \Pi$ and \eqref{eq:expfast} implies
\begin{align}
\widehat{\mathcal P}_{n,s}^{n-1}\widehat{\mathcal P}\mathbbm 1_M(x)&=\lambda_{n,s}^{n-1}\left(\eta(x)+\smallO(1)\right)+\smallO(1),
\label{eq:poisson-spectral-numerator}
\end{align}
where we used $\Pi\widehat{\mathcal P}\mathbbm 1_M=\eta$. Moreover,
$\widehat{\mathcal P}^{n}\mathbbm 1_M(x)\to\eta(x)>0$.
Since $x\neq x_1$ and $U_n$ shrinks to $x_1$, we also have $\mathbbm 1_{U_n}(x)=0$ for all sufficiently large $n$. Combining \eqref{eq:poisson-eigenvalue-power}--\eqref{eq:poisson-spectral-numerator}, we conclude that
$$
\mathbb E_{x,n}\left[
\exp\left(is\sum_{j=0}^{n-1}\mathbbm 1_{U_n}(X_j)\right)
\right]
\xrightarrow[]{n\to\infty}
\exp\left(t(e^{is}-1)\right).
$$
The limit is the characteristic function of $\operatorname{Poi}(t)$. Since $s\in\mathbb R$ is arbitrary, Lévy's continuity theorem~\cite[Theorem~26.3 and Corollary~1]{Billingsley95} completes the proof.
\end{proof}

\section{Lean formalisation}
\label{sec:lean-formalisation}

Theorems~\ref{thm:ppp}, \ref{thm:main-stable-laws},
\ref{thm:conditional-clt-l2},
\ref{thm:conditional-exponential-concentration},
\ref{thm:conditional-poisson}
and~\ref{thm:refined-conditional-lln}, together with the nine remarks,
have been formalised in Lean~4.35.0 using Mathlib.
The source code is available at
\url{https://github.com/Manzatto/conditional-stable-laws}.

The main declarations are collected in
\texttt{ConditionedStableLaws/MainResults.lean}, in the namespace
\texttt{ConditionedStableLaws} as \texttt{TheoremA}, \texttt{TheoremB}, \texttt{TheoremC}, \texttt{TheoremD}, 
\texttt{TheoremE}, and \texttt{TheoremF}.
The remark declarations are collected in
\texttt{ConditionedStableLaws/Remarks.lean}.
The correspondence between the manuscript and the formal statements
is documented in \texttt{docs/ StatementsAudit.md}.

The proofs of these results contain no \texttt{sorry} placeholders
and depend only on \texttt{propext}, \texttt{Classical.choice}
and \texttt{Quot.sound}. The compiler and dependency versions are
fixed by \texttt{lean- toolchain} and \texttt{lake-manifest.json}.
To compile the development and run the declaration and axiom audits:
\begin{verbatim}
lake exe cache get
python3 scripts/check.py --require-remarks
\end{verbatim}

\section*{Acknowledgements}
MMC thanks Max Auer for useful discussions at during the ``SDG workshop Dynamical Systems -- Theory and Application'' in Kioloa, 10--14 November 2025, and during the conference ``Statistical Properties and Extremes in Dynamical Systems: Theory and Numerics'', at MATRIX, 19--30 January 2026, which inspired the approach used to prove Theorem~\ref{thm:ppp}.
BBC is supported by a Chapman Fellowship at Imperial College London. MMC is supported by the São Paulo Research Foundation (FAPESP, grant no. 2025/26997-9).

\bibliographystyle{abbrvnat}
\bibliography{refs}

\newpage
\appendix
\renewcommand{\theequation}{A.\arabic{equation}}
\setcounter{equation}{0}
\section{Refined conditional laws of large numbers}
\label{ref:appendix}

In this Appendix we prove Theorem \ref{thm:refined-conditional-lln}. We start by establishing a result which allows us to transfer convergence properties of the $Q$-process to the original absorbing chain conditioned upon survival. 
Below, we adapt the argument in the proof of \cite[Theorem~2.10]{Castroetall} to Birkhoff averages $\Gamma_n=\frac{1}{n}\sum_{i=0}^{n-1}h\circ X_i$. We do not require $\rho(Z)=0$; instead, we assume that $h$ vanishes on $Z$.

\begin{theorem}[{Adapted from \cite[Theorem~2.10\textnormal{(i)}]{Castroetall}}]
\label{thm:conditional-transfer}
Let $X_n$ be an absorbing Markov chain satisfying Hypothesis~\ref{hyp:H},
let $x\in M\setminus Z$, and let $h:M\to\mathbb R$ be measurable,
with $h(z)=0$ for every $z\in Z$. Set
$$
\Gamma_n:=\frac1n\sum_{i=0}^{n-1}h\circ X_i.
$$
\begin{enumerate}[label=(\roman*)]
\item \label{it:cond-trans-i}
If $\Gamma_n$ converges to a constant $\Gamma^\star\in\mathbb R$
in $\mathbb Q_x$-probability, then, for every $\varepsilon>0$,
$$
\mathbb P_x\left(
\left.|\Gamma_n-\Gamma^\star|>\varepsilon\,\right|\,\tau>n
\right)\xrightarrow[]{n\to\infty}0.
$$

\item \label{it:cond-trans-ii}
Under the assumption of~\ref{it:cond-trans-i}, if $h$ is bounded, then
$$
\mathbb E_x\left[
\left.|\Gamma_n-\Gamma^\star|\,\right|\,\tau>n
\right]\xrightarrow[]{n\to\infty}0.
$$
\end{enumerate}
\end{theorem}

\begin{proof}
Set $\widetilde M:=M\setminus Z$ and
$\widetilde\tau:=\inf\{j\geq0:X_j\notin\widetilde M\}$.
We extend $h$ by zero at $\partial$ and fix
$\vartheta\in(\lambda_0/\lambda,1)$.
Since $\mathcal P\eta=\lambda\eta$, $\eta|_Z=0$ and $\eta>0$ on $\widetilde M$,
we have $\mathcal P(z,\widetilde M)=0$ for every $z\in Z$.
Thus $X_n$ cannot return to $\widetilde M$ after leaving it.

Define $\widetilde X_n$ by killing $X_n$ at time $\widetilde\tau$.
For every bounded measurable $g:\widetilde M\to\mathbb R$, extended by zero to $M$,
every $x\in\widetilde M$ and every $n\geq1$, we have
\begin{align}
\lambda^{-n}\widetilde{\mathcal P}^{\,n}g(x)
&=\lambda^{-n}\mathcal P^n(g\mathbbm 1_{\widetilde M})(x)
=\eta(x)\mu(g\mathbbm 1_{\widetilde M})
+R^n(g\mathbbm 1_{\widetilde M})(x).
\label{eq:eqspectral}
\end{align}
Since $\mu(\eta)=1$, we have $\mu(\widetilde M)>0$.
Equation~\eqref{eq:eqspectral} identifies the unique quasi-stationary measure
and normalised survival function of $\widetilde X_n$ as
$$
\widetilde\mu=\frac{\mu(\cdot\cap\widetilde M)}{\mu(\widetilde M)}
\quad \text{and}\quad 
\widetilde\eta=\mu(\widetilde M)\eta,
$$
respectively. In particular,
$\widetilde{\mathcal P}\widetilde\eta=\lambda\widetilde\eta$
and $\widetilde\eta>0$ on $\widetilde M$.

Theorem~\ref{thm:Qprocess} implies that $X_n$ and $\widetilde X_n$ have the same
$Q$-process. Moreover, \eqref{eq:eqspectral} verifies the convergence assumptions
used in the proof of \cite[Theorem~2.10\textnormal{(i)}]{Castroetall}.
That argument also applies when $\lambda=1$.
Since $h|_Z=0$, its averages along $\widetilde X_n$ agree with $\Gamma_n$.
Applying that result at time $m-1$ to the $\mathcal F_{m-1}$-measurable variable
$\Gamma_m$ therefore gives, for every $\varepsilon>0$,
\begin{align}
\mathbb P_x\left(
\left.|\Gamma_m-\Gamma^\star|>\varepsilon
\,\right|\,\widetilde\tau>m-1
\right)
&\xrightarrow[]{m\to\infty}0.
\label{eq:transfer-killed-Z}
\end{align}

We prove~\ref{it:cond-trans-i} by decomposing the survival event $\{\tau>n\}$
according to whether the chain enters $Z$ before time $n$ or remains in
$\widetilde M$ through time $n-1$.
On $\{\widetilde\tau=m,\tau>n\}$ with $1\leq m<n$,
we have $X_i\in Z$ for every $m\leq i<n$. Since $h|_Z=0$,
\begin{align}
\Gamma_n
&=\frac1n\sum_{i=0}^{m-1}h\circ X_i
+\underbrace{\frac1n\sum_{i=m}^{n-1}h\circ X_i}_{=\,0}
=\frac mn\Gamma_m.
\label{eq:transfer-stopped-average}
\end{align}
Fix $\varepsilon>0$. Splitting according to the first entrance into $Z$, we obtain
\begin{align}
\mathbb P_x\left(
\left.|\Gamma_n-\Gamma^\star|>\varepsilon\,\right|\,\tau>n
\right)
&=\sum_{m=1}^{n-1}\mathbb P_x\left(
\left.|\Gamma_n-\Gamma^\star|>\varepsilon,\,
\widetilde\tau=m\,\right|\,\tau>n
\right)
\nonumber\\
&\phantom{={}}+\mathbb P_x\left(
\left.|\Gamma_n-\Gamma^\star|>\varepsilon,\,
\widetilde\tau\geq n\,\right|\,\tau>n
\right).
\label{eq:dec}
\end{align}

We first estimate the terms in the sum, which are zero if $Z=\varnothing$.
Suppose that $Z\neq\varnothing$ and fix $1\leq m<n$.
On $\{\widetilde\tau=m,\tau>n\}$, the chain enters $Z$ at time $m$
and must survive another $n-m$ steps.
Using \eqref{eq:transfer-stopped-average} and the
$\mathcal F_{m-1}$-measurability of $\Gamma_m$, the Markov property gives
\begin{align}
&\mathbb P_x\left(
\left.|\Gamma_n-\Gamma^\star|>\varepsilon,\,
\widetilde\tau=m\,\right|\,\tau>n
\right)
\nonumber\\
&\leq\frac{\mathbb P_x(\widetilde\tau>m-1)}{\mathbb P_x(\tau>n)}
\sup_{z\in Z}\mathbb P_z(\tau>n-m)
\mathbb P_x\left(
\left.\left|\frac mn\Gamma_m-\Gamma^\star\right|>\varepsilon
\,\right|\,\widetilde\tau>m-1
\right).
\label{eq:ineq22}
\end{align}
Equation~\eqref{eq:eqspectral} gives
$\mathbb P_x(\widetilde\tau>m-1)\leq C\lambda^{m-1}$,
including $m=1$ after increasing $C$.
Moreover, $\mathbb P_x(\tau>n)\sim\eta(x)\lambda^n$.
Since $\eta|_Z=0$, \eqref{eq:eqspectral} implies
$$
\mathbb P_z(\tau>k)
=\lambda^kR^k\mathbbm 1_M(z)
\leq C\lambda^k\vartheta^k
\ \text{for every }z\in Z\text{ and }k\geq1.
$$
Consequently, enlarging $C_x$ to cover finitely many small $n$, we obtain
\begin{align}
\mathbb P_x\left(
\left.|\Gamma_n-\Gamma^\star|>\varepsilon,\,
\widetilde\tau=m\,\right|\,\tau>n
\right)
&\leq C_x\vartheta^{n-m}
\mathbb P_x\left(
\left.\left|\frac mn\Gamma_m-\Gamma^\star\right|>\varepsilon
\,\right|\,\widetilde\tau>m-1
\right),
\label{eq:firstterm}
\end{align}
where $C_x$ is independent of $m$ and $n$.

For the remaining term, $\{\widetilde\tau\geq n\}=\{\widetilde\tau>n-1\}$ implies
\begin{align}
\mathbb P_x\left(
\left.|\Gamma_n-\Gamma^\star|>\varepsilon,\,
\widetilde\tau\geq n\,\right|\,\tau>n
\right)
&\leq\frac{\mathbb P_x(\widetilde\tau>n-1)}{\mathbb P_x(\tau>n)}
\mathbb P_x\left(
\left.|\Gamma_n-\Gamma^\star|>\varepsilon
\,\right|\,\widetilde\tau>n-1
\right)
\nonumber\\
&\leq C_x\mathbb P_x\left(
\left.|\Gamma_n-\Gamma^\star|>\varepsilon
\,\right|\,\widetilde\tau>n-1
\right).
\label{eq:secondterm}
\end{align}
Indeed, since $\widetilde\tau\leq\tau$,
$$
\frac{\mathbb P_x(\widetilde\tau>n-1)}{\mathbb P_x(\tau>n)}
\leq\frac{\mathbb P_x(\tau>n-1)}{\mathbb P_x(\tau>n)}
\xrightarrow[]{n\to\infty}\frac1\lambda,
$$
so the ratio is bounded independently of $n$.

Substituting \eqref{eq:firstterm} and \eqref{eq:secondterm} into \eqref{eq:dec},
and writing $j=n-m$, yields
\begin{align}
&\mathbb P_x\left(
\left.|\Gamma_n-\Gamma^\star|>\varepsilon\,\right|\,\tau>n
\right)\leq C_x\sum_{j=0}^{n-1}\vartheta^j
\mathbb P_x\left(
\left.\left|\frac{n-j}{n}\Gamma_{n-j}-\Gamma^\star\right|>\varepsilon
\,\right|\,\widetilde\tau>n-j-1
\right).
\label{eq:transfer-deviation-bound}
\end{align}
The term $j=0$ accounts for $\{\widetilde\tau\geq n\}$.
For each fixed $j$ and all sufficiently large $n$,
$$
\left|\frac{n-j}{n}\Gamma_{n-j}-\Gamma^\star\right|
\leq|\Gamma_{n-j}-\Gamma^\star|+\frac jn|\Gamma^\star|.
$$
Thus the probability in the $j$th summand tends to zero by
\eqref{eq:transfer-killed-Z}.
Extending the summands by zero for $j\geq n$, each is bounded by $\vartheta^j$.
Since $\sum_{j\geq0}\vartheta^j<\infty$, dominated convergence proves~\ref{it:cond-trans-i}.

To prove~\ref{it:cond-trans-ii}, suppose that $h$ is bounded and write
$\|h\|_\infty:=\sup_{y\in M}|h(y)|$.
On $\{\tau>n\}$, we have $|\Gamma_n-\Gamma^\star|\leq\|h\|_\infty+|\Gamma^\star|$.
Hence, for every $\delta>0$,
\begin{align}
\mathbb E_x\left[
\left.|\Gamma_n-\Gamma^\star|\,\right|\,\tau>n
\right]
&\leq\delta+(\|h\|_\infty+|\Gamma^\star|)
\mathbb P_x\left(
\left.|\Gamma_n-\Gamma^\star|>\delta\,\right|\,\tau>n
\right).
\label{eq:transfer-holder}
\end{align}
By~\ref{it:cond-trans-i}, the probability on the right tends to zero.
Taking the upper limit as $n\to\infty$ and then letting $\delta\to 0$
proves~\ref{it:cond-trans-ii}.
\end{proof}

We now prove Theorem~\ref{thm:refined-conditional-lln}.
\begin{proof}[Proof of Theorem~\ref{thm:refined-conditional-lln}]
Define
$$
\Gamma_n:=\frac1n\sum_{i=0}^{n-1}h\circ X_i
\quad \text{and}\quad
\Gamma^\star:=\nu(h).
$$
For bounded measurable functions, $\|\cdot\|_\infty$ denotes the pointwise supremum norm. Fix $\vartheta\in(\lambda_0/\lambda,1)$. From the spectral decomposition of Theorem~\ref{thm:spectralgap}~\ref{it:sg-it2}, we have
\begin{align}
\|R^kv\|_\infty&\leq C\vartheta^k\|v\|_\infty
\ \text{for }k\geq1,\ \text{and }\sup_{k\geq0}\|\widehat{\mathcal P}^{\,k}\|_{L^\infty\to L^\infty}<\infty.
\label{eq:lln-operator-bounds}
\end{align}
Moreover, for every $x\in M\setminus Z$,
\begin{align}
\widehat{\mathcal P}^{\,n}\mathbbm 1_M(x)
&=\eta(x)+R^n\mathbbm 1_M(x)
\xrightarrow[]{n\to\infty}\eta(x)>0.
\label{eq:lln-survival}
\end{align}
Since all survival probabilities are positive for such $x$, the sequence in \eqref{eq:lln-survival} is also bounded away from zero. We divide the proof into four steps.
\color{black}{
\begin{step}[1]
\label{step:conditional-lln-1}
We prove~\ref{it:cond-lln-1}.
\end{step}

\begin{stepproof}[Proof of Step~\ref{step:conditional-lln-1}]


Let $h:M\to\mathbb R$ be bounded and measurable, and set $\overline h:=h-\nu(h)$. Since $\mu(\eta\overline h)=\nu(\overline h)=0$, it follows that $\widehat{\mathcal P}^{\,m}(\eta\overline h)=R^m(\eta\overline h)$ for every $m\geq1$.

For $0\leq i<j<n$, the Markov property, Theorem~\ref{thm:Qprocess}~\ref{it:Qproc-ii} and \eqref{eq:lln-operator-bounds} therefore imply
\begin{align}
\left|\mathbb E_x^{\mathbb Q}
\left[\overline h\circ X_i\,\overline h\circ X_j\right]\right|&=\left|\mathcal Q^i\left(\overline h\,\mathcal Q^{j-i}\overline h\right)(x)\right|
\nonumber=\frac1{\eta(x)}
\left|\widehat{\mathcal P}^{\,i}
\left(\overline h\,R^{j-i}(\eta\overline h)\right)(x)\right|
\nonumber\\
&\leq C_x\vartheta^{j-i}\|\overline h\|_\infty^2.
\label{eq:lln-q-correlation-bound}
\end{align}
Since $\Gamma_n-\Gamma^\star=\frac{1}{n}\sum_{i=0}^{n-1}\overline h\circ X_i$, expanding its square and applying \eqref{eq:lln-q-correlation-bound} gives
\begin{align}
\mathbb E_x^{\mathbb Q}\left[|\Gamma_n-\Gamma^\star|^2\right]&=\frac1{n^2}\mathbb E_x^{\mathbb Q}\left[
\left(\sum_{i=0}^{n-1}\overline h\circ X_i\right)^2
\right]\\
&=\frac1{n^2}\sum_{i=0}^{n-1}
\mathbb E_x^{\mathbb Q}\left[(\overline h\circ X_i)^2\right]
+\frac2{n^2}\sum_{0\leq i<j<n}
\mathbb E_x^{\mathbb Q}\left[
\overline h\circ X_i\,\overline h\circ X_j
\right]\nonumber\\
&\leq\frac{\|\overline h\|_\infty^2}{n}
+\frac{2C_x\|\overline h\|_\infty^2}{n^2}
\sum_{0\leq i<j<n}\vartheta^{j-i}\leq\frac{\widetilde C_x\|\overline h\|_\infty^2}{n}
\xrightarrow[]{n\to\infty}0,
\label{eq:lln-q-l2}
\end{align}
where we also used
$\mathbb E_x^{\mathbb Q}[(\overline h\circ X_i)^2]\leq\|\overline h\|_\infty^2$.
By Cauchy--Schwarz, \eqref{eq:lln-q-l2} also gives convergence in $L^1(\mathbb Q_x)$.

Since $X_n$ almost surely never visits $Z$ under $\mathbb Q_x$ and $\nu(Z)=0$, the same convergence holds with $h$ replaced by $h\mathbbm 1_{M\setminus Z}$. This observable vanishes on $Z$ and is bounded by $\|h\|_\infty$. Theorem~\ref{thm:conditional-transfer}~\ref{it:cond-trans-ii} consequently gives
\begin{align}
\mathbb E_x\left[
\left.
\left|\frac1n\sum_{i=0}^{n-1}(h\mathbbm 1_{M\setminus Z})\circ X_i-\nu(h)\right|
\,\right|\,\tau>n
\right]
&\xrightarrow[]{n\to\infty}0.
\label{eq:lln-bounded-active-limit}
\end{align}

It remains to control the contribution from $Z$. Since $\eta|_Z=0$, we have $\widehat{\mathcal P}^{\,j}\mathbbm 1_M=R^j\mathbbm 1_M$ on $Z$ for every $j\geq1$. The Markov property, followed by \eqref{eq:lln-operator-bounds} and \eqref{eq:lln-survival}, gives
\begin{align}
\mathbb E_x\left[
\left.
\left|\frac1n\sum_{i=0}^{n-1}(h\mathbbm 1_Z)\circ X_i\right|
\,\right|\,\tau>n
\right]&\leq\frac1{n\widehat{\mathcal P}^{\,n}\mathbbm 1_M(x)}
\sum_{i=0}^{n-1}\widehat{\mathcal P}^{\,i}
\left(|h|\mathbbm 1_Z\widehat{\mathcal P}^{\,n-i}\mathbbm 1_M\right)(x)
\nonumber\\
&\leq\frac{C_x\|h\|_\infty}{n}
\sum_{i=0}^{n-1}\vartheta^{n-i}
\leq\frac{C_x\|h\|_\infty}{n}\frac{\vartheta}{1-\vartheta}
\xrightarrow[]{n\to\infty}0.
\label{eq:lln-bounded-z-limit}
\end{align}
The triangle inequality, together with \eqref{eq:lln-bounded-active-limit} and \eqref{eq:lln-bounded-z-limit}, proves~\ref{it:cond-lln-1}.
\end{stepproof}

\begin{step}[2]
\label{step:conditional-lln-2}
We prove~\eqref{ref:l1-conditional-probability} for $h\in L^1(M,\nu)$ satisfying $h|_Z=0$.
\end{step}

\begin{stepproof}[Proof of Step~\ref{step:conditional-lln-2}]
By Theorem~\ref{thm:Qprocess}~\ref{it:Qproc-iii}, $\mathbb Q_\nu$ is stationary and ergodic. Birkhoff's ergodic theorem, applied to the stationary shift \cite[Theorem~11.4]{EisnerFarkasHaaseNagel2015}, therefore gives, for $\nu$-almost every $x\in M\setminus Z$,
\begin{align}
\Gamma_n&\xrightarrow[]{n\to\infty}\nu(h)
\ \text{$\mathbb Q_x$-almost surely}.
\label{eq:lln-q-almost-sure}
\end{align}
Applying Theorem~\ref{thm:conditional-transfer}~\ref{it:cond-trans-i} to \eqref{eq:lln-q-almost-sure} proves~\eqref{ref:l1-conditional-probability}.
\end{stepproof}

\begin{step}[3]
\label{step:conditional-lln-3}
For $h\in L^1(M,\mu)$ satisfying $h|_Z=0$, we prove that,
for $\nu$-almost every $x\in M\setminus Z$,
\begin{align}
\mathbb E_x\left[
\left.\Gamma_n\,\right|\,\tau>n
\right]
&\xrightarrow[]{n\to\infty}\nu(h).
\label{ref:l1-quasi-ergodic-limit2}
\end{align}
\end{step}

\begin{stepproof}[Proof of Step~\ref{step:conditional-lln-3}]
Since $\eta$ is bounded, $\nu(|h|)=\mu(\eta|h|)\leq\|\eta\|_\infty\mu(|h|)<\infty.$
It is enough to consider $h\geq0$. The identity $\widehat{\mathcal P}^{*}\mu=\mu$ implies that
\begin{align}
\mu(\widehat{\mathcal P}^{\,i}h)&=\mu(h)<\infty
\ \text{for every }i\geq0.
\label{eq:lln-mu-finiteness}
\end{align}
Thus all $\widehat{\mathcal P}^{\,i}h$ are finite outside a common $\mu$-null set. By \eqref{eq:lln-operator-bounds}, the functions $\widehat{\mathcal P}^{\,k}\mathbbm 1_M$ are uniformly bounded. Consequently, for $\nu$-almost every $x\in M\setminus Z$, all terms in the following decomposition are finite.

Using $\widehat{\mathcal P}^{\,n-i}\mathbbm 1_M=\eta+R^{n-i}\mathbbm 1_M$, the Markov property and Theorem~\ref{thm:Qprocess}~\ref{it:Qproc-ii} give
\begin{align}
\mathbb E_x\left[\left.\Gamma_n\,\right|\,\tau>n\right]
&=\frac1{n\widehat{\mathcal P}^{\,n}\mathbbm 1_M(x)}
\sum_{i=0}^{n-1}\widehat{\mathcal P}^{\,i}
\left(h\widehat{\mathcal P}^{\,n-i}\mathbbm 1_M\right)(x)
\\
&=\frac{\eta(x)}{\widehat{\mathcal P}^{\,n}\mathbbm 1_M(x)}
\frac1n\sum_{i=0}^{n-1}\mathcal Q^ih(x)+\mathcal E_n(x)
\label{ref:l1-quasi-ergodic-decomposition}
\end{align}
where
$$
\mathcal E_n(x):=
\frac1{n\widehat{\mathcal P}^{\,n}\mathbbm 1_M(x)}
\sum_{i=0}^{n-1}\widehat{\mathcal P}^{\,i}
\left(hR^{n-i}\mathbbm 1_M\right)(x).
$$

Define $h/\eta$ to be zero on $Z$. Since $h|_Z=0$, we have $h=\eta(h/\eta)$ on $M$. From Theorem~\ref{thm:Qprocess}~\ref{it:Qproc-ii}, extended to non-negative functions, we obtain
\begin{align}
\widehat{\mathcal P}^{\,i}h(x)
&=\widehat{\mathcal P}^{\,i}\left(\eta\frac h\eta\right)(x)
=\eta(x)\mathcal Q^i(h/\eta)(x).
\label{eq:lln-h-over-eta}
\end{align}
Using positivity, \eqref{eq:lln-operator-bounds} and \eqref{eq:lln-h-over-eta} in the definition of $\mathcal E_n$, we obtain
\begin{align}
|\mathcal E_n(x)|
&\leq\frac{C}{n\widehat{\mathcal P}^{\,n}\mathbbm 1_M(x)}
\sum_{i=0}^{n-1}\vartheta^{n-i}\widehat{\mathcal P}^{\,i}h(x)=\frac{C\eta(x)}{\widehat{\mathcal P}^{\,n}\mathbbm 1_M(x)}
\frac1n\sum_{i=0}^{n-1}\vartheta^{n-i}\mathcal Q^i(h/\eta)(x).
\label{ref:l1-quasi-ergodic-error}
\end{align}
Furthermore, $
\nu(h/\eta)=\int_{M\setminus Z} (h/\eta)\eta\,\d\mu=\mu(h)<\infty.$

By Theorem~\ref{thm:Qprocess}~\ref{it:Qproc-iii}, $\nu$ is stationary for $\mathcal Q$ and $\mathbb Q_\nu$ is ergodic. The pointwise ergodic theorem applied to $\mathcal Q$ (see \cite[Theorem~11.4 and Example~13.24]{EisnerFarkasHaaseNagel2015}) therefore gives, for every $f\in L^1(M,\nu)$,
\begin{align}
\frac1n\sum_{i=0}^{n-1}\mathcal Q^i f(x)
&\xrightarrow[]{n\to\infty}\nu(f)
\ \text{for $\nu$-almost every }x.
\label{eq:lln-q-ergodic}
\end{align}
Applying \eqref{eq:lln-q-ergodic} to $h$ and $h/\eta$, and using $\nu(h/\eta)=\mu(h)$, we obtain
\begin{align}
\frac1n\sum_{i=0}^{n-1}\mathcal Q^ih(x)
&\xrightarrow[]{n\to\infty}\nu(h),
\ \text{and }\frac1n\sum_{i=0}^{n-1}\mathcal Q^i(h/\eta)(x)
\xrightarrow[]{n\to\infty}\mu(h)
\label{ref:q-ergodic-u-over-eta}
\end{align}
for $\nu$-almost every $x$.

Taking differences of consecutive averages in \eqref{ref:q-ergodic-u-over-eta} gives
\begin{align}
\frac{\mathcal Q^i(h/\eta)(x)}{i+1}
&=\frac1{i+1}\sum_{r=0}^{i}\mathcal Q^r(h/\eta)(x)
-\frac{i}{i+1}\frac1i\sum_{r=0}^{i-1}\mathcal Q^r(h/\eta)(x)\xrightarrow[]{i\to\infty}0.
\label{eq:lln-single-iterate}
\end{align}
The weights $(i+1)\vartheta^{n-i}/n$ tend to zero for each fixed $i$, and
\begin{align}
\sum_{i=0}^{n-1}\frac{i+1}{n}\vartheta^{n-i}
&\leq\sum_{j=1}^{\infty}\vartheta^j
=\frac{\vartheta}{1-\vartheta}.
\label{eq:lln-toeplitz-weights}
\end{align}
Thus \eqref{eq:lln-single-iterate}, \eqref{eq:lln-toeplitz-weights} and Toeplitz's lemma~\cite{Toeplitz1911} (see also~\cite[p.~250]{Loeve1977} and~\cite[Theorem~1.1\textnormal{(i)}]{LiHu2017}) yield
\begin{align}
\frac1n\sum_{i=0}^{n-1}\vartheta^{n-i}\mathcal Q^i(h/\eta)(x)
&=\sum_{i=0}^{n-1}\frac{i+1}{n}\vartheta^{n-i}
\frac{\mathcal Q^i(h/\eta)(x)}{i+1}
\xrightarrow[]{n\to\infty}0
\label{eq:zero}
\end{align}
for $\nu$-almost every $x$.

Combining \eqref{ref:l1-quasi-ergodic-error}, \eqref{eq:zero} and \eqref{eq:lln-survival} gives $\mathcal E_n(x)\to0$ for $\nu$-almost every $x$. Substituting this and \eqref{ref:q-ergodic-u-over-eta} into \eqref{ref:l1-quasi-ergodic-decomposition}, and using \eqref{eq:lln-survival} once more, proves~\eqref{ref:l1-quasi-ergodic-limit2}. The result for real-valued $h$ follows by applying the argument to its positive and negative parts.
\end{stepproof}

\begin{step}[4]
\label{step:conditional-lln-4}
We prove~\ref{it:cond-lln-2}.
\end{step}

\begin{stepproof}[Proof of Step~\ref{step:conditional-lln-4}]
Let $h\in L^1(M,\mu)$ satisfy $h|_Z=0$. Since $\eta$ is bounded, $\nu(|h|)=\mu(\eta|h|)\leq\|\eta\|_\infty\mu(|h|)<\infty$.
For each integer $K\geq1$, define $h_K:=h\mathbbm 1_{\{|h|\leq K\}}$ and observe that $|h|\mathbbm 1_{\{|h|>K\}}=|h-h_K|.$
Then $h_K$ is bounded, $|h|\mathbbm 1_{\{|h|>K\}}\in L^1(M,\mu)$, and both functions vanish on $Z$.
Moreover, dominated convergence gives $\nu(|h|\mathbbm 1_{\{|h|>K\}})\to0$ as $K\to\infty$.

The triangle inequality yields
\begin{align}
\mathbb E_x\left[
\left.|\Gamma_n-\Gamma^\star|\,\right|\,\tau>n
\right]
&\leq
\mathbb E_x\left[
\left.
\left|\frac1n\sum_{i=0}^{n-1}h_K\circ X_i-\nu(h_K)\right|
\,\right|\,\tau>n
\right]
\nonumber\\
&\phantom{\leq{}}+
\mathbb E_x\left[
\left.\frac1n \sum_{i=0}^{n-1} \left(|h| \mathbbm 1_{\{|h|>K\}}\right)\circ X_i
\,\right|\,\tau>n
\right]
+\nu(|h|\mathbbm 1_{\{|h|>K\}}).
\label{eq:lln-truncation-bound}
\end{align}
For each fixed $K$, the first term tends to zero by
Step~\ref{step:conditional-lln-1}, while the second tends to $\nu(|h|\mathbbm 1_{\{|h|>K\}})$
by Step~\ref{step:conditional-lln-3}, for $\nu$-almost every $x$.
Since $K$ ranges over the positive integers, these conclusions hold
simultaneously for all $K$ outside a common $\nu$-null set.

Fix $x\in M\setminus Z$ outside this exceptional set. Taking the upper limit
in \eqref{eq:lln-truncation-bound} gives, for every integer $K\geq1$,
$$
\limsup_{n\to\infty}
\mathbb E_x\left[
\left.|\Gamma_n-\Gamma^\star|\,\right|\,\tau>n
\right]
\leq 2\nu(|h|\mathbbm 1_{\{|h|>K\}}).
$$
Letting $K\to\infty$ proves~\ref{it:cond-lln-2}.
\end{stepproof} 
}
\end{proof}

\end{document}